\documentclass[11pt,reqno]{amsart}

\usepackage{amsmath}
\usepackage{amssymb}
\usepackage{amsfonts}
\usepackage{amsthm}
\usepackage[hidelinks]{hyperref}
\usepackage{todonotes}
\usepackage{url}
\usepackage[normalem]{ulem}
\usepackage{orcidlink}

\newcommand{\eqn}[2]{\begin{equation}\label{#1}#2\end{equation}}
\newcommand{\eqnst}[1]{\begin{equation*}#1\end{equation*}}
\newcommand{\eqnspl}[2]{\begin{equation}\begin{split}\label{#1}%
   #2\end{split}\end{equation}}
\newcommand{\eqnsplst}[1]{\begin{equation*}\begin{split}%
   #1\end{split}\end{equation*}}

\def\es{\emptyset}
\def\Zd{\mathbb{Z}^d}

\def\N{\mathbb{N}}
\def\cnctd{\longleftrightarrow}
\def\dcnctd{\Longleftrightarrow}
\def\P{\mathbf{P}}

\def\E{\mathbf{E}}
\def\cE{\mathcal{E}}
\def\tB{\widetilde{\mathcal{B}}}

\def\tC{\widetilde{\mathcal{C}}}
\def\tD{\widetilde{\mathcal{D}}}

\def\tPsi{\widetilde{\Psi}}
\def\Chi{\mathrm{X}}
\def\tR{\widetilde{R}}

\newcommand{\LG}{{\mathcal L_{\bf G}}}

\newcommand{\col}{}
\newcommand{\colAJ}{}

\newcommand{\abs}[1]{\left\vert #1\right\vert}

\newcommand{\Pbiic}{\P_{2-\operatorname{IIC}}}
\newcommand{\Piic}{\P_{\operatorname{IIC}}}

\theoremstyle{plain}
\newtheorem{theorem}{Theorem}[section]
\newtheorem{proposition}[theorem]{Proposition}
\newtheorem{lemma}[theorem]{Lemma}
\newtheorem{corollary}[theorem]{Corollary}
\newtheorem{example}[theorem]{Example}

\theoremstyle{definition}
\newtheorem{definition}[theorem]{Definition}
\newtheorem{remark}[theorem]{Remark}

\numberwithin{equation}{section}

\title{Bi-infinite incipient cluster in high dimensions}
\author{Manuel Cabezas}
\address{Manuel Cabezas\\ Pontificia Universidad Cat\'{o}lica de Chile, Facultad de Matem\'{a}ticas, Campus San Joaqu\'{i}n, Avenida Vicu\~{n}a Mackenna 4860, Santiago, Chile.}
\email{cabezas.mn@gmail.com}
\author{Alexander Fribergh}
\address{Alexander Fribergh\\ Universit\'e de Montr\'eal, DMS\\
Pavillon Andr\'e-Aisenstadt\\     2920, chemin de la Tour Montréal (Qu\'ebec),  H3T 1J4}
\email{alexander.fribergh@umontreal.ca}
\author{Markus Heydenreich}
\address{Markus Heydenreich\\Universität Augsburg, Department of Mathematics, D-86135 Augsburg, Germany}
\email{markus.heydenreich@uni-a.de}
\author{Antal A.~J\'arai}
\address{Antal A.~J\'arai\\ University of Bath, Department of Mathematical Sciences, 
Bath BA2 7AY, United Kingdom}
\email{A.Jarai@bath.ac.uk}
\date\today

\begin{document}

\begin{abstract}{\col
We consider high-dimensional percolation at the critical threshold. 
We condition the origin to be disjointly connected to two points, $x$ and $x'$, and subsequently take the limit as $|x|$, $|x'|$ as well as $|x-x'|$ diverge to infinity. This limiting procedure gives rise to a new percolation measure that locally resembles critical percolation but is concentrated on configurations with two disjoint infinite occupied paths. 
We coin this the bi-infinite incipient percolation cluster. It is mutually singular with respect to incipient infinite clusters that have been constructed in the literature. 
We achieve the construction through a double lace expansion of the cluster. 
}
\end{abstract}

\maketitle

\section{Introduction}
\paragraph{\bf Percolation.} 
Consider bond percolation on the hypercubic lattice $\mathbb Z^d$. Each bond $\{x,y\}$ ($x,y\in\Zd$; $x\neq y$) is \emph{occupied} with probability $p\in[0,1]$ independently of the occupation status of other bonds. The resulting product measure is denoted $\mathbb P_p$. 
We furthermore write $\{x\cnctd y\}$ for the event that two sites $x,y\in\Zd$ are connected by a finite path that solely consists of occupied bonds. 
The corresponding probability is known as \emph{two-point} function, 
\[ \tau_p(x,y):=\P_p(x\cnctd y)=\P_p(0\cnctd y-x),\qquad x,y\in\mathbb Z^d,\]
and we note its spatial shift-invariance.  
Our focus is on \emph{critical} percolation $p=p_c$, which is characterised through 
\[ p_c = \sup \Big\{ p\in[0,1] \;\Big| \sum_{x\in\Zd}\tau_p(0,x)<\infty\Big\}. \]
It is well-known that $0<p_c<1$ in dimension $d\ge2$, e.g.\ \cite{Grimm99}. 
Throughout this article we shall assume that $p=p_c$ (unless explicitly stated otherwise), and further drop it from the notation. 
We denote this version as \emph{nearest-neighbor model} in order to contrast it to another model, which we introduce next. 

\medskip
\paragraph{\bf Spread-out percolation.} 
We consider a second variant known as (uniform) \emph{spread-out percolation}. Here, we enrich the original (nearest-neighbor) lattice by extra bonds, namely every pair of vertices $x,y\in\Zd$ with $\|x-y\|_\infty\le L$ is linked by an edge. As before, we make bonds occupied independently with probablity $p$. 
The extra bonds change the value of $p_c$, and it is known that $\lim_{L\to\infty}p_c(L)=0$ (cf.\ \cite{HaraSlade90a,HaraSlade95}), but we fix a finite (though large) value of $L$. 
The main reason to study this variant is that the upper critical dimension is coming out more clearly for spread-out percolation. More general spread-out models have been considered \cite{HeydeHofstSakai08}, but will not be further studied here. 

\subsection{Assumptions} 
{\col We can in fact consider more general percolation models, where edge $\{x,y\}$ is occupied independently with probability $pD(x,y)$ for some kernel $D\colon\Zd\times\Zd\to[0,1]$ that is invariant under translations (i.e., it depends only on the difference $y-x$) and rotations of 90 degress, and sums to one: $\sum_yD(x,y)=1$ for all $x\in\Zd$. 
This notion allows us to unify the nearest-neighbor case and the uniform spread-out case by setting 
\[
	D(x,y)=
	\begin{cases}
		\frac1{2d} \,I\big[|x-y|=1\big]\qquad&\text{in the nearest-neighbor model;}\\
		\frac1{(2L+1)^d-1}\, I\big[0<\|x-y\|_\infty\le L\big]\qquad&\text{in the uniform spread-out model.}
	\end{cases}
\]
}

We assume throughout that we are in the high-dimensional setting, where percolation is in the mean-field regime. See \cite{HeydeHofst17} or \cite[Chapters 9--11]{Slade06} for a comprehensive treatment of percolation in these regimes. 
To this end, we require that there is a (dimension-dependent) constant $A_d$ such that at criticality (i.e., for $p=p_c$) we have  
\begin{equation}\label{eq:tauAsy}
	\tau(0,x)=A_d\, {|x|}^{2-d}(1+o(1))\quad\text{ as }|x|\to\infty.
\end{equation}
This is indeed verified for nearest-neighbor bond percolation on $\mathbb Z^d$ in dimension $d\ge11$ (Hara \cite{Hara08}, Fitzner \& van der Hofstad \cite{FitznHofst17}) as well as for uniform spread-out percolation in $d>6$ (Hara, van der Hofstad and Slade \cite{HaraHofstSlade03}). 
All these references provide precise estimates on the error term as well as a construction for the constant $A_d$, but we need here only that $A_d\in(0,\infty)$. It is expected to be true in dimension $d>6$ also for the nearest-neighbor model, but this an open problem, cf.~\cite[Open Problem 10.1]{HeydeHofst17}.  
The condition in \eqref{eq:tauAsy} might be rephrased in the language of critical exponents (see \cite[Section 2]{HeydeHofst17}) by saying that the critical exponent $\eta$ exists and attains its mean-field value $\eta=0$. 

{\col
We furthermore assume that there is a uniform constant $C>0$ such that for all $x\in\Zd$, 
\begin{equation}\label{eq:DtauBd}
	\sum_y D(0,y)\tau(y,x)\le C\tau(0,x).
\end{equation}

We finally need a third assumption to assure that the lace expansion is converging. 
To this end, we define 
\eqnspl{eq:Amid}{
 &   A_{1}(u,v,x,y)\
  := \sum_{s,t}\tau(u,s)\tau(s,y)\tau(v,t)\tau(s,t)pD\ast \tau(t,x)=
        \begin{minipage}{0.3\textwidth}
      \includegraphics[width=0.5\linewidth]{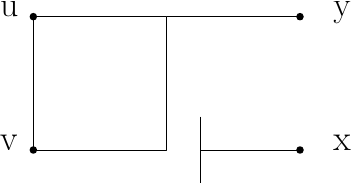}
    \end{minipage}\ \\
 &       A_{2}(u,v,x,y)\
  := \sum_{s,t}\tau(u,y)\tau(u,s)\tau(v,s)\tau(v,t)\tau(s,t)pD\ast \tau(t,x)=
        \begin{minipage}{0.3\textwidth}
      \includegraphics[width=0.5\linewidth]{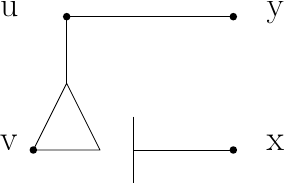}
    \end{minipage}\ \\
 &       A^{(\text{end})}(u,v,x)\
   := \tau(u,x)\tau(u,v)\tau(v,x) = 
        \begin{minipage}{0.3\textwidth}
      \includegraphics[width=0.4\linewidth]{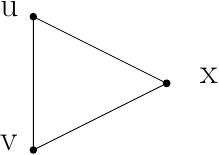}
    \end{minipage}\
}

The final expressions are a graphical representation for the $\tau$'s: each line represents a factor of $\tau$, and a dashed line consists of a convolution $pD\ast\tau(t,x)=\sum_wpD(w-t)\tau(w,x)$. Unlabelled vertices are summed over (here these are $s$ and $t$). 
Such pictorial representations of analytic quantities are convenient tools to organise and bound complex formulae, and will be used more extensively in Section \ref{sec:BoundsFramework}. 
For $N\ge 1$ we define for $\Zd$-valued arguments 
\begin{equation}\label{def_A}
	A^{\scriptscriptstyle(N)}(u_1,v_1,x):= \sum_{\substack{u_2,v_2,\ldots,\\u_N,v_N}}\prod_{i=1}^{N-1}\Big[A_1(u_i,v_i,u_{i+1},v_{i+1})+A_2(u_i,v_i,u_{i+1},v_{i+1})\Big] A^{(\text{end})}(u_N,v_N,x),
\end{equation}
where for $N=1$ we interpret the empty product as 1 so that $A^{(1)}(u,v,x)=A^{(\text{end})}(u,v,x)$. 
Our second assumption now is 
\begin{equation}\label{eq:Aass}
	\sup_{u,v,x\in\Zd}\frac{\sum_{N=1}^\infty A^{\scriptscriptstyle(N)}(u,v,x)}{\tau(u,v)\tau(u,x)\tau(v,x)}<\infty,
\end{equation}
The arguments in \cite[Proof of Prop.~3.1]{HofstJarai04} imply that \eqref{eq:Aass} is true in the uniform spread-out model whenever $d>6$ and $L$ is sufficiently large. Similarly, it is true in the nearest-neighbor model whenever the dimension $d$ is sufficiently large (cf.\ \cite{HeydeHofst17,HofstJarai04}).} 

\subsection{Incipient infinite cluster} {\col While for $p\le p_c$, we know that there are no infinite clusters, there exists almost surely a (unique) infinite cluster of $p>p_c$. The \emph{incipient infinite cluster} is devised as a critical cluster at the verge of appearance of an infinite cluster, and is constructed through a suitable limit of conditional distributions. The first such construction in dimension 2 is due to Kesten \cite{Keste86a}, and his motivation was to study anomalous diffusive behavior of random walks on such clusters. He proved two constructions: (1) by conditioning on the origin to be connected to the boundary of an $n$-box, and then taking the limit $n\to\infty$, and (2) conditioning  the cluster of the origin being infinite for $p>p_c$, and then taking the limit $p\searrow p_c$. He proved that both constructions gave the same limit, for which he coined the term incipient infinite cluster. J\'arai \cite{Jarai03b} extended this constructions in dimension 2 to several other limiting schemes, thereby establishing that the incipient infinite cluster is quite a robust object. 
Damron and Sapozhnikov \cite{DamroSapoz11} derived multiple-arm asymptotics for two-dimensional incipient infinite clusters, and these were sharpened by Yao \cite{Yao18} in the special case of site percolation on the triangular lattice. 

The first construction in high dimension is due to J\'arai and van der Hofstad \cite{HofstJarai04}, who considered two constructions: first by conditioning on $\{0\cnctd x\}$ for some $x\in\Zd$ and then letting $|x|\to\infty$ and secondly by an appropriately weighted subcritical limit. 
Heydenreich, van der Hofstad and Hulshof \cite{HeydeHofstHulsh14b} extended these constructions to long-range percolation, and further verified that Kesten's original construction (1) is also valid in high dimensions. The high-dimensional constructions all rest on the lace expansion, which has been pioneered by Hara and Slade \cite{HaraSlade90a} in the percolation context. In these constructions it was unclear how large the dimension has to be in order for the expansion to converge; Fitzner and van der Hofstad \cite[Theorem 1.5]{FitznHofst17} clarified that $d\ge11$ suffices. 
Very recently, Chatterjee, Chinmoy, Hanson and Sosoe \cite{ChattChinmHansoSosoe25} unified and generalized all the previous approaches in high dimensions by providing a robust new construction that is solely based on the two-point function \eqref{eq:tauAsy}. In their approach they combine techniques from the two-dimensional and high-dimensional constructions. 

We now describe this construction in more detail.} We call an event $F\in\mathfrak F$ a \emph{cylinder event} if it depends on the occupation status of only finitely many edges. We denote the subset of cylinder events by $\mathfrak F_0$. 
Van der Hofstad and J\'arai \cite{HofstJarai04} prove that the limit 
\eqnst{\Piic(F):=\lim_{|x|\to\infty}\P_{p_c}(F\mid\{ o \cnctd x \}), \qquad F\in\mathfrak F_0,}
exists whenever \eqref{eq:tauAsy} is true. 
Since $\mathfrak F_0$ is a $\cap$-stable generator of $\mathfrak F$, we have thus characterised a measure on $\mathfrak F$, which is called the \emph{incipient infinite cluster} measure. 
It is concentrated on configurations that locally resemble critical percolation but where the cluster of the origin is infinitely large ($\Piic$-)almost surely. 
Unlike the critical percolation measure $\P_{p_c}$, the incipient infinite cluster measure is not (shift-)ergodic. 

\subsection{Main result}
In this work we give a variation of the infinite incipient cluster measure, which we do believe to be ergodic w.r.t.\ certain shifts. 
Our main result is the following. 
\begin{theorem}\label{thm:ExistenceBIIC} 
For any independent, translation invariant percolation model on $\Zd$ in dimension $d>6$ which satisfies \eqref{eq:tauAsy}, \eqref{eq:DtauBd} and \eqref{eq:Aass}, the following is true. For any cylinder set $F\in\mathfrak F_0$, the limit 
\[\Pbiic(F):=\lim_{|x|,|x'|,|x-x'|\to\infty}\P_{p_c}(F\mid\{ o \cnctd x \} \circ \{ o \cnctd x' \})\]
exists, and can thus be extended to a measure $\Pbiic$ on the $\sigma$-field $\mathfrak F:=\sigma(\mathfrak F_0)$. 
\end{theorem}
We call $\Pbiic$ the \emph{bi-infinite incipient cluster}, as it resembles critical percolation clusters but has two disjoint connected paths to infinity. It is singular both w.r.t.\ $\P_{p_c}$ as well as w.r.t.\ the ``usual'' incipient infinite cluster measures $\Piic$. 

{\col We recall that for the nearest neighbor model, assumption \eqref{eq:tauAsy} is verified for $d\ge 11$, and there exists $d_*>6$ such that \eqref{eq:Aass} is true for $d>d_*$.  
For the uniform spread-out model, for any $d>6$ there exists $L_*$ such that both \eqref{eq:tauAsy} and \eqref{eq:Aass} are true for $L>L_*$.
Assumption \eqref{eq:DtauBd} is satisfied for both the nearest-neighbor model and the uniform spread-out model, see Lemma \ref{lem:D*tau}. }

In \eqref{IICconstr} we give an explicit construction for the measure $\Pbiic$ in terms of certain lace expansion coefficients. 

Hara and Slade \cite{HaraSlade00a,HaraSlade00b} used a double expansion to study the probability that vertices $o, x, y$ are in the same cluster of large size $n$, and they derived the scaling limit if this probability (under suitable scaling of $x$ and $y$) using the Fourier transform. Our expansion is different and studies the cluster with unrestricted size. 

\subsection{Motivation for this work} 
Our main motivation for introducing the bi-infinite cluster comes from~\cite{BCF1}, which gave sufficient conditions for proving that the simple random walk on a critical graph converges to a Brownian on a conditioned super Brownian motion. {\col The abstract framework of~\cite{BCF1}} was applied to critical branching random walks~\cite{BCF2} and lattice trees~\cite{BCF3} in high dimensions. It is believed that the same result should be applicable to critical percolation in high dimensions. 

One of the aforementioned conditions states that, loosely speaking, the resistance between two distant points should be proportional intrinsic distance. In order to prove this, one needs to understand the \lq\lq backbone\rq\rq~connecting those two points and typical points on that backbone have two long connections emanating from them. This means this backbone should resemble the bi-infinite cluster.

Furthermore, we believe that this bi-infinite cluster has an ergodic property with respect to the shift along pivotal edges of this backbone (here a pivotal edge is an edge whose removal split the bi-infinite cluster into two infinite clusters). This ergodic property should be central to proving the proportionality between the resistance and the intrinsic distance.

Finally, the methods developed in this article provide lots of tools for expanding a 3-arm event. This is also interesting since  large critical clusters contain {\lq\lq branching points\rq\rq} from which three long connections emerge. Those branching points are of particular importance since they are key to understanding $k$-point  functions in critical percolation. This condition is not only important from a lace expansion perspective, but it is central for proving the convergence of the simple random walk on critical graphs (see~\cite{CFHP} for a further discussion on this in the case of lattice trees and branching random walks).

For the model of self-avoiding walk, Markering \cite{Marke24} has recently constructed a two-sided infinite self-avoiding walk in dimension 5 and higher. In contrast to our work, his approach is based on the two-point function alone, and does not use the lace expansion directly. 

\subsection{Contribution and strategy of proof} 
{\col
Apart from the main result, this paper offers another important contribution which is a new framework for {\colAJ bounding} lace expansion coefficients.

The standard approach has been to carefully study the events in lace expansion coefficients in order to extract as few diagrams as possible. In our context, the number of potential diagrams would prove to be extremely high even with great care. This led us to choosing a different approach where the bounding of events by diagrams is standardized, which produces a higher number of diagrams but focusing our efforts on compact notations for those diagrams and  a systematic approach for reducing the diagrams to easier ones.

More specifically, we introduce the notion of a diagrammatic event (see~Section~\ref{sect_diagram_graph}) which is an event composed of disjoint connections given by the edges of a multigraph. When performing a lace expansion argument, those diagrammatic events often appear and they are typically intersected with certain events loosely describe as \lq\lq a point intersects the cluster from a previous percolation\rq\rq~or~\lq\lq a point is doubly connected to a box $W$\rq\rq. We show that, under very general conditions, those intersections systematically provide the existence of \lq\lq an extra disjoint path\rq\rq in the sense that the initial diagrammatic event (corresponding to a multigraph $G$) is shown to be included in a {\it union} of diagrammatic events corresponding to a multigraph $G'$ obtained by adding edges to $G$ (see Lemma~\ref{lem_extra_path2} and Lemma~\ref{lem_extra_path4} for examples).  This union of diagrammatic events can be written with a very convenient graphical representation akin to standard lace expansion diagrams (this representation is described in Section~\ref{sect_gen_diag}). 

As mentioned previously, the straightforward nature of this method for bounding lace expansion events in terms of diagrammatic graphs has a drawback: it produces many diagrams that we have to bound. For this, we use a tool that we named the H-reduction (see Proposition~\ref{prop_Hred}) which had previously been used in some form throughout the lace expansion literature but not in a systematic way. Informally, this result states that a diagram where the edges form the letter \lq\lq H\rq\rq~ can be bounded, up to multiplicative constants, by the same diagram but where the central edge has been removed. This turns our problem of bounding diagrams into a graph theoretical problem of reducing graphs through successive H-reductions.

In Proposition~\ref{kill_bill}, we prove that any graph with good connectivity properties and which has three marked points (which will correspond to the origin and the ends of {\colAJ the expanded-out parts of} the two long arms of our conditioned cluster) can be H-reduced to one of the following two diagrams
 \[
\begin{minipage}{0.25\textwidth}
      \includegraphics[width=0.45\linewidth]{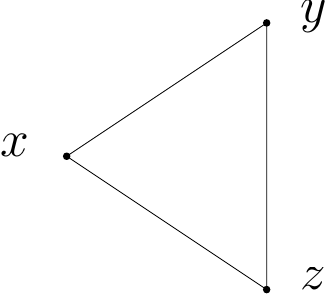}
    \end{minipage}
{\rm or }\qquad
    \begin{minipage}{0.25\textwidth}
      \includegraphics[width=0.45\linewidth]{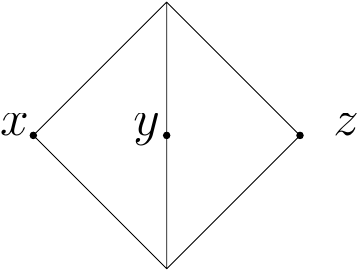}
    \end{minipage}
   \]
turning our problem of bounding tens of thousands of diagrams into understanding those two simple ones.  In Section~\ref{sect_big_class} we provide a tool for proving that many graphs which are obtained from natural building blocks of lace expansion have the connectivity properties (see Proposition~\ref{expand_class}) which allows them to be efficiently H-reduced (where Proposition~\ref{prop:strong-H-reduce} provides 
{\colAJ a suitable general tool}).

We believe this framework should be useful more broadly deriving a lace expansion for three arms and thus for the analysis of $k$-point functions. 
}

%
\subsection{Open problems.} 
{\col
Our construction paves the road to thrice-infinite incipient clusters, where there are three disjoint arms to infinity, which in turn would be the starting point for studying $k$-point functions for high-dimensional percolation clusters. 

A second set of open questions addresses the preliminaries. Is it possible to base the analysis on \eqref{eq:tauAsy} alone (so that it is true for $d\ge 11$ in the nearest-neighbor case), as done in \cite{ChattChinmHansoSosoe25}? Or is the (even weaker) strong triangle condition sufficient, as in \cite{HeydeHofstHulsh14b}, which then would include also long-range percolation models? {\colAJ We note that an advantage of our stronger assumptions is that one can extract from our arguments a quantitative estimate on the rate of convergence.}
{\col Indeed, if we have bounds on the error term in \eqref{eq:tauAsy}, e.g.\ $\tau(0,x)=A_d\, {|x|}^{2-d}(1+O(|x|^\varepsilon))$, then this gives quantitative control in \eqref{eq:taucomparison}, which may then be turned into polynomial bounds on the difference  $\Pbiic(F)-\P_{p_c}(F\mid\{ o \cnctd x \} \circ \{ o \cnctd x' \})$. 
}

{\col
\subsection{Organisation} 
We provide a double lace expansion in the upcoming section. The conceptional most interesting part is the expansion of a ``second arm'' in Section \ref{sec:secondarm}. 
In Section \ref{sec:bounds} we formulate diagrammatic bounds for the various lace expansion coefficients appearing in Section \ref{sec:expansion}, and use these bounds to prove Theorem \ref{thm:ExistenceBIIC} in Section \ref{sec:convergence}. 
{\colAJ In Section \ref{sec:BoundsFramework} we develop a general framework for bounding diagrams. Finally, in Section \ref{sec:bounding-diagrams} we prove the diagrammatic bounds of Proposition \ref{prop:PsiBd}}. 

As we are concentrating on leading order factors without any effort to constant factors, we simplify notation by writing $C$ for a generic though positive constant. It might depend on the dimension $d$ (and on $L$ in case of the spread-out model), but not on other arguments. Its value might change from line to line. }

\bigskip

\section{Expansion of the backbone probability}
\label{sec:expansion}
For cylinder events $F$ we let 
\[\sigma(F; x,x')   := \P_{p_c} [ F \cap \{ o \cnctd x \} \circ \{ o \cnctd x' \} ]\]
where $x \not= x'$. 
Then \( \sigma( x,x')  := \sigma(\Omega; x,x') = \P_{p_c} [ \{ o \cnctd x \} \circ \{ o \cnctd x' \} ], \)
and 
\[\Pbiic(F)=\lim_{|x|,|x'|,|x-x'|\to\infty}\frac{\sigma(F;x,x')}{\sigma(x,x')}\]
provided that the limit exists. 

It is sufficient to show the existence of the limit for special cylinder events $F$ of the form
\[ F 
   = \{ \omega : \text{$\omega(b) = 1$ for $b \in B_1$ and $\omega(b) = 0$ for $b \in B_2$} \}, \]
where $B_1$ and $B_2$ are disjoint and there union equals $E_W$, the set of all bonds with both 
endpoints in a finite box $W = W(M) = [-M,M]^d$ for some $M \ge 1$. Indeed, any element 
of $\mathfrak{F}_0$ can be written as a finite disjoint union of special cylinder events of the 
above type for some $M$. In what follows we denote by $\mathfrak{F}_{00}$ the set of special 
cylinder events.

\begin{proposition}\label{prop:laceexpansion}
There exists families of functions $(\tPsi_{n})_{ n,\ge0}$, $(R_{n,n'})_{n,n'\ge0}$ (defined on $\Zd\times\Zd$) as well as $(\hat\Psi_{n,n'})_{ n,n'\ge0}$ (defined on $\Zd\times\Zd\times \Zd$) and $(\bar\Psi_{n,n'})_{ n,n'\ge0}$ (defined on $\Zd\times\Zd\times \Zd\times\Zd$) such that 
\begin{align} 
	\sigma(x,x')
	=&
		\tPsi_{n}(x,x')
	+\sum_{u,v,u',v'}\bar\Psi_{n,n'}(u,v,u',v')\,pD(u,v)\tau(u,x)\,pD(u',v')\tau(u',x') \nonumber \\
	&
		+ \sum_{u,v}pD(u,v)\hat\Psi_{n,n'}(u,v,x')\tau(v,x) 
	+ R_{n,n'}(x,x')
\end{align}
	for $ n,n'\in\mathbb N, x,x'\in\Zd$. 
	Similarly, for any special cylinder event $F\in\mathfrak F_{00}$ there are $\tPsi_{n,n'}(F;x,x')$, $\bar\Psi_{n,n'}(F;u,v,u',v')$, $\hat\Psi_{n,n'}(F;u,v,u')$, $R_{n,n'}(F;x,x')$ such that 
\begin{align}
	 \sigma(F;x,x')
	 =&
		\tPsi_{n}(F;x,x')
	+\sum_{u,v,u',v'}\bar\Psi_{n,n'}(F;u,v,u',v')\,pD(u,v)\tau(v,x)\,pD(u',v')\tau(v',x') \nonumber \\
	&
		+ \sum_{u,v}pD(u,v)\hat\Psi_{n,n'}(F;u,v,x')\tau(v,x)  
	+ R_{n,n'}(F;x,x').
\end{align}
\end{proposition} 
In this section, we prove these identities thereby establishing the particular form of $\tPsi_{n}$,  $R_{n,n'}$, $\hat\Psi_{n,n'}$ and $\bar\Psi_{n,n'}$ cf.\ \eqref{eq:DefBarPiNN}--\eqref{eq:DefRNN}. 
In Section \ref{sec:bounds}, we provide bounds on these quantities. 
Note that it is sufficient to prove the expansion of $\sigma(F;x,x')$, as $\sigma(x,x')=\sigma(\Omega;x,x')$ {\col (note that $\Omega$ is a special cylinder event when we take $W$ to be the origin without any edges)}. 

What follows is a double expansion of the probability $\sigma(F;x,x')$: we first expand the connection from $o\cnctd x$ (similar as in \cite{HaraSlade90a,HofstJarai04}), and then expand the ``second arm'', i.e., the connection of the resulting structure to $x'$. 

\subsection{Expanding the first arm}
We first expand the probability of the connection $o \cnctd x$. 
{\col 
For events $E,F$, we denote the disjoint occurrence $E \circ F$ by 
\[ E\circ F = \big\{ \omega\in\Omega\mid \text{there exists subsets $B_1,B_2$ of the edges s.t.\ $B_1\cap B_2=\varnothing$, $\omega\in E|_{B_1}\cap F|_{B_2}$}\big\}, \]
where $E|_{B}$ is the set of those $\omega'$ whose restriction to the edge set $B$ implies that it is in $E$ (in such case we call $B$ a \emph{witness set} of $E$).}
We further write 
\[\{x\dcnctd y\}:=\{x\cnctd y\}\circ\{x\cnctd y\}, \qquad x,y\in\Zd,\] 
for the \emph{double connection} from $x$ to $y$ (i.e., either $x=y$ or there exist two edge-disjoint occupied paths from $x$ to $y$), and 
\[\{ W \dcnctd u \}:=\bigcup_{w,w'\in W}\{w\cnctd u\}\circ\{w'\cnctd u\},\qquad W\subset \Zd, u\in\Zd.\] 
A major tool in our analysis is the BK-inequality \cite{BergKeste85}, which states that the probability of disjoint occurrence of two event is bounded above by the product of their respective probabilities. Here we solely use the version of \cite[Prop.~9.12]{Slade06}. 
Let 
\[ E(F; u, x'):= E(F; u, x';\omega)
   := F \cap \{ W \dcnctd u \} \cap \left( \{ o \cnctd u \} \circ \{ o \cnctd x' \} \right), \]
and for $\{u_0,v_0\} \not\in E_W$, define
\[ \tC_0
   = \tC_0({\colAJ{W; u_0,v_0; \omega_0}})
   = \{ y \in \Zd : \text{$W \cnctd y$ on the configuration $(\omega_0)_{\{ u_0,v_0 \}}$} \}, \]
where
\[ (\omega_0)_{\{u,v\}}(b)
   = \begin{cases}
     \omega_0(b) & \text{if $b \not= \{u,v\}$;} \\
     0           & \text{if $b = \{u,v\}$.}
     \end{cases} \]
We will also need the following variant of $E$:
\[ \widetilde{E}(F; u, v, x'):= \widetilde{E}(F; u, v, x';\omega_0) 
   := F \cap \Bigg( \text{$\{ W \dcnctd u \} \cap 
      \Big( \{ o \cnctd u \} \circ \{ o \cnctd x' \} \Big)$ on $(\omega_0)_{\{u,v\}}$} \Bigg). \]
 
Our strategy now is as follows. We first expand the connection $\{0 \cnctd x\}$ as in \cite{HofstJarai04,HeydeHofstHulsh14b}, and keep the effect of $\{0 \cnctd x'\}$ in the ``0th order expansion term''. This yields quantities $\tilde \Psi$ and $\tilde R$. Subsequently, we make a second expansion, where we deal with these. 
In the following lemma, we add a subindex $0$ to $\P$ (and $\E$) to later distinguish it from further probability measures (which will be denoted $\P_1,\P_2,\dots$) which will arise when considering independent percolation configurations. Also, for $A\subset\Zd$, the notation $\tau^A(x,y)$ denotes $ \P [ x \cnctd y\text{ in }\Zd \setminus A ]$.

\begin{lemma}
\label{lem:1st-pivotal}
For $x \not= x'$ and $x, x' \not\in W$, we have
\eqnspl{e:1st-expansion-1st-stage}
{ \sigma(F;x,x')
  &= \P_0 [ E(F;x, x') ]
    + \sum_{\{u_0,v_0\} \not\in E_W} p D(u_0,v_0) 
      \E_0 \Big[ I [ \widetilde{E}(F;u_0,v_0,x') ] \tau^{\tC_0}(v_0,x) \Big]. }
\end{lemma}

\begin{proof}
First, note that $E(F; x, x')$ is a subset of the event $F \cap \{ o \cnctd x \} \circ \{ o \cnctd x' \}$ 
defining $\sigma(F; x, x')$. Observe that 
\[ \left( F \cap \{ o \cnctd x \} \circ \{ o \cnctd x' \} \right) \setminus E(F; x,x')
   = (F \cap \{ o \cnctd x \} \circ \{ o \cnctd x' \}) \cap \{ W \not\dcnctd x \}. \]
On the event in the right hand side, there exists a unique first pivotal bond, $(u_0,v_0)$,
for the connection from $W$ to $x$. {\col(Even though we generally work with undirected edges, we think of the pivotal edge $(u_0,v_0)$ to be directed in the sense that any self-avoiding path from $W$ to $x$ must first pass through $u_0$, then traverse the edge to $v_0$, and from there continue to $x$.)} 
We claim that 
\eqnspl{e:events-eq}
{ &\left[ \left( F \cap \{ o \cnctd x \} \circ \{ o \cnctd x' \} \right) \setminus E(F; x,x') \right]
  \cap \{ \text{$(u_0,v_0)$ is the 1st pivotal for $W \cnctd x$} \} \\
  &\qquad = \widetilde{E}(F;u_0,v_0,x')
    \cap \big\{ \text{$(u_0,v_0)$ occupied} \big\}
    \cap \big\{ \text{$v_0 \cnctd x$ in $\Zd \setminus \tC_0$} \big\}. }   
The claim of the lemma will follow easily from this equality.

Let us first show that the left hand side is contained in the right hand side. 
Assume that the left hand event occurs. Since $o$ is connected to $x$, $o \in W$ and $(u_0,v_0)$ is
pivotal for $W \cnctd x$, we have that $(u_0,v_0)$ is occupied. Fix bond-disjoint occupied paths 
$\pi_1$ from $o$ to $x$ and $\pi_2$ from $o$ to $x'$. The pivotality of $(u_0,v_0)$ implies that
$\pi_1$ uses $(u_0,v_0)$, and in particular $F \cap \{ o \cnctd u_0 \} \circ \{ o \cnctd x' \}$
occurs. In order to deduce that $\widetilde{E}(F;u_0,v_0,x')$ occurs, we use that $(u_0,v_0)$ is the \emph{first} 
pivotal for the connection from $W$ to $x$, to see that $W \dcnctd u_0$, without using the bond
$(u_0,v_0)$. Let $\pi'_1$ be the portion of 
$\pi_1$ between $o$ and $u_0$, so $\pi'_1$ does not use $\{u_0,v_0\}$. 
Since  $\{u_0,v_0\}$ is used by $\pi_1$, which is edge-disjoint from $\pi_2$, it is not used by $\pi_2$. 
It follows that $\widetilde{E}$ occurs.
Finally, the existence of $\pi_1$ and the pivotality of $(u_0,v_0)$ also implies that $v_0 \cnctd x$ 
and that this connection takes place in $\Zd \setminus \tC_0$.

Let us now show that the right hand side of \eqref{e:events-eq} is contained in the left hand side.
The event $\widetilde{E}(F;u_0,v_0,x')$ ensures that we can find bond-disjoint paths $\tilde{\pi}_1$ 
from $o$ to $u_0$ and $\tilde{\pi}_2$ from $o$ to $x'$ that do not use $\{u_0,v_0\}$. 
Now the concatenation of $\tilde{\pi}_1$,
$(u_0,v_0)$ and a path from $v_0$ to $x$ that avoids $\tC_0$ shows that 
$\{ o \cnctd x \} \circ \{ o \cnctd x' \}$ occurs. Since $x \not\in \tC_0$, we have that $(u_0, v_0)$
is pivotal for $W \cnctd x$. Since $W \dcnctd u_0$ without using $(u_0,v_0)$ 
(due to the event $\widetilde{E}(F;u_0,v_0,x')$),
we have that $(u_0, v_0)$ is the first pivotal for $W \cnctd x$.

To re-write the probability of the right hand side of \eqref{e:events-eq},
let us partition according to the value of the cluster $\tC_0$.
(Note we have $p \le p_c$, so $\P_p [ |\tC_0| < \infty ] = 1$).
{\colAJ{The event $\widetilde{E}(F;u_0,v_0,x')$ is determined by the value of 
$\tC_0$ and the event $\{ \text{$(u_0, v_0)$ is occupied} \}$ is independent of 
$\tC_0$. 
We get that the probability of the right hand side of \eqref{e:events-eq} equals
\eqnspl{e:firstFactorization}
{ &p D(u_0,v_0) \E_0 \left[ I[ \widetilde{E}(F;u_0,v_0,x') ] 
     \tau^{\tC_0}(v_0,x) \right]. }
}}
\end{proof}

We proceed by writing
\eqnspl{e:tauS}
{ \tau^{\tC_0}(v_0,x)
  = \tau(v_0,x) - \P_1 [ v_0 \stackrel{\tC_0}{\cnctd} x ], }
{\colAJ{where $\{ v_0 \stackrel{S}{\cnctd} x \}$ is the event that $v_0 \cnctd x$, but the connection
is broken, when all edges touching $S$ are turned off.
}}
The index $1$ emphasizes that $\tC_0$ is determined by a percolation configuration
$\omega_0$, whereas the connection event $v_0 \stackrel{\tC_0}{\cnctd} x$ is on a percolation
configuration $\omega_1$ that is independent of $\omega_0$.

Define 
\eqnspl{e:PsiZero}
{ \tPsi^{(0)}(F;x,x')
  &:= \P [ E(F; x, x') ] \\
  \hat{\tPsi}^{(0)}(F; u_0, v_0, x')
  &:= \P \left[ \widetilde{E}(F;u_0,v_0,x') \right]. } 
We get the following for the `1st stage' of the expansion.
\eqnspl{e:1st-stage}
{ \sigma(F;x,x') 
  &= \tPsi^{(0)}(F;x,x') 
     + \sum_{u_0,v_0} p D(u_0,v_0) \hat{\tPsi}^{(0)}(F;u_0,v_0,x') \tau(v_0,x) 
     - \tR_0(F;x,x'), }
where
\eqnspl{e:R0}
{ \tR_0(F;x,x')
  &:= \sum_{u_0,v_0} p D(u_0,v_0) \E_0 \left[
      I \left[ \widetilde{E}(F;u_0,v_0,x') \right] 
      \P_1 \left[ v_0 \stackrel{\tC_0}{\cnctd} x \right] \right]. }
Note that for obtaining diagrammatic upper bounds, we will be able to neglect the extra 
condition that the connection from $o$ to $x'$ does not use $\{u_0,v_0\}$. For example, 
we have the simple upper bound $\hat{\tPsi}^{(0)}(F;u_0,v_0,x') \le \tPsi^{(0)}(F;u_0,x')$.

The Hara-Slade expansion \cite{HaraSlade90a} provides an expression for the probability of 
$v_0 \stackrel{S}{\cnctd} x$, for any finite set $S$, that we insert without change into \eqref{e:R0}. 
We recall some notation for this.
Write $P_{(v_0,x)}$ for the set of pivotal bonds for the connection $v_0 \cnctd x$ in the configuration 
$\omega_1$. We call $E'(v_0,x;S)=E'(v_0,x;S;\omega)$ the event that $v_0 \stackrel{S}{\cnctd} x$ and either $P_{(v_0,x)} = \es$ or 
for all $(u'',v'') \in P_{(v_0,x)}$ we have that $v_0 \stackrel{S}{\cnctd} u''$ does not occur 
(i.e.~for all such pivotal bonds we have $v_0 \cnctd u''$ in $\Zd \setminus S$).

We call 
\[
E(v_0,u',v',x;S)=\{v_0 \stackrel{S}{\cnctd} x\} \cap  E'(v_0,u';S) \cap \{(u',v')\text{ is 
occupied and pivotal for }v_0 \cnctd x\},
\]
 (all in the configuration $\omega_1$). We call $(u',v')$ the \textbf{cutting bond}, which is the \emph{first} pivotal bond along the 
connection from $v_0$ to $x$ such that $v_0 \stackrel{S}{\cnctd} u'$.

These definitions allow us to write the following decomposition: 
\eqnspl{eq:expansion}
{ \P_1 \left[ v_0 \stackrel{S}{\cnctd} x \right]
  = \P_1 \left[ E'(v_0,x;S) \right] + \sum_{(u',v')} \P_1 \left[ E(v_0,u',v',x;S) \right]. }
Let 
\[ \tC_1 
   = \tC_1^{\{u_1,v_1\}}(v_0)
   = \{ y \in \Zd : \text{$v_0 \cnctd y$ in $(\omega_1)_{\{u_1,v_1\}}$} \} \]
be the restricted cluster of $v_0$ in $\omega_1$ after the bond $\{u',v'\}$ is made vacant. 
The following lemma is taken from \cite{HaraSlade90a}. 

\begin{lemma}[{Hara-Slade \cite{HaraSlade90a}}]
\label{lem:on-in}
We have
\eqnsplst
{ \P_1 \left[ E(v_0,u',v',x; \tC_1) \right]
  = p D(u',v') \E_1 \left[ I [ E'(v_0,u';\tC_1) ] \tau^{\tC_1}(v',x) \right]. }
\end{lemma}

Using this, we get the following for the `2nd stage' of the expansion:
\eqnspl{e:2nd-stage}
{ \sigma(F;x,x') 
  &= \tPsi^{(0)}(F;x,x') 
     + \sum_{u_0,v_0} p D(u_0,v_0) \hat{\tPsi}^{(0)}(F;u_0,v_0,x')
      \tau(v_0,x) \\
  &\quad - \sum_{u_0,v_0} p D(u_0,v_0)
      \E_0 \left[ I \left[ \widetilde{E}(F;u_0,v_0,x') \right] 
      \E_1 \left[ E'(v_0,x;\tC_0) \right] \right] \\
  &\quad - \sum_{u_0,v_0} p D(u_0,v_0) \sum_{u_1,v_1} p D(u_1,v_1) \\
  &\quad\qquad \E_0 \left[
      I \left[ \widetilde{E}(F;u_0,v_0,x') \right] 
      \E_1 \left[ E'(v_0,u_1;\tC_0) \right] \right] \tau(v_1,x) \\
  &\quad + \tR_1(x,x'), }
where
\eqnspl{e:R1}
{ &\tR_1(F;x,x') \\
  &\quad := \sum_{u_0,v_0} p D(u_0,v_0) \sum_{u_1,v_1} p D(u_1,v_1) \\
  &\quad\qquad \E_0 \left[ I \left[ \widetilde{E}(F; u_0, v_0, x') \right] 
     \E_1 \left[ E'(v_0,u_1;\tC_0) \P_2 \left[ v_1 \stackrel{\tC_1}{\cnctd} x \right] \right] \right]. }
For better organisation we can write this in more compact form using 
\eqn{eqDefPsi1}{
	 \tPsi^{(1)}(F;u,x')= \sum_{u_0,v_0} p D(u_0,v_0)
      \E_0 \left[ I \left[ \widetilde{E}(F;u_0,v_0,x') \right] 
      \E_1 \left[ E'(v_0,u;\tC_0) \right] \right]}
as 
\eqnspl{e:2nd-stage-a}
{ \sigma(F;x,x') 
  &= \tPsi^{(0)}(F;x,x')-\tPsi^{(1)}(F;x,x') \\
  &\qquad  +\sum_{u,v} \left(\hat{\tPsi}^{(0)}(F;u,v,x')-\tPsi^{(1)}(F;u,x')\right)p D(u,v)\tau(v,x) 
     +\tR_1(F;x,x').
}

Applying Lemma \ref{lem:on-in} and inclusion-exclusion repeatedly, we get the general form of the 1st expansion.
To this end, we need a couple of definitions: For $N \ge 2$ define
\eqnspl{e:tPsiN}
{ \tPsi^{(N)}(F;x,x')
  &= \sum_{u_0,v_0} \dots \sum_{u_{N-1}, v_{N-1}} \prod_{i=0}^{N-1} p D(u_i,v_i)
     \E_0 \Big[
     I \left[ \widetilde{E}(F;u_0,v_0,x') \right]  \\
  &\qquad \E_1 \Big[ I \big[ E'(v_0,u_1; \tC_0) \big] \cdots 
     \E_{N-1} \Big[ I \big[ E'(v_{N-2},u_{N-1}; \tC_{N-2}) \big] \\
  &\qquad \E_N \Big[ I \big[ E'(v_{N-1},x; \tC_{N-1}) \big] \Big] \Big] \cdots \Big] \Big]. }
Here $\tC_{j} = \tC_j^{\{u_j,v_j\}}(v_{j-1})$ for $j \ge 1$.
Also, for $N \ge 2$ we define
\eqnsplst
{ \tR_N(F;x,x')
  &= (-1)^{N+1} \sum_{u_0,v_0} \dots \sum_{u_{N}, v_{N}} \prod_{i=0}^{N} p D(u_i,v_i)
     \E_0 \Big[
     I \left[ \widetilde{E}(F;u_0,v_0,x') \right] \\
  &\qquad \E_1 \Big[ I \big[ E'(v_0,u_1; \tC_0) \big] \cdots 
     \E_N \Big[ I \big[ E'(v_{N-1},u_N; \tC_{N-1}) \big] \\
  &\qquad \P_{N+1} \Big[ v_N \stackrel{\tC_N}{\cnctd} x \Big] \Big] \cdots \Big] \Big]. }

In order to split off the parts that behave differently from the 1-arm case and to keep the rest 
of the expression compact, we introduce the following $\omega_1$-random variables: for $N \ge 3$, let
\eqnsplst
{ &\Xi^{(N)}(v_1,x;\tC_1) \\
  &\quad := \sum_{u_2,v_2} \dots \sum_{u_{N-1},v_{N-1}} \prod_{i=2}^{N-1} p D(u_i,v_i)
     \E_2 \Big[ I \big[ E'(v_1,u_2; \tC_1) \big] \cdots
     \E_{N-1} \Big[ I \big[ E'(v_{N-2},u_{N-1};\tC_{N-2}) \big] \\
  &\hskip6em \E_{N} \Big[ I \big[ E'(v_{N-1},x;\tC_{N-1}) \big] \Big] \Big] \cdots \Big] }
and for $N = 2$ we set
\eqnst
{ \Xi^{(2)}(v_1,x;\tC_1) 
  := \E_2 \Big[ I \big[ E'(v_1,x;\tC_1) \big] \Big]. }

Similarly, in order to abbreviate the remainder terms, for $N \ge 2$ we define
\eqnsplst
{ &\Chi^{(N)}(v_1,x;\tC_1) \\
  &\quad := \sum_{u_2,v_2} \dots \sum_{u_{N},v_{N}} \prod_{i=2}^{N} p D(u_i,v_i)
     \E_2 \Big[ I \big[ E'(v_1,u_2; \tC_1) \big] \cdots
     \E_{N} \Big[ I \big[ E'(v_{N-1},u_{N};\tC_{N-1}) \big] \\
  &\qquad\qquad\qquad \P_{N+1} \Big[ v_N \stackrel{\tC_N}{\cnctd} x \Big] \Big] \cdots \Big]. }
For $N = 1$, we set
\eqnsplst
{ \Chi^{(1)}(v_1,x;\tC_1)
  &:= \P_{2} \Big[ v_1 \stackrel{\tC_1}{\cnctd} x \Big]. }
This allows us to rewrite $\tPsi^{(N)}$ with $N \ge 2$ in more compact form as
\eqnspl{e:tPsiN-short}
{ \tPsi^{(N)}(F;x,x')
  &= \sum_{u_0,v_0} \sum_{u_{1},v_{1}} p D(u_0,v_0) p D(u_1,v_1) 
     \E_0 \times \E_1 \Big[ I \left[ \widetilde{E}(F;u_0,v_0,x') \right] \\
  &\qquad\qquad I \big[ E'(v_0,u_1;\tC_0) \big] \Xi^{(N)}(v_1,x;\tC_1) \Big]. }
We also rewrite $\tR_N$ for $N \ge 1$ in more compact form as
\eqnspl{e:RN-short}
{ \tR_N(F;x,x')
  &= (-1)^{N+1} \sum_{u_0,v_0} \sum_{u_{1}, v_{1}} p D(u_0,v_0) p D(u_1,v_1) 
     \E_0 \times \E_1 \Big[ I \left[ \widetilde{E}(F;u_0,v_0,x') \right] \\
  &\qquad\qquad I \big[ E'(v_0,u_1;\tC_0) \big] \Chi^{(N)}(v_1,x;\tC_1) \Big]. }

\begin{proposition}
\label{prop:1st-expansion}
For any $n \ge 2$, we have the expansion 
\eqnspl{e:N-th-stage}
{ \sigma(F; x,x')
  &= \sum_{N=0}^n (-1)^N\tPsi^{(N)}(F;x,x') \\
  &\qquad + \sum_{u,v} p D(u,v) \left( \hat{\tPsi}^{(0)}(F;u,v,x') 
  + \sum_{N=1}^n (-1)^N \tPsi^{(N)}(F;u,x') \right) \tau(v,x) \\
  &\qquad + \tR_n(F;x,x'). }
\end{proposition}

Everything so far is conceptually following the original Hara-Slade lace expansion \cite{HaraSlade90a}.
The new ingredient is a second (and separate) expansion for the connection to $x'$, which we derive next. 

\subsection{Expanding the second arm}\label{sec:secondarm}
We now turn to expanding the arm towards $x'$ contained in the cluster $\tC_0$.
Note that this cluster, that lives on the configuration $\omega_0$, only directly interacts with $\omega_1$
through the terms $I[ E'(v_0,u_1;\tC_0) ](\omega_1)$. Since the terms containing $\tPsi^{(m)}(F;x,x')$
will be error terms, we only need to expand further the terms that contain a convolution with $\tau(v,x)$
for some $m = 0, 1, \dots$. Among these, the $m = 0$ case will be special (and somewhat simpler than 
the rest), so we start with assuming $m \ge 1$.

We will need the notation 
\eqnsplst
{ \tB_1
  &= \tB_1(v_0,u_1;\omega_1)
  := \left\{ y \in \Zd : \{ v_0 \cnctd_{\omega_1} y \} \circ \{ y \cnctd_{\omega_1} u_1 \} \right\} \\
  &= \text{`backbone between $v_0$ and $u_1$'}. }
Note that while in the definition of $\tC_0$ we require the bond $\{ u_0, v_0 \}$ to be set vacant, we will 
have no need for a similar requirement in the definition of $\tB_1$.

In \eqref{e:N-th-stage} we are interested in realisations of the subgraph $\tC_0$ 
that contain an $(\omega_0)_{\{u_0,v_0\}}$-connection from $o$ to $x'$.
In addition, due to the presence of the event $E'(v_0,u_1;\tC_0)$, the intersection of the 
cluster $\tC_0$ with $\tB_1$ has to realize this event, which puts contraints 
on the pattern of the intersection. Our aim is to split $\tC_0$ at a bond that is pivotal 
for $(W \cup \tB_1) \cnctd x'$ in such a way that the part containing $W$, that we will call 
$\tD_0$, witnesses the required intersection pattern ($E'(v_0,u_1;\tD_0)$ occurs) 
and the part containing $x'$ avoids $\tB_1$. This motivates the definitions below.

Keeping $\omega_1$ fixed, let $P'_{x'}$ denote the set of pivotal bonds for 
$(W \cup \tB_1) \cnctd x'$ when $(u_0,v_0)$ is held vacant, that is
\eqnsplst
{ P'_{x'}
  &= P'_{x'}(\{u_0,v_0\}, \tB_1(\omega_1);\omega_0) \\
  &:= \left\{ (u',v') : \parbox{7.5cm}{$\{u',v'\} \not\in E_W$ and any $(\omega_0)_{\{u_0,v_0\}}$-occupied 
     path from $W \cup \tB_1$ to $x'$ uses $(u',v')$} \right\}, } 
whenever there is no $(\omega_0)_{\{u_0,v_0\}}$-occupied path between $W \cup \tB_1$ to $x'$, we define
$P'_{x'} = \es$.
We define
\eqnsplst
{ \tD_0 
  = \tD_0({\colAJ{W; u_0,v_0,u'_0,v'_0;\omega_0}})
  := \left\{ y \in \Zd : 
     \text{$W \cnctd_{(\omega_0)_{\{ u_0, v_0\} \cup \{u'_0,v'_0\}}} y$} \right\}, } 
which is the union of the clusters intersecting $W$ after both bonds $\{ u_0, v_0 \}$ and 
$\{ u'_0, v'_0 \}$ are turned off.
Note that the set of vertices $\tD_0$, together with the set of bonds $\tD_0$ induces from the configuration
$(\omega_0)_{\{ u_0, v_0\} \cup \{u'_0,v'_0\}}$ is a subgraph of $\tC_0$.

The key steps in starting the second expansion are the following definitions and lemma.
Let 
\eqnsplst
{ \tC(\tB_1)
  &= \tC(u_0,v_0,u_1,u'_0,v'_0;\omega_0,\omega_1) \\
  &:= \left\{ y \in Z^d : \tB_1 \cnctd_{(\omega_0)_{\{u_0,v_0\} \cup \{u'_0,v'_0\}}} y \right\} \\
  &= \{ \text{`all $\omega_0$-occupied paths starting in $\tB_1$ that do not use $\{u_0,v_0\}$ nor
    $\{u'_0,v'_0\}$'} \}. }
Let $E''(F;u_0,v_0,u_1,x')$ denote the following event measurable with respect to the 
pair $(\omega_0,\omega_1)$, and which depends on $\omega_1$ only through the set $\tB_1$ {\col (note: $E'(v_0,u_1,\tC_0)$ {\colAJ{only}} depends on $\tB_1$)}:
\eqnsplst
{E''(F;u_0,v_0,u_1,x'; \omega_0,\omega_1)
  &:= \widetilde{E}(F;u_0,v_0,x';\omega_0) \cap E'(v_0,u_1,\tC_0;\omega_1) \cap \{ P_{x'} = \es \}. } 
We will also need the following variant of $E''$: 
\eqnsplst
{ &\hat{E}''(F;u_0,v_0,u_1,u'_0,v'_0;\omega_0,\omega_1) \\ 
  &\quad := \widetilde{E}(F;u_0,v_0,u'_0;\omega_0) \cap E'(v_0,u_1,\tD_0(F;u_0,v_0,u'_0,v'_0);\omega_1) 
     \cap \{ P'_{u'_0} = \es \}. } 
We will mostly drop the dependence on $\omega_0$ and $\omega_1$ in $E''$ and $\hat E'$. 

\begin{lemma}
\label{lem:factor-2nd} 
For any $\{u_0,v_0\} \not\in E_W$ and $u_1 \in \Zd$ we have
\eqnspl{e:factor-event}
{ &\widetilde{E}(F;u_0,v_0,x';\omega_0) \cap E'(v_0,u_1,\tC_0;\omega_1) 
   = E''(F;u_0,v_0,u_1,x';\omega_0,\omega_1) \cup \\
   &\qquad \bigcup_{(u'_0,v'_0) : \{u'_0,v'_0\} \not\in E_W \cup \{u_0,v_0\}}
    \Big( \hat{E}''(F;u_0,v_0,u_1,u'_0,v'_0;\omega_0,\omega_1) \cap \big\{ \text{$\{u'_0,v'_0\}$ $\omega_0$-occupied} \big\} \\
   &\qquad\qquad\qquad\qquad\qquad 
   \cap \left\{ \text{$v'_0 \cnctd_{\omega_0} x'$ in $\Zd \setminus (\tD_0 \cup \tC(\tB_1))$} \right\} \Big), }
where $\tC_0 = \tC_0(W;u_0,v_0;\omega_0)$, $\tD_0 = \tD_0(W;u_0,v_0,u'_0,v'_0;\omega_0)$ and the union on the right hand side is disjoint.
\end{lemma}

The following corollary will be an easy consequence (and the main purpose) of Lemma \ref{lem:factor-2nd}.

\begin{corollary}
\label{lem:factor-2nd-rv}
(i) For any $\{u_0,v_0\} \not\in E_W$ and $u_1, v_1, u \in \Zd$ and $N \ge 2$ we have
\eqnspl{e:factor-rv}
{ &\E_0 \times \E_1 \Big[ I \left[ \widetilde{E}(F;u_0,v_0,x') \right]
       I \big[ E'(v_0,u_1;\tC_0) \big] \Xi^{(N)}(v_1,u,\tC_1) \Big] \\
  &\quad = \E_0 \times \E_1 \Big[ I \left[ E''(F;u_0,v_0,u_1,x') \right] \Xi^{(N)}(v_1,u,\tC_1) \Big] 
      + \!\!\!\!\!\sum_{\{u'_0,v'_0\} \not\in E_W \cup \{u_0,v_0\}} \!\!\!\!\! p D(u'_0,v'_0) \\
  &\qquad\qquad\qquad \E_0 \times \E_1 \Big[ I \left[ \hat{E}''(F;u_0,v_0,u_1,u'_0,v'_0) \right] 
       \Xi^{(N)}(v_1,u,\tC_1) \tau^{\tD_0 \cup \tC(\tB_1)}(v'_0,x') \Big]. }
(ii) \emph{($N = 1$ case)} For any $\{u_0,v_0\} \not\in E_W$ and $u \in \Zd$ we have  
\eqnspl{e:factor-rv-N=1}
{ &\E_0 \times \E_1 \Big[ I \left[ \widetilde{E}(F;u_0,v_0,x') \right]
       I \big[ E'(v_0,u;\tC_0) \big] \Big] \\
  &\quad = \E_0 \times \E_1 \Big[ I \left[ E''(F;u_0,v_0,u,x') \right] \Big] \\
  &\quad\quad + \!\!\!\sum_{\{u'_0,v'_0\} \not\in E_W \cup \{u_0,v_0\}} \!\!\! p D(u'_0,v'_0)
       \E_0 \times \E_1 \Big[ I \left[ \hat{E}''(F;u_0,v_0,u,u'_0,v'_0) \right] 
       \tau^{\tD_0 \cup \tC(\tB_1)}(v'_0,x') \Big]. }
\end{corollary}

\begin{proof}[Proof of Lemma \ref{lem:factor-2nd}]
We first show that the left hand side of \eqref{e:factor-event} is contained in the right hand side.
Negating the event $P'_{x'} = \es$ in the definition of $E''$, let $(u'_0,v'_0)$
be the first pivotal in $P'_{x'}$, when paths are traversed from $W \cup \tB_1$ towards $x'$.

Since $\{u'_0,v'_0\}$ is occupied in $(\omega_0)_{\{u_0,v_0\}}$, it cannot equal $\{u_0,v_0\}$. 
By the definition of $P'_{x'}$, we also have $\{u'_0,v'_0\} \not\in E_W$. We now check that the 
required events inside the union in the right hand side of \eqref{e:factor-event} all occur.
\begin{itemize}

    \item[(a)] \emph{The event $\hat{E}'' = \widetilde{E}(F;u_0,v_0,u'_0) \cap E'(v_0,u_1,\tD_0(W;u_0,v_0,u'_0,v'_0)) \cap \{ P_{u'_0} = \es \}$.}

\begin{itemize}

    \item[(a1)] \emph{Occurrence of $\widetilde{E}(F,u_0,v_0,u'_0)$.}
    Since $\widetilde{E}(F,u_0,v_0,x')$ is present in the left hand side of 
\eqref{e:factor-event}, we have that $F$ occurs and $\{ W \dcnctd u_0 \}$ without using 
$(u_0,v_0)$. We can also select bond-disjoint $(\omega_0)_{\{u_0,v_0\}}$-occupied paths 
witnessing the connections $o \cnctd u_0$ and $o \cnctd x'$, respectively, present in this event. 
The path from $o$ to $x'$ must use $(u'_0,v'_0)$, so the portion up to $u'_0$
witnesses the connection $o \cnctd u'_0$, and is disjoint from the connection from $o$ to $u_0$. 
It follows that $\widetilde{E}(F,u_0,v_0,u'_0)$ occurs.

     \item[(a2)] \emph{Occurrence of $E'(v_0,u_1,\tD_0(W;u_0,v_0,u'_0,v'_0))$.} 
     Since $E'(v_0,u_1,\tC_0(W;u_0,v_0))$ occurs in the left hand side of \eqref{e:factor-event},
     it is sufficient to show that 
     \[ \tC_0(W;u_0,v_0) \cap \tB_1 
        = \tD_0(W;u_0,v_0,u'_0,v'_0) \cap \tB_1. \]
     Indeed, the right hand side here is contained in the left hand side, and if the left hand side was
     strictly greater, there would be an $(\omega_0)_{\{u_0,v_0\}}$-occupied path from 
     $v'_0$ to $\tB_1$, contradicting pivotality of $(u'_0,v'_0)$.

    \item[(a3)] \emph{Occurrence of $\{ P'_{u'_0} = \es \}$.} 
    When $P'_{u'_0} \not= \es$ any element $(u',v')$ of it is also a member of $P'_{x'}$, and any such
    $(u',v')$ necessarily precedes $(u'_0,v'_0)$. Thus the choice of 
    $(u'_0,v'_0)$ implies that $P'_{u'_0} = \es$.

\end{itemize}

\item[(b)] \emph{The event $\{ \text{$\{u'_0,v'_0\}$ occupied} \}$.}
It was already noted above that this event occurs.

\item[(c)] \emph{The event $\{ \text{$v'_0 \cnctd x'$ in $\Zd \setminus (\tD_0 \cup \tC(\tB_1))$} \}$.}
Choose bond-disjoint $(\omega_0)_{\{u_0,v_0\}}$-occupied paths $\pi_1$ and $\pi_2$ from $o$ to $u_0$ 
and from $o$ to $x'$ (witnessing these connections present in the event $\widetilde{E}(F,u_0,v_0,x')$ in the 
left hand side of \eqref{e:factor-event}). Then $\pi_2$ uses the bond $(u'_0,v'_0)$. Let $\pi'_2$ be the 
portion of $\pi_2$ from $v'_0$ to $x'$. To prove claim (c) it will be sufficient to show that $\pi'_2$ does not pass through 
$\tD_0$ nor $\tC(\tB_1)$. The pivotality of the oriented edge $(u'_0,v'_0)$ implies that 
$\pi'_2$ cannot pass through $\tD_0$. Suppose that $\pi'_2$ included a vertex in $\tC(\tB_1)$. That is, 
there exists an $(\omega_0)_{\{u_0,v_0\} \cup \{u'_0,v'_0\}}$-open path between $\pi'_2$ and $\tB_1$.
This would contradict pivotality of $(u'_0,v'_0)$.
\end{itemize}

Let us now show that the right hand side of \eqref{e:factor-event} is contained in the left hand side.

\begin{itemize}

    \item[(a)] \emph{The event $E''$.} By its definition, this event is a subset of the event in the 
left hand side.

    \item[(b)] \emph{A member of the union over $(u'_0,v'_0)$.} We consider the occurrence of the
two events in the left hand side of \eqref{e:factor-event} in turn. We also show that 
$(u'_0,v'_0) \in P'_{x'}$, and that it is the first element in $P_{x'}$, implying disjointness over
$(u'_0,v'_0)$.

    \begin{itemize}
    
        \item[(b1)] \emph{Occurrence of the event} \\
        $\widetilde{E}(F,u_0,v_0,x') = F \cap \Big( \text{$\{ W \dcnctd u_0 \} \cap 
      \left( \{ o \cnctd u_0 \} \circ \{ o \cnctd x' \} \right)$ on $(\omega_0)_{\{u_0,v_0\}}$} \Big)$. \\
      It clear that $F$ occurs and that $\{ W \dcnctd u_0 \}$ without using $(u_0,v_0)$, since these are 
      part of the event 
      $\widetilde{E}(F,u_0,v_0,u'_0)$ which is included in the event $\hat{E}''$ in the right 
      hand side term being considered. In order to show the required bond-disjoint connections,
      let us fix bond-disjoint occupied paths $\pi_1, \pi_2$ from $o$ to $u_0$ and from $o$ to $u'_0$,
      respectively, that do not use $\{u_0,v_0\}$. Existence of such paths is again guaranteed by 
      the event $\widetilde{E}(F,u_0,v_0,u'_0)$. We can further fix an $\omega_0$-occupied path 
      $\pi_3$ from $v'_0$ to $x'$ that avoids $\tD_0 \cup \tC(\tB_1)$, due to the last member 
      of the right hand side term being considered. Since we also know that 
      $\{u'_0,v'_0\}$ is occupied, the concatenation of $\pi_2$, $\{u'_0,v'_0\}$, and $\pi_3$ 
      furnishes an $\omega_0$-occupied path $\pi_4$ from $o$ to $x'$ (at this point, {\colAJ{it is 
      not yet clear that $\pi_4$ is self-avoiding}}).
      
      We show that $\pi_4$ does not use $\{u_0,v_0\}$, which amounts to showing this for $\pi_3$. 
      This is immediate, since $v_0 \in \tB_1$, which $\pi_3$ avoids. We next show that 
      $\{u'_0,v'_0\}$ and $\pi_3$ are bond-disjoint from $\pi_1$. First, note that $\pi_1$ 
      cannot use the oriented bond $(v'_0,u'_0)$, because than we had $v'_0 \in \tD_0$, 
      a contradiction. Consider next the possibility that $\pi_1$ uses the oriented bond 
      $(u'_0,v'_0)$. Then, since $W \dcnctd u_0$, there is a path $\pi_5$ from $W$ to $u_0$ 
      that does not use $\{u'_0,v'_0\}$, nor $\{u_0,v_0\}$. Considering the concatenation of $\pi_5$ 
      and the piece of the reversal of $\pi_1$ between $u_0$ and $v'_0$, it follows that $v'_0 \in \tD_0$, 
      and this is again a contradiction. Finally, we show that $\pi_3$ is even vertex-disjoint from $\pi_1$.
      Indeed, assume that $\pi_3$ visits a vertex $y \in \pi_1$. Then we again have 
      $v'_0 \in \tD_0$ (via the piece of $\pi_1$ from $o$ to $y$ and the piece of the reversal of 
      $\pi_3$ from $y$ to $v'_0$), a contradiction. Observe that $\pi_2$ and $\pi_3$ cannot use 
      $\{u'_0,v'_0\}$, since that would contradict $v'_0 \not\in \tD_0$.  

      \item[(b2)] \emph{$(u'_0,v'_0)$ is the first pivotal in $P'_{x'}$.} \\
      The arguments in the previous part (b1) showed that there is an $(\omega_0)_{(u_0,v_0)}$-occupied
      path between $W$ and $x'$ (namely the path $\pi_4$ there). Let now $\pi$ be any 
      $(\omega_0)_{(u_0,v_0)}$-occupied path from $W \cup \tB_1$ to $x'$, and assume that it does not use
      $(u'_0,v'_0)$. It cannot use $(v'_0,u'_0)$ instead, as that would contradict that 
      $v'_0 \not\in \tD_0 \cup \tC(\tB_1)$. Hence $\pi$ does not use $\{u'_0,v'_0\}$ in either direction.
      Due to $x' \not\in \tC(\tB_1)$, we have that $\pi$ cannot intersect $\tB_1$. 
      Therefore, $\pi$ connects $W$ to $x'$ in $(\omega_0)_{(u_0,v_0),(u'_0,v'_0)}$. 
      This, however, contradicts $x' \not\in \tD_0$. Hence $(u'_0,v'_0) \in P_{x'}$.
      Finally, assume that there exists $(u',v') \in P_{x'}$ preceding $(u'_0,v'_0)$.
      Then any $(\omega_0)_{(u_0,v_0)}$-occupied path from $W \cup \tB_1$ to $u'_0$
      can be extended, via $(u'_0,v'_0)$ and $\pi_3$ to an 
      $(\omega_0)_{(u_0,v_0)}$-occupied path from $W \cup \tB_1$ to $x'$. Thus this path
      must use $(u',v')$ (before reaching $u'_0$), and therefore $(u',v') \in P'_{u'_0}$,
      a contradiction to the event $\{ P'_{u'_0} = \es \}$ present in $\hat{E}''$.

      \item[(b3)] \emph{Occurrence of $E'(v_0,u_1,\tC_0)$.} \\
      We show that $\tC_0 \cap \tB_1 = \tD_0 \cap \tB_1$, which implies the statement, since 
      $E'(v_0,u_1,\tD_0)$ was assumed. Suppose there exists $y \in \tB_1 \cap (\tC_0 \setminus \tD_0)$.
      Then an $(\omega_0)_{(u_0,v_0)}$-occupied path exists from $W$ to $y$. This path must use
      $\{u'_0,v'_0\}$, because $y \not\in \tD_0$. It cannot use it in the direction 
      $(u'_0,v'_0)$, else $v'_0 \in \tC(\tB_1)$. It also cannot use it in the direction
      $(v'_0,u'_0)$, because then $v'_0 \in \tD_0$. The claim follows.
        
    \end{itemize}
    
\end{itemize}
\end{proof}

\begin{proof}[Proof of Corollary \ref{lem:factor-2nd-rv}]
Both cases (i)--(ii) can be proved the same way. For any $(u'_0,v'_0)$, the three events 
in the union in the right hand side of \eqref{e:factor-event} are conditionally independent, 
given the subgraph $\tD_0$ and the configuration $\omega_1$. This, and disjointness over $(u'_0,v'_0)$,
imply that we can factor the expectation and sum over $(u'_0,v'_0)$ as stated.
\end{proof}

The following lemma deals with the easier case of $N = 0$. We need a couple of definitions.
\eqnsplst
{ P^{',0}_{x'}
  &= P^{',0}_{x'}(\{u,v\};\omega_0) 
  := \left\{ (u',v') : \parbox{7.5cm}{$\{u',v'\} \not\in E_W$ and any $(\omega_0)_{\{u,v\}}$-occupied 
     path from $W$ to $x'$ uses $(u',v')$} \right\}. } 
When there is no $(\omega_0)_{\{u,v\}}$-occupied path from $W$ to $x'$, we define
$P^{',0}_{x'} = \es$. 
We will use
\eqnsplst
{ \tD_0 
  = \tD_0(F;\{u,v\},\{u'_0,v'_0\},\omega_0)
  := \left\{ y \in \Zd : 
     \text{$W \cnctd_{(\omega_0)_{\{ u, v\} \cup \{u'_0,v'_0\}}} y$} \right\}. } 
Let $E^{'',0}(F;u,v,x')$ denote the following event measurable with respect $\omega_0$:
\begin{equation}\label{def_epp}
 E^{'',0}(F;u,v,x')
  := \widetilde{E}(F;u,v,x') \cap \{ P^{',0}_{x'} = \es \}. 
  \end{equation}

\begin{lemma}
\label{lem:factor-2nd-N=0}
\emph{($N = 0$ case)} For any $\{u,v\} \not\in E_W$ we have  
\eqnspl{e:factor-rv-N=0}
{ \E_0 \Big[ I \left[ \widetilde{E}(F;u,v,x') \right] \Big] 
  &= \E_0 \Big[ I \left[ E^{'',0}(F;u,v,x') \right] \Big] + \!\!\!\!\!\sum_{\{u'_0,v'_0\} \not\in E_W \cup \{u,v\}} \!\!\!\!\! p D(u'_0,v'_0) \\
  &\qquad\qquad  \E_0 \Big[ I \left[ E^{'',0}(F;u,v,u'_0) \right] 
       \tau^{\tD_0}(v'_0,x') \Big]. }
\end{lemma}

\begin{proof}
The proof is analogous to, but simpler than, the proof of Lemmas \ref{lem:factor-2nd} and \ref{lem:factor-2nd-rv}.
\end{proof}

With the expressions \eqref{e:factor-rv}--\eqref{e:factor-rv-N=1} and \eqref{e:factor-rv-N=0} at hand, 
we can apply inclusion-exclusion expansion to the terms
$\tau^{\tD_0 \cup \tC(\tB_1)}(v'_0,x')$ and $\tau^{\tD_0}(v'_0,x')$, writing them as
\eqnspl{eqTauDCBexpansion}
{ \tau^{\tD_0 \cup \tC(\tB_1)}(v'_0,x')
  = \tau(v'_0,x') - \P'_1 \left[ v'_0 \stackrel{\tD_0 \cup \tC(\tB_1)}{\cnctd} x' \right], }
and 
\eqnspl{eqTauDexpansion}
{ \tau^{\tD_0}(v'_0,x')
  = \tau(v'_0,x') - \P'_1 \left[ v'_0 \stackrel{\tD_0}{\cnctd} x' \right], }
respectively, 
where the prime on $\P'_1$ indicates that the expectation is over a new independent percolation 
configuration $\omega'_1$. This probability in turn can be decomposed, using~\eqref{eq:expansion} and Lemma~\ref{lem:on-in}, as 
\eqnspl{eqPp1prob}
{ \P'_1 \left[ v'_0 \stackrel{\tD_0 \cup \tC(\tB_1)}{\cnctd} x' \right]
  &= \E'_1 \Big[ I \big[ E'(v'_0,x';\tD_0 \cup \tC(\tB_1)) \big] \Big] \\
  &\qquad  + \sum_{(u'_1,v'_1)} p D(u'_1,v'_1) 
  \E'_1 \Big[ I \big[ E'(v'_0,u'_1; \tD_0 \cup \tC(\tB_1)) \big] \tau^{\tC'_1}(v'_1,x') \Big], }
where 
\eqnsplst
{ \tC'_1
  = \left\{ y \in \Zd : v'_0 \cnctd_{(\omega'_1)_{(u'_1,v'_1)}} y \right\}. }
Continuing with the expansion of $\tau^{\tC'_1}(v'_1,x')$ yields Lemma \ref{lem:tPsiN-expanded} below. 
For its statement, let us introduce some abbreviations.

For $N, N' \ge 2$, let
\eqnspl{eq:PsiNN'}{
  \Psi^{N,N'}(F;u,u')
  &= \!\!\!\!\!\sum_{\substack{\{u_0,v_0\} \not\in E_W \\ \{u'_0,v'_0\} \not\in E_W \cup \{u_0,v_0\}}} 
     \sum_{\substack{\{u_1,v_1\} \\ \{u'_1,v'_1\}}} 
    p D(u_0,v_0) p D(u'_0,v'_0) p D(u_1,v_1) p D(u'_1,v'_1) \\
  &\qquad \E_0 \times \E_1 \times \E'_1 \Big[ I \left[ \hat{E}''(F;u_0,v_0,u_1,u'_0,v'_0) \right] 
       I \left[ E'(v'_0,u'_1;\tD_0 \cup \tC(\tB_1)) \right](\omega'_1) \\
  &\qquad\qquad \Xi^{(N)}(v_1,u,\tC_1) \Xi^{(N')}(v'_1,u',\tC'_1) \Big]. }
For $N \ge 2$, $N' = 1$, let
\eqnsplst{
  \Psi^{N,1}(F;u,u')
  &= \!\!\!\!\!\sum_{\substack{\{u_0,v_0\} \not\in E_W \\ \{u'_0,v'_0\} \not\in E_W \cup \{u_0,v_0\}}}
     \sum_{\{u_1,v_1\}} 
    p D(u_0,v_0) p D(u'_0,v'_0) p D(u_1,v_1) \\
  &\qquad \E_0 \times \E_1 \times \E'_1 \Big[ I \left[ \hat{E}''(F;u_0,v_0,u_1,u'_0,v'_0) \right] 
       I \left[ E'(v'_0,u';\tD_0 \cup \tC(\tB_1)) \right](\omega'_1) \\
  &\qquad\qquad \Xi^{(N)}(v_1,u,\tC_1) \Big]. }
For $N \ge 2$, $N' = 0$, let
\eqnsplst{
  \hat\Psi^{N,0}(F;u,u',v')
  &= \!\!\!\!\!\sum_{\substack{\{u_0,v_0\} \not\in E_W \\ \{u'_0,v'_0\} \not\in E_W \cup \{u_0,v_0\}}}  
     \sum_{\{u_1,v_1\}} 
    p D(u_0,v_0) p D(u'_0,v'_0) p D(u_1,v_1) \\
  &\qquad \E_0 \times \E_1 \Big[ I \left[ \hat{E}''(F;u_0,v_0,u_1,u',v') \right] 
     \Xi^{(N)}(v_1,u,\tC_1) \Big] }
{\col      and 
\eqnsplst{
  \Psi^{N,0}(F;u,x')
  &= \!\!\!\!\!\sum_{\substack{\{u_0,v_0\} \not\in E_W \\ \{u'_0,v'_0\} \not\in E_W \cup \{u_0,v_0\}}}  
     \sum_{\{u_1,v_1\}} 
    p D(u_0,v_0) p D(u'_0,v'_0) p D(u_1,v_1) \\
  &\qquad \E_0 \times \E_1 \Big[ I \left[ {E}''(F;u_0,v_0,u_1,x') \right] 
     \Xi^{(N)}(v_1,u,\tC_1) \Big]. }     
     }
Moreover, for $N = 1$, $N' \ge 2$, let
\eqnsplst{
  \Psi^{1,N'}(F;u,u')
  &= \!\!\!\!\!\sum_{\substack{\{u_0,v_0\} \not\in E_W \\ \{u'_0,v'_0\} \not\in E_W \cup \{u_0,v_0\}}}
     \sum_{\{u'_1,v'_1\}}
    p D(u_0,v_0) p D(u'_0,v'_0) p D(u'_1,v'_1) \\
  &\qquad \E_0 \times \E_1 \times \E'_1 \Big[ I \left[ \hat{E}''(F;u_0,v_0,u,u'_0,v'_0) \right]
       I \left[ E'(v'_0,u'_1;\tD_0 \cup \tC(\tB_1)) \right](\omega'_1) \\
  &\qquad\qquad \Xi^{(N')}(v'_1,u',\tC'_1) \Big]. }
For $N = 1$, $N' = 1$, let
\eqnsplst{
  \Psi^{1,1}(F;u,u')
  &= \!\!\!\!\!\sum_{\substack{\{u_0,v_0\} \not\in E_W \\ \{u'_0,v'_0\} \not\in E_W \cup \{u_0,v_0\}}} \!\!\!\!\! 
    p D(u_0,v_0) p D(u'_0,v'_0) \\
  &\qquad \E_0 \times \E_1 \times \E'_1 \Big[ I \left[ \hat{E}''(F;u_0,v_0,u,u'_0,v'_0) \right]
       I \left[ E'(v'_0,u';\tD_0 \cup \tC(\tB_1) )\right](\omega'_1) \Big]. }
For $N = 1$, $N' = 0$, let
\eqnsplst{
  \hat\Psi^{1,0}(F;u,u',v')
  &= \!\!\!\!\!\sum_{\{u_0,v_0\} \not\in E_W} \!\!\!\!\! 
    p D(u_0,v_0) \E_0 \times \E_1 \Big[ I \left[ \hat{E}''(F;u_0,v_0,u,u',v') \right] \Big] }
{\col and
\eqnsplst{
  \Psi^{1,0}(F;u,x')
  &= \!\!\!\!\!\sum_{\{u_0,v_0\} \not\in E_W} \!\!\!\!\! 
    p D(u_0,v_0) \E_0 \times \E_1 \Big[ I \left[ {E}''(F;u_0,v_0,u,x') \right] \Big]. }
}
Finally, for $N = 0$, $N' \ge 2$, let
\eqnsplst{
  \hat\Psi^{0,N'}(F;u,v,u')
  &= \!\!\!\!\!\sum_{\{u'_0,v'_0\} \not\in E_W} 
     \sum_{\{u'_1,v'_1\}}
     p D(u'_0,v'_0) p D(u'_1,v'_1) \\
  &\qquad \E_0 \times \E'_1 \Big[ I \left[ E^{'',0}(F;u,v,u'_0) \right]
       I \left[ E'(v'_0,u'_1;\tD_0) \right](\omega'_1)
       \Xi^{(N')}(v'_1,u',\tC'_1) \Big]. }
For $N = 0$, $N' = 1$, let
\eqnsplst{
  \hat\Psi^{0,1}(F;u,v,u')
  &= \!\!\!\!\!\sum_{\{u'_0,v'_0\} \not\in E_W \cup \{u,v\}} \!\!\!\!\! 
    p D(u'_0,v'_0) \\
  &\qquad \E_0 \times \E'_1 \Big[ I \left[ E^{'',0}(F;u,v,u'_0) \right]
       I \left[ E'(v'_0,u';\tD_0) \right](\omega'_1) \Big]. }
For $N = 0$, $N' = 0$, let
\eqnsplst{
  \hat\Psi^{0,0}(F;u,v,u')
  &= \E_0 \Big[ I \left[ E^{'',0}(F;u,v,u') \right] \Big]. }   

We finally define the error term $R^{N,N'}(F;x,x')$. As in \eqref{e:tPsiN-short}--\eqref{e:RN-short}, we define $R^{N,N'}(F;x,x')$ for $N'\ge2$ as $\Psi^{N,N'}(F;u,u')$ in \eqref{eq:PsiNN'} but with $\Xi^{(N')}(v'_1,u',\tC'_1)$ replaced by $\Chi^{(N')}(v'_1,u',\tC'_1)$ and an additional factor $(-1)^{N+1}$ (a factor that we shall neglect later on). 
\begin{lemma}
\label{lem:tPsiN-expanded} 
For $N \ge 1$, $n' \ge 2$ and $u, x' \in \Zd$ we have
\eqnspl{e:tPsiN-expanded}
{ \tPsi^{(N)}(F;u,x')
  = & \sum_{u',v'} \left[ \hat\Psi^{N,0}(F;u,u',v')+\sum_{N'=1}^{n'} (-1)^{N'} \Psi^{N,N'}(F;u,u') \right] 
    p D(u',v') \tau(v',x') \\
    & + \sum_{N'=0}^{n'} (-1)^{N'} \Psi^{N,N'}(F;u,x') + R^{N,n'}(F;u,x'). }
Furthermore, for $N = 0$, $n' \ge 2$ and $u,v, x' \in \Zd$ we have
\eqnspl{e:tHatPsiN-expanded}
{ \hat\tPsi^{(0)}(F;u,v,x')
  & = \sum_{u',v'} \left[ \sum_{N'=0}^{n'} (-1)^{N'} \hat\Psi^{0,N'}(F;u,v,u') \right] p D(u',v') \tau(v',x') \\
    & + \sum_{N'=0}^{n'} (-1)^{N'} \hat\Psi^{0,N'}(F;u,v,x') + R^{0,n'}(F;u,x'). }
\end{lemma}
{\col 
\proof
For $N\ge2$, we start with \eqref{e:tPsiN-short} and substitute the joint expected value $\E_0 \times \E_1 \big[\cdots\big]$ according to \eqref{e:factor-rv}. 
For the factor $\tau^{\tD_0 \cup \tC(\tB_1)}(v'_0,x')$ on the right hand side of \eqref{e:factor-rv} we then apply \eqref{eqTauDCBexpansion} and \eqref{eqPp1prob}. With this replacement we are now in business for a second Hara-Slade expansion as in Lemma \ref{prop:1st-expansion}, which readily gives the expansion of $\tPsi^{(N)}(F;u,x')$ and identifies $R^{N,N'}$. 
This yields \eqref{e:tPsiN-expanded} for $N\ge2$. 
Likewise for $N=1$, where we use instead \eqref{e:factor-rv-N=1} rather then \eqref{e:factor-rv}, and plug it into \eqref{eqDefPsi1}. 
This establishes \eqref{e:tPsiN-expanded}.

In the identity \eqref{e:tHatPsiN-expanded} we have $v$ as a separate argument. This is inherited from  \eqref{e:PsiZero}, where we have an extra condition that the bond $(u_0,v_0)$ must be vacant for the connection from $W$ to $x'$, see Lemma \ref{lem:1st-pivotal}. 
The starting point is \eqref{e:PsiZero}, and we expand using Lemma \ref{lem:factor-2nd-N=0}. 
A notable difference with the previous terms is that now we are expanding $ \tau^{\tD_0}(v'_0,x')$, whereas previously we had $\tau^{\tD_0 \cup \tC(\tB_1)}(v'_0,x')$. We therefore get the indicator $I \big[ E'(v'_0,u';\tD_0)\big]$ rather than $I \big[ E'(v'_0,u';\tD_0 \cup \tC(\tB_1) )\big]$ for $N\ge1$. 
The rest of the expansion, using \eqref{eqTauDexpansion}, is the same as in the previous case. 
\qed
}

\bigskip

For a better organisation of the resulting sum, we write the partial sum for 
$n\ge1$: 
{\col 
\eqnspl{eq:DefBarPiNN}{ 
	\bar\Psi_{n,n'}(F;u,v,u',v'):=&\sum_{N'=0}^{n'} (-1)^{N'} \hat\Psi^{0,N'}(F;u,v,u') 
				+ \sum_{N=1}^n (-1)^{N}\hat\Psi^{N,0}(F;u,u',v')\\
	&+\sum_{N=1}^n\sum_{N'=1}^{n'} (-1)^{N+N'} \Psi^{N,N'}(F;u,u'). }
Mind that the quantities have three arguments when $N=0$ or $N'=0$, these are accommodated in the first line. 
Similarly, we have 
\eqnspl{eq:DefHatPiNN}{ 
	\hat\Psi_{n,n'}(F;u,v,x')
	:=&\sum_{N'=0}^{n'} (-1)^{N'}\hat\Psi^{0,N'}(F;u,v,x') 
	+\sum_{N=1}^n\sum_{N'=0}^{n'} (-1)^{N+N'} \Psi^{N,N'}(F;u,x'). 
	}
The difference between $\bar\Psi_{n,n'}$ and $\hat\Psi_{n,n'}$ is in the terms $\hat\Psi^{N,0}$ (which differ when ending at $x'$ or at an intermediate pivotal edge $\{u',v'\}$). } 
We further let 
\eqn{eq:tPsiDef}{
	\tPsi_{n}(F;x,x') := \sum_{N=0}^n (-1)^N \tPsi^{(N)}(F;x,x').
}
Finally, we also collect all error terms as 
\eqn{eq:DefRNN}{ R_{n,n'}(F;x,x'):=\tilde R_n(F;x,x') + \sum_{N=0}^n (-1)^{N}
 	 \sum_{u,v}R^{N,n'}(F;u,x') pD(u,v)\tau(v,x) . }
These are the quantities that appear in Proposition \ref{prop:laceexpansion}, and we now present its proof.

\proof[Proof of Proposition \ref{prop:laceexpansion}]
We start with the identity in \eqref{e:N-th-stage}, and replace the quantities $\hat{\tPsi}^{(0)}$ and $\tPsi^{(m)}$ according to Lemma \ref{lem:tPsiN-expanded}. 
This yields for $n,n'\ge2$,
\eqnspl{e:N-th-stage-inserted}
{ &\hskip-1em \sigma(F; x,x')\\
  =& \sum_{N=0}^n (-1)^N \tPsi^{(N)}(F;x,x') \\
  &{} + \sum_{u,v} p D(u,v) \Biggl( 
  \sum_{u',v'}  \sum_{N'=0}^{n'} (-1)^{N'} \hat\Psi^{0,N'}(F;u,v,u') p D(u',v') \tau(v',x') \\
    &\qquad\hskip5em + \sum_{N'=0}^{n'} (-1)^{N'} \hat\Psi^{0,N'}(F;u,v,x') + R^{0,n'}(F;u,x')\\
  & \qquad+ \sum_{N=1}^n (-1)^N 
  	\biggl[\sum_{u',v'}  \Big( \Psi^{N,0}(F;u,u',v') + \sum_{N'=1}^{n'} (-1)^{N'} \Psi^{N,N'}(F;u,u') \Big)
    p D(u',v') \tau(v',x') \\
    & \qquad\hskip5em+ \sum_{N'=0}^{n'} (-1)^{N'} \Psi^{N,N'}(F;u,x') + R^{N,n'}(F;u,x')\biggr] \Biggr) \tau(v,x) \\
  & + \tR_n(F;x,x'). }
The resulting expression is then organised using \eqref{eq:DefHatPiNN}--\eqref{eq:DefRNN}. 
\qed

\bigskip
\section{Convergence of the limit}\label{sec:convergence}
\subsection{Bounds on the expansion coefficients}\label{sec:bounds}
The main result of this section is a bound on the lace expansion quantities $\Psi$, $\hat\Psi$, and the (asymptotically vanishing) remainder term $R$: 
\begin{proposition}\label{prop:PsiBd}
If \eqref{eq:tauAsy}, \eqref{eq:DtauBd} and \eqref{eq:Aass} are satisfied, then for any $F\in\mathfrak F_{00}$ there exits $C>0$ such that for all $x,x'\in\Zd$ and $p\le p_c$, 
\begin{equation}\label{eq:momentbound}
	\sum_{N \ge 0} \tPsi^{(N)}(F;x,x')
	   \le C\tau(x)\,\tau(x') \left( \min\{|x|,|x'|,|x-x'|\}\right)^{-(d-4)} 
\end{equation}
\begin{equation}\label{eq:momentbound2}
\sup_{v'}\sum_{N\ge1}\left(
\hat\Psi^{(N,0)}(F;x,x',v') + \sum_{N'\ge1}\Psi^{(N,N')}(F;x,x')\right)\le C 
	\tau(x)\,\tau(x')\,\tau(x-x').
\end{equation} 
\begin{equation}\label{eq:momentbound3}
\sup_v\sum_{N'\ge0}\hat\Psi^{(0,N')}(F;x,v,x') 
	\le C 
	\tau(x)\,\tau(x')\,\tau(x-x').
\end{equation} 
Finally, for all $x,x'\in\Zd$, 
\begin{equation}\label{eq:RNN}
\lim_{N,N'\to\infty}R_{N,N'}(F;x,x')=0.
\end{equation}
\end{proposition}
\noindent\emph{Remark.} Even though the bounds hold uniformly in $p\le p_c$, we only need it at the critical point $p=p_c$. 

\medskip
The relevance of these bounds comes from the fact that the right hand side of \eqref{eq:momentbound} 
{\colAJ{is of smaller order than the probability of the conditioning in the main theorem}}, 
{\colAJ{whereas}} the right hand side of \eqref{eq:momentbound2} is the well-known \emph{triangle diagram}, 
which is 
{\colAJ{summable}} whenever $d>6$ (see Lemma \ref{lem:ConvBd} below). 
In particular, the sums in \eqref{eq:momentbound} and \eqref{eq:momentbound3} as well as the double sum in \eqref{eq:momentbound2} are finite if  \eqref{eq:tauAsy} is satisfied, and we can thus take the double limit $n,n'\to\infty$ in Proposition \ref{prop:laceexpansion} to obtain the expansion identity 
\eqnspl{eq:OZ}{
	&\sigma(F;x,x')=  \tPsi(F;x,x')
	+\sum_{u,v,u',v'}\bar\Psi(F;u,v,u',v')\,pD(u,v)\tau(v,x)\,pD(u',v')\tau(v',x') \\
	& \hskip8em
	+ \sum_{u,v}\hat\Psi(F;u,v,x')pD(u,v)\tau(v,x), 
}
for any $x,x'\in\Zd$ and special cylinder event $F$, where $\tPsi(F;x,x')=\lim_{n\to\infty}\tPsi_{n}(F;x,x')$, 
\[  \bar\Psi(F;u,v,u',v;)=\lim_{n,n'\to\infty}\bar\Psi_{n,n'}(F;u,v,u',v'),  \text{ and }
	\hat\Psi(F;u,v,u')=\lim_{n,n'\to\infty}\hat\Psi_{n,n'}(F;u,v,u'). \]
We prove this Proposition \ref{prop:PsiBd} in Section \ref{sec:bounding-diagrams}. 
Indeed, \eqref{eq:momentbound} is proven in Lemma \ref{lem:tPsiBd0} (the $N=0$ term) and Lemma \ref{lem:tPsiBd} (the remaining sum). Further, \eqref{eq:momentbound2} is proven in Proposition \ref{prop:n=1error} (the $N=1$ contribution), and in Proposition \ref{prop:n>=2error} (the remaining double sum). 
The bound \eqref{eq:momentbound3} is proven in Proposition \ref{prop:n=0error}.  
The bound in \eqref{eq:RNN} follows as in \cite[(4.21)--(4.22)]{HofstJarai04}. 
In the remainder of this section, we show how Proposition \ref{prop:PsiBd} implies our main theorem.

\subsection{Convergence of conditional probabilities}\label{sec:Conv}
Throughout this section, we abbreviate \[\tau(x):=\tau(0,x)=\tau(y,y+x),\qquad x,y\in\Zd.\] 

The diagrams bounding the lace expansion coefficients contain numerous convolution terms, it is hence crucial to have a sharp estimates on such convolutions. These are readily provided in the literature: 
\begin{lemma}[{Convolution bounds {\cite[Prop.~1.7]{HaraHofstSlade03}}}]\label{lem:ConvBd}
Suppose $a\ge b>0$ and $f,g\colon \mathbb Z^d\to\mathbb R$ satisfy $f(x)\le c_1|x|^{-a}$ and $g(x)\le c_2|x|^{-b}$ for constants $c_1,c_2>0$. 
Then there exists $c_3>0$ such that 
\begin{equation}
	(f\ast g)(x)\le c_3
	\begin{cases}
		|x|^{-b}\quad&\text{ if }a>d,\\
		|x|^{d-(a+b)}\quad&\text{ if }a<d, a+b>d.
	\end{cases}
\end{equation}
\end{lemma}
As an example, the lemma implies (assuming validity of \eqref{eq:tauAsy}) that 
\begin{equation}\label{eq:TauTauBd}
	(\tau\ast\tau)(x)=\sum_{y\in\Zd}\tau(y)\tau(x-y)\le C\,|x|^{d-4}\quad\text{whenever $d>6$.}
\end{equation}
We further observe that the asymptotics of $\tau$ in \eqref{eq:tauAsy} implies that $x\mapsto\frac{\tau(x-v)}{\tau(x)}$ is uniformly bounded for any $v\in\Zd$, and 
\eqn{e:tautaulimit}{\lim_{|x|\to\infty}\frac{\tau(x-v)}{\tau(x)}=1}
for any $v\in\Zd$. 
This relation is crucial in the following representation of the two-arm event, which uses the earlier identified lace expansion coefficients.

\begin{lemma}\label{lem:sigmaBd}
Assume that \eqref{eq:tauAsy} and \eqref{eq:Aass} are valid and $d>6$, then 
\[\lim_{|x|,|x'|,|x-x'|\to\infty}\frac{\sigma(F;x,x')}{\tau(x)\,\tau(x')}
	=\sum_{u,v,u',v'}\hat\Psi(F;u,v,u')p D(u,v)p D(u',v').\]
\end{lemma}
It is a consequence of Proposition \ref{prop:PsiBd} and Lemma \ref{lem:ConvBd} that the above limit is \emph{finite} for $p\le p_c$. 
\begin{proof}
The proof crucially uses the upper bound in  Proposition \ref{prop:PsiBd}. 
Our starting point is then \eqref{eq:OZ}, which we recall as 
\eqnspl{newExpan}{
  &\sigma(F;x,x')
  =  \tPsi(F;x,x') + \sum_{u,v}\hat\Psi(F;u,v,x')pD(u,v)\tau(v,x)\\
	& \hskip6.5em
	+ \sum_{v,v',u,u'}\hat\Psi(F;u,v,u') p D(u,v)p D(u',v') \tau(x-v)\tau(x'-v'). 
}
The bound \eqref{eq:momentbound} in Proposition \ref{prop:PsiBd} in conjunction with \eqref{e:tautaulimit} implies that 
\[ \lim_{|x|,|x'|\to\infty}\frac{\tPsi(F;x,x')}{\tau(x)\,\tau(x')}
   \le C \left( \min\{|x|,|x'|,|x-x'|\}\right)^{-(d-4)} \longrightarrow 0 
   \qquad\text{as $|x|, |x'|, |x-x'|\to\infty$}. \]
We next argue that also the contribution from the second {\colAJ{term}} of \eqref{newExpan} vanishes in the limit. 
For the second summand in \eqref{newExpan} 
{\colAJ{we first bound $\hat{\Psi}(u,v,x') \le C \tau(u) \tau(x') \tau(u-x')$, and then use \eqref{eq:DtauBd}. 
We can then bound 
\begin{align}
   \sum_{u,v}\hat\Psi(F;u,v,x')pD(u,v)\tau(v,x)
   &\le \tau(x') \sum_u \tau(u) \tau(u,x) \tau(u,x')\nonumber \\
   &\le \tau(x) \tau(x') C \left[ |x-x'|^{4-d} + |x'|^{4-d} \right]. 
\end{align} 
After dividing by $\tau(x) \tau(x')$, the resulting expression vanishes as 
$|x|, |x'|, |x-x'| \to \infty$.}}

We are left to deal with the third summand on the right hand side of \eqref{newExpan}. 
In order to simplify notation, we abbreviate $\Phi(F;v,v') := \sum_{u,u'}\Psi(F;u,u') p D(u,v)p D(u',v')$, and have to show that  
\eqn{eq:newExpan2}{
	\sum_{v,v',u,u'}\hat\Psi(F;u,v,u') p D(u,v)p D(u',v') \frac{\tau(x-v)\tau(x'-v')}{\tau(x)\tau(x')}
	= \sum_{v,v'}\Phi(F;v,v') \frac{\tau(x-v)\tau(x'-v')}{\tau(x)\tau(x')}
}
yields the limit in Lemma \ref{lem:sigmaBd} as $|x|,|x'|,|x-x'|\to\infty$. 
We proceed by splitting the sum over $v,v'\in\mathbb Z^d$ in six terms (depending on the location of $x,x'$): 
\begin{enumerate}
	\item $|v|\le|x|^{1/2}$ and $|v'|\le|x'|^{1/2}$,   \qquad(``tiny $v$, tiny $v'$'')
	\item $|x|^{1/2}<|v|\le|x|/2$ and $|v'|\le|x'|/2$,   \qquad(``average $v$, tiny or average $v'$'')
	\item[(2')]  $|x'|^{1/2}<|v'|\le|x'|/2$ and $|v|\le|x|/2$,   \qquad(``tiny or average $v$, average $v'$'')
	\item $|x|/2<|v|$ and $|v'|\le|x'|/2$,   \qquad(``large $v$, tiny or average $v'$'')
	\item[(3')] $|x'|/2<|v'|$ and $|v|\le|x|/2$,   \qquad(``tiny or average $v$, large $v'$'')
	\item $|v| > |x|/2$, $|v'| > |x'|/2$.   \qquad(``large $v$, large $v'$'')
\end{enumerate}
We now show that only the first contribution is solely responsible for the expression in the limit, all other contributions are asymptotically negligible. 
For the first case, observe that \eqref{eq:tauAsy} implies to following strengthening of \eqref{e:tautaulimit}: 
\eqn{eq:taucomparison}{
\lim_{|x|\to \infty}\sup_{v:|v|\le |x|^{1/2}}\frac{\tau(x-v)}{\tau(x)}=1. 
}
We thus get that for the contribution from (1) that 
\eqnst{\sum_{\substack{v:|v|\le|x|^{1/2}\\ v':|v'|\le|x'|^{1/2}}}\Phi(F;v,v') \frac{\tau(x-v)}{\tau(x)}\frac{\tau(x'-v')}{\tau(x')}
	\to \sum_{v,v'\in\mathbb Z^d}\Phi(F;v,v')\quad\text{ as $|x|,|x'|\to\infty$,} 
}
where summability of the limit comes from Proposition \ref{prop:PsiBd} in conjunction with \eqref{eq:tauAsy}. 

We now show that all the remaining cases (2)--(4) all vanish in the limit. To this end, we use the diagrammatic bound in \eqref{eq:momentbound} and 
\eqref{eq:tauAsy} to get 
\eqnspl{eq:PsiTauTauBd}{
	\sum_{v,v'}\Phi(F;v,v') \frac{\tau(x-v)\tau(x'-v')}{\tau(x)\,\tau(x')}
	& \le C \sum_{v,v'}\frac{\tau(v) \tau(v') \tau(v-v') \tau(x-v)\tau(x'-v')}{\tau(x)\,\tau(x')}\\
	& \le C \sum_{v,v'}\frac{|v|^{2-d}\, |v'|^{2-d}\, |v-v'|^{2-d}\, |x-v|^{2-d}\,|x'-v'|^{2-d}}{|x|^{2-d}\,|x'|^{2-d}}.
}
For case (4), we observe that $|v|\ge |x|/2$, and hence $|v|^{2-d}\le |x|^{2-d}2^{d-2}$, and the same for $v'$ resp.\ $x'$. Further, Lemma \ref{lem:ConvBd} gives 
\eqnst{
	\sum_{v,v'} |v-v'|^{2-d}\, |x-v|^{2-d}\,|x'-v'|^{2-d}
	\le C |x-x'|^{6-d},
	}
so that 
\eqnst{
	\sum_{\substack{v\colon |v| > |x|/2\\v'\colon |v'| > |x'|/2}}\Phi(F;v,v') \frac{\tau(x-v)}{\tau(x)}\frac{\tau(x'-v')}{\tau(x')}
	\le C |x-x'|^{6-d},
}
which vanishes as $|x-x'|\to\infty$ in dimension $d>6$. 

The cases (2') and (3') are symmetric to the cases (2) and (3), and we thus consider only the latter. 
For case (2) we have that $|x-v|\ge |x|-|v|\ge |x|/2$ and thus $|x-v|^{2-d}\le |x|^{2-d}2^{d-2}$ and similarly $|x'-v'|^{2-d}\le |x'|^{2-d}2^{d-2}$. 
Then the corresponding proportion of \eqref{eq:PsiTauTauBd} is bounded above (again using Lemma \ref{lem:ConvBd}) by 
\eqnst{
	C\sum_{\substack{v\colon |x|^{1/2}<|v| \le |x|/2\\v'\colon |v'| \le |x'|/2}}
	|v|^{2-d} |v-v'|^{2-d} |v'|^{2-d}
	\le C \sum_{v\colon |x|^{1/2}<|v| }
	|v|^{(4-d)+(2-d)}
	= C |x|^{\frac{6-d}2},
}
which also vanishes as $|x|\to\infty$. 

Finally case (3). Then we have $|v|^{2-d}\le |x|^{2-d}2^{d-2}$ (as in case (4)) and $|x'-v'|^{2-d}\le |x'|^{2-d}2^{d-2}$ (as in case (2)). This yields for the corresponding proportion of \eqref{eq:PsiTauTauBd} the upper bound 
\eqnst{
	C\sum_{\substack{v\colon |v| > |x|/2\\v'\colon |v'| \le |x'|/2}}
	|v'|^{2-d} |v-v'|^{2-d} |x-v'|^{2-d}
	\le C |x|^{\frac{6-d}2},
}
which also vanishes as $|x|\to\infty$. 
\end{proof}

We write the special case of $F=\Omega$ as a corollary: 
\begin{corollary}\label{corol:sigmaBd}
Assume that \eqref{eq:tauAsy} and 
\eqref{eq:Aass} are valid. Then 
\[\lim_{|x|,|x'|,|x-x'|\to\infty}\frac{\sigma(x,x')}{\tau(x)\,\tau(x')}
	=\sum_{u,v,u',v'}\hat\Psi(u,v,u')p D(u,v)p D(u',v')>0.\]
\end{corollary}
{\col\begin{proof}
The identity follows directly from Lemma \ref{lem:sigmaBd} for $F=\Omega$, and we only need to show (strict) positivity of the series. A complication is that $\hat\Psi(u,v,u')$ is representing an alternating series, and we have no control of monotonicity properties. We therefore give a direct argument for positivity via the limit representation on the left-hand side. 
The argument is standard in high-dimensional percolation literature, it combines an `ultra-violet regularization' with 'inverse BK'-type arguments. We argue as follows. 

For $n\in\N$ let $Q_n:=[-n,n]\cap\Zd$ and denote $u=(n-2)e_1$, $u'=-(n-2)e_1$ (with $e_1$ being the first unit vector). Any reconfiguration of the occupation of edges `touching' $Q_n$ (by this we mean that at least one endpoint is in $Q_n$) changes the probability at most by a finite ($n$-dependent) factor $c_n>0$. So in particular, when we have the disjoint occurrence $\big\{(0\leftrightarrow x)\circ(0\leftrightarrow x')\big\}$ then we can `rewire' the connections of edges inside $Q_n$ so that instead the event $\big\{(u\leftrightarrow x)\circ(u'\leftrightarrow x')\big\}$ occurs, hence 
\eqnsplst{
	\sigma(x,x')
	& \ge c_n\,\P\big((u\leftrightarrow x)\circ(u'\leftrightarrow x')\big)
} 
We bound the latter factor from below as 
\eqnsplst{
	\P\big((u\leftrightarrow x)\circ(u'\leftrightarrow x')\big)
	&\ge \P\big((u\leftrightarrow x)\cap(u'\leftrightarrow x')\cap(u \not\leftrightarrow u')\big)\\
	&= \E\left[ I\{u\leftrightarrow x\}\tau^{\tC(u)}(u',x')\right],
}
similar to \eqref{e:firstFactorization}. We now employ \eqref{e:tauS} followed by BK to further bound 
\eqnsplst{
	&\E\left[ I\{u\leftrightarrow x\}\tau^{\tC(u)}(u',x')\right]\\
	&= \P(x\leftrightarrow x)\,\P(x'\leftrightarrow x')\\
	&\hskip3em 	-\P\Big(\bigcup_{v,v'} \big((u\leftrightarrow v)\circ(v\leftrightarrow x)\circ(v\leftrightarrow v')\big)\text{ on $\omega_0$, }\big((u'\leftrightarrow v')\circ(v'\leftrightarrow x')\big)\text{ on $\omega_1$}\Big)\\
	&\ge \tau(u,x)\,\tau(u',x') - \sum_{v,v'} \tau(u,v)\,\tau(v,x)\,\tau(v,v')\,\tau(u',v')\,\tau(v',x').
}
Dividing both sides by $\tau(x)\,\tau(x')$ and taking the limit as $|x|,|x'|\to\infty$ gives the lower bound 
\[\lim_{|x|,|x'|,|x-x'|\to\infty}\frac{\sigma(x,x')}{\tau(x)\,\tau(x')}
	\ge c_n\Big(1-\sum_{v,v'} \tau(u,v)\,\tau(v,v')\,\tau(u',v')\Big)
	\ge c_n\big(1-C|u-u'|^{6-d}\big)
\]
by \eqref{eq:tauAsy} and Lemma \ref{lem:ConvBd}. Since $d>6$, we can take $n$ large enough to make the term in parenthesis strictly positive. 
\end{proof}
}

We can now give the proof of our main theorem: 
\proof[Proof of Theorem \ref{thm:ExistenceBIIC}.]
For any special cylinder event $F\in\mathfrak F_{00}$, we therefore get
\eqnspl{IICconstr}{
	\Pbiic(F)
	&=\lim_{|x|,|x'|,|x-x'|\to\infty}\frac{\sigma(F;x,x')}{\sigma(x,x')}\\
	&=\lim_{|x|,|x'|,|x-x'|\to\infty}\left(\frac{\sigma(F;x,x')}{\tau(x)\,\tau(x')}\;\Big/\; \frac{\sigma(x,x')}{\tau(x)\,\tau(x')}\right)\\
	&=\frac{\sum_{u,v,u',v'}\hat\Psi(F;u,v,u')p D(u,v)p D(u',v')}
		{\sum_{u,v,u',v'}\hat\Psi(u,v,u')p D(u,v)p D(u',v')}
}
by Lemma \ref{lem:sigmaBd} and Corollary \ref{corol:sigmaBd}. Since $\mathfrak F_{00}$ is a $\cap$-stable generator of the sigma algebra $\mathfrak F=\sigma(\mathfrak F_{00})$, this therefore determines the measure $\Pbiic$ on $\mathfrak F$ uniquely. 
\qed

\bigskip

\section{A robust theoretical framework for upper-bounding lace expansion diagrams}
\label{sec:BoundsFramework}
It remains to show the upper bounds on the lace expansion coefficients as formulated in Proposition \ref{prop:PsiBd}. 
{\col In order to derive a proof of this statement, we first express summands $\Psi^{NN'}$ in terms of lace expansion diagrams. We have to deal with a large number of these diagrams, and in order to deal with them systematically, we organise the lace expansion diagrams using a new object that we call diagrammatic graphs. Then we introduce a method of reducing diagrammatic graphs to simpler ones by means of an ``H-reduction'', see Prop.\ \ref{prop_Hred}, and use graph theory to show that  \emph{every} diagrammatic graph that appears here can be reduced to only two very simple graphs, see Prop.\ \ref{super_reduc}. This enables us to readily bound various diagrams that appear in the lace expansion of bi-infinite incipient cluster (see Lemmas \ref{lem_triangle_plus}--\ref{diag13}). Finally, in Section \ref{sec:bounding-diagrams} we demonstrate that controlling these diagrams is sufficient to bound the lace expansion coefficients thereby finishing the proof of Proposition \ref{prop:PsiBd}.
}

\subsection{Diagrammatic graphs}\label{sect_diagram_graph}

Let us define an object which will be central in our analysis.
\begin{definition}[Diagrammatic graphs]
A diagrammatic graph ${\bf G}$ is a triple  $(G,S, \ell)$ where 
\begin{itemize}
\item $G$ is  a multigraph with a finite number of edges {\col without self-loops},
\item $S\subset V(G)$ is called the set of labelled vertices,
\item $\ell\colon S\to \mathbb{Z}^d$ is an injective function that encodes the labels of the vertices in $S$
\end{itemize}

If $V(G)=S$, we say that the graph is fully labelled.

We denote by $E(G)$ the set of edges of $G$ and  $E^*(G)$ the set of pairs of endpoints of the edges. For any ${\bf e}\in E^*(G)$  denote  $n_G({\bf e})$ the multiplicity of ${\bf e}$, i.e.~the number of edges in $E(G)$ whose endpoints are given by ${\bf e}$  (with the convention that $n_G({\bf e})=0$ if ${\bf e}\notin E^*(G) $). 

We denote $\mathfrak{G}$ the set of diagrammatic graphs. 
\end{definition}
\medskip

The link between diagrammatic graphs and lace expansion coefficients becomes apparent in the following definition. 
\begin{definition}\label{def_diagram_event}[Diagrammatic events]
For ${\bf G}$ a fully labelled diagrammatic graph, we introduce the diagrammatic event of $G$ as
\[
\mathcal{D}_{{\bf G}}:=
\bigcirc_{e\in E(G)} \{\ell(\underline e)\leftrightarrow \ell(\bar e)\},
\]
where 
$\bigcirc_{e\in E(G)} \{\ell(\underline e)\leftrightarrow \ell(\bar e)\} := \{\ell(\underline e_1)\leftrightarrow \ell(\bar e_1)\} \circ \cdots \circ \{\ell(\underline e_{|E(G)|})\leftrightarrow \ell(\bar e_{|E(G)|})\}$ for any enumeration $e_1=\{\underline e_1,\bar e_1\},\dots,e_{|E(G)|}=\{\underline e_{|E(G)|},\bar e_{|E(G)|}\}$ of the edges of $G$.

For a general diagrammatic graph ${\bf G}=(G,S,\ell)$ we introduce the set of admissible label functions by 
\begin{equation}\label{eq:LG}
	\mathcal L_{\bf G}=\big\{f\colon V(G)\to\Zd \text{ s.t. }f|_S=\ell\big\}
\end{equation}
and define 
\[
\mathcal{D}_{{\bf G}}:= \bigcup_{f\in\LG} \bigcirc_{\{x,y\}\in E(G)} \{f(x) \leftrightarrow f(y)\}.
\]
\end{definition}

We will define a few operations on diagrammatic graphs. First, let us explain how to add a labeled edge to a diagrammatic graph (with, if necessary, increased multiplicity).

\begin{definition}[Union of diagrammatic graphs] Let ${\bf G}=(G,S,\ell)$ be a diagrammatic graph and $x,y\in \mathbb{Z}^d$. We will denote ${\bf G}\uplus \{x,y\}$ the diagrammatic graph $(G’,S’,\ell’)$ defined as follows
\begin{enumerate}
\item if $x,y\in \{\ell(s), s\in S\}$, then 
\begin{itemize}
 \item $V(G’)=V(G)$ 
 \item $E(G’)=E(G)\cup e_{\{x,y\}}$ where an edge $e_{\{x,y\}}$ is added between the vertices $s,s’\in S$ such that $x=\ell(s)$ and $y=\ell(s’)$ 
  \item $S=S’$ and $\ell=\ell’$.
\end{itemize}
\item if $x\in \{\ell(s), s\in S\}$ and $y\notin  \{\ell(s), s\in S\}$,
\begin{itemize}
 \item $V(G’)=V(G)\cup\{y_*\}$ 
 \item $E(G’)=E(G)\cup e_{\{x,y\}}$ where an edge $e_{\{x,y\}}$ is added between $y_*$ and the vertex $s\in S$ such that $x=\ell(s)$
  \item $S’=S \cup \{y_*\}$, $\ell(s)=\ell’(s)$ if $s\in S$ and $\ell’(y_*)=y$.
\end{itemize}
\item if $x,y\notin \{\ell(s), s\in S\}$
\begin{itemize}
 \item $V(G’)=V(G)\cup\{x_*,y_*\}$ 
 \item $E(G’)=E(G)\cup e_{\{x,y\}}$ where an edge $e_{\{x,y\}}$ is added between $x_*$ and $y_*$.
  \item $S’=S \cup \{x_*,y_*\}$, $\ell(s)=\ell’(s)$ if $s\in S$, $\ell’(x_*)=x$ and $\ell’(y_*)=y$.
\end{itemize}
\end{enumerate}

We will extend this notation to ${\bf G}\uplus {\bf G}'$ when ${\bf G}$ is a diagrammatic graph and ${\bf G}'$ a fully labeled graph simply, by successively adding to ${\bf G}$ the finite number of labeled edges composing ${\bf G}'$.
\end{definition}

\begin{remark} We are working with multigraph and thus, for $e\in E(G)$, when writing $E(G) \setminus e$ we only remove the edge $e$ which means the endpoints of $e$ may still be connected by an edge (in that case, the multiplicity of the connection has been diminished by 1).

{\col This is only pertinent up until Section~\ref{sec:reduction}, as from Section~\ref{sect_graph_theory}  onward we focus on diagrammatic graphs whose graph component are simple graphs.} 
\end{remark}

The following operation will appear naturally in lace expansion (the reader may skip forward to Lemma~\ref{lem_extra_path} and its more important consequence Lemma~\ref{lem_extra_path2} to see the connection) as a way to explain the existence of an \lq\lq extra path\rq\rq~(with an endpoint summed over) in a lace expansion diagram. 
\begin{definition}[Convolution of diagrammatic graphs]
Take a diagrammatic graph ${\bf G}=(G,S,\ell) \in \mathfrak{G}$ and a vertex $u\in \mathbb{Z}^d$. If $u\not\in\ell(S)$ ($u$ is not yet a label of $\bf G$), then we define define ${\bf G} \circledast u$ as the family of diagrammatic graphs $\{(G_{u,e},S',\ell'), e\in E(G)\}$ where $G_{u,e}$ is such that
\begin{itemize}
\item $V(G_{u,e})=V(G)\cup \{u,u_*\}$.
\item $E(G_{u,e}) =(E(G) \setminus e) \cup  \{\{x,u_*\},\{y,u_*\},\{u,u_*\}\}$, where $x$ and $y$ are the endpoints of $e$. That is, we erase the edge $e$ and replace it by with 3 new edges,
\item $S'=S\cup \{u\}$,
\item $\ell'(z)=\ell(z)$ if $z\in S$ and $\ell'(u)=u$.
\end{itemize}
If $u\in\ell(S)$ then set ${\bf G}\circledast u:=\{{\bf G}\}$. 
\end{definition}
{\col
For $x,y\in\Zd$, we simplify by writing the edge $\{x,y\}$ for the graph with vertices $x$ and $y$ and one edge between them. With this notation we now exemplify the above notions.} 

\begin{example}\label{ex_graph}
For $x,x'\in\Zd$, we can draw $(\{0,x\} \circledast x')\uplus \{0,x'\}$ as
\[
        \begin{minipage}{0.4\textwidth}
      \includegraphics[width=0.4\linewidth]{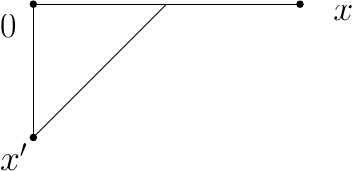}
    \end{minipage}\
    \]
where the unlabelled vertex is the vertex $u_*$ in the previous definition.

Since $\{0,x\} \circledast x'$ and $(\{0,x\} \circledast x')\uplus \{0,x'\}$ correspond to the union over $u_*\in \mathbb{Z}^d$ of the following graphs:
\[
        \begin{minipage}{0.4\textwidth}
      \includegraphics[width=0.4\linewidth]{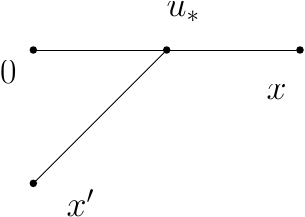}
    \end{minipage}\
\qquad 
            \begin{minipage}{0.4\textwidth}
      \includegraphics[width=0.4\linewidth]{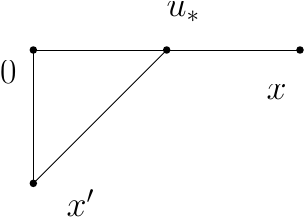}
    \end{minipage}\
\]
we may notice that  $(\{0,x\} \circledast x')\uplus \{0,x'\}=\bigcup_{u_*\in \mathbb{Z}^d} (K_{0,x',u_*} \uplus \{u_*,x\})$ {\col(with $K_{0,x',u_*}$ being the complete graph with vertices $0,x',u_*$)}. 
\end{example}

We will also need the following operation
\begin{definition}[Double-convolution of diagrammatic graphs]
Take a graph ${\bf G}=(G,S,\ell) \in \mathfrak{G}$ and a vertex $u\in \mathbb{Z}^d$. We define define ${\bf G} \circledast^2 u$ as the set of diagrammatic graphs $\{(G_{e,u},S',\ell'); e\in E(G)\} \cup \{(G_{e_1,e_2,u},S,\ell); e_1,e_2\in E(G), {\colAJ{e_1 \not= e_2}}\}$ where $(G_{e,u},S',\ell)$ is the diagrammatic graph such that
\begin{itemize}
\item $V(G_{e,u})=V(G)\cup \{u,u_*,u_{**}\}$,
\item $E(G_{e,u}) =(E(G) \setminus e) \cup \{\{u,u_*\},\{u,u_{**}\},\{u_*,x\},{\colAJ{\{u_*,u_{**}\}}},\{u_{**},y\}\}$, where $x,y$ are the endpoints of $e$,
\item $S'=S\cup \{u\}$,
\item $\ell'(z)=\ell(z)$ if $z\in S$ and $\ell'(u)=u$,
\end{itemize}
and $(G_{e_1,e_2,u},S',\ell')$ is the diagrammatic graph such that 
\begin{itemize}
\item $V(G_{e_1,e_2,u})=V(G)\cup \{u,u_*,u_{**}\}$
\item $E(G_{e_1,e_2,u}) =(E(G) \setminus \{e_1,e_2\}) \cup \{\{u,u_*\},\{x_1,u_*\},\{u_*,y_1\},\{u,u_{**}\},\{x_2,u_{**}\},\{u_{**},y_2\}\}$, where $x_1,y_1$ (resp.~$x_2,y_2$) are the endpoints of $e_1$ (resp.~$e_2$).
\item $S'=S\cup \{u\}$,
\item $\ell'(z)=\ell(z)$ if $z\in S$ and $\ell'(u)=u$.
\end{itemize}
\end{definition}

\begin{example} 
The diagrammatic graphs $\{0,x\} \circledast^2 x'$ can be drawn as  
\[
        \begin{minipage}{0.4\textwidth}
      \includegraphics[width=0.4\linewidth]{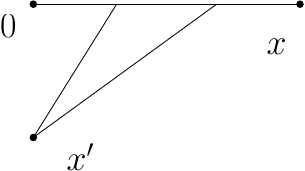}
    \end{minipage}\
    \]
and can be seen to differ from $(\{0,x\} \circledast x')\circledast x'$ (which contains graphs where $x'$ has an edge that attaches to the edge $\{x',u_*\}$ in the first graph of Example~\ref{ex_graph}).
\end{example}

\begin{definition}[Generalized diagrammatic graphs]
Any set of diagrammatic graphs obtained by a finite sequence of unions and/or convolutions ($\circledast$ or $\circledast^2$) of diagrammatic graphs is called a generalized diagrammatic graph. \end{definition}

\begin{remark} \label{gen_W} Fix $x\in \mathbb{Z}^d$ and $W\subset \mathbb{Z}^d$. In our proofs, it will be convenient to use the shorthand notation $\{x,W\}$ for the generalized diagrammatic graph with vertex set $0\cup W$ and edge set $\cup_{w\in W} \{0,w\}$. 
\end{remark}

\subsection{Generalized diagrams}\label{sect_gen_diag}
In lace expansion, the term diagram is used to describe a product of 2-point functions which is typically represented by a diagrammatic graph. That diagrammatic graph ${\bf G}$ is meant to signify 
\[
\text{Diag}({\bf G}):= \sum_{f\in\LG}  \prod_{\{x,y\}\in E(G)} \tau(f(x),f(y)),
\]
where we recall that $f$ is a label function as defined in \eqref{eq:LG}. 

\begin{definition}[Generalized diagrams] 
 For a set of diagrammatic graphs $\mathcal{G}$, we say that the \emph{generalized diagram of $\mathcal{G}$} is given by the product of 2-point functions
\[
\text{Diag}(\mathcal{G}):=\sum_{{\bf G}\in \mathcal{G}} \text{Diag}({\bf G}),
\]
which means it is the sum of diagrams appearing in $\mathcal{G}$ (which is a set of diagrammatic graphs).

In particular, if $\mathcal{G}$ is a generalized diagrammatic graph, and  $E\subset \bigcup_{{\bf G}\in \mathcal{G}} \mathcal{D}_{{\bf G}}$, then the BK inequality will imply that 
\[
{\bf P}[E]\leq {\bf P}\Bigl[ \bigcup_{{\bf G}\in \mathcal{G}} \mathcal{D}_{{\bf G}}\Bigr] \leq   \text{Diag}(\mathcal{G}).
\]
\end{definition}

Diagrams are very convenient in lace expansion, because the picture of a simple graph symbolizes a complex sum. It is not the case  that any generalized diagram can represented graphically in a compact manner, indeed in principle it is just a sum of classical diagrams. However, a simple form exists when dealing with generalized diagrams arising from generalized diagrammatic graphs and even products of such items. The latter will turn out to be useful when dealing with events occurring on different percolation configurations ($\omega_0$, $\omega_1$ ...).

\begin{definition}[Graphical visualization of certain generalized diagrams] 
The diagram of a diagrammatic graph ${\bf G}$ is simply represented by its multigraph component $G$ where the vertices in $S$ are labelled according to $\ell$.

If $\mathcal{G}$ is a generalized diagrammatic graph, then
\begin{itemize}
\item if $G'$ is a fully-labeled diagrammatic graph, the graphical representation of $\mathcal{G}\uplus G'$ is that of $\mathcal{G}$ where we take the union of the two graphs but identify vertices with the same label,
\item if $u\in \mathbb{Z}^d$, for the graphical representation of $\mathcal{G} \circledast u$, we add a label, say $A$, to each edge in the graphical representation of $\mathcal{G}$ and we add a point $u$ connected to an \lq\lq edge\rq\rq~whose other endpoint will be labelled \lq\lq$A?$\rq\rq. That last edge is called a connecting edge and can be labelled in future iterations.
\end{itemize} 

Since a generalized diagrammatic graph is obtained as an iteration of the two previous iterations, we can give a graphical visualization of generalized diagrammatic graphs. We refer to this graphical visualization as a generalized diagram.
\end{definition}
The allows us to give a graphical representation of any generalized diagrams arising from generalized diagrammatic graphs. This graphical representation is very naturally related to traditional diagrams. For example, the generalized diagram of $(\{0,x\} \circledast x')\uplus \{0,x'\}$ (pictured on the left side) is equal to a standard diagram
\[
        \begin{minipage}{0.4\textwidth}
      \includegraphics[width=0.4\linewidth]{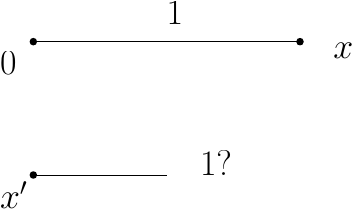}
    \end{minipage}\
\qquad = \qquad 
            \begin{minipage}{0.4\textwidth}
      \includegraphics[width=0.4\linewidth]{base_convol3-eps-converted-to.pdf}
    \end{minipage}\
\]
where the notation basically means that $1?$ will get connected to an edge labelled 1. Here there is only one labelled edge (so simplifying into a diagram makes sense), but it is not always the case, see for example $((\{0,x\} \circledast x')\uplus \{0,x'\})\circledast u_0$. This notation combines well with multiplication with other generalized diagrams or 2-point functions. For example, we can see that the generalized diagram of $((\{0,x\} \circledast z_0)\uplus \{0,z_0\})\circledast y_1$ multiplied by the generalized diagram of  $\{x',z_0\} \circledast y_1$ (drawn in bold in the first picture to emphasize its contribution) multiplied by $\tau(y_1,x')$  (drawn in dashed) verifies the following equality 
\[
        \begin{minipage}{0.4\textwidth}
      \includegraphics[width=0.4\linewidth]{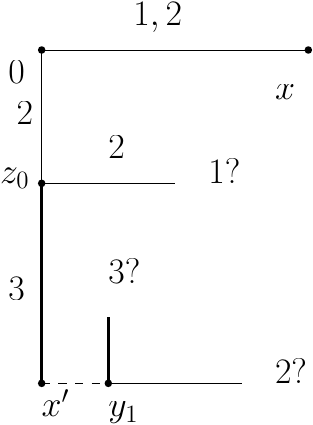}
    \end{minipage}\
 = \qquad 
            \begin{minipage}{0.4\textwidth}
      \includegraphics[width=0.4\linewidth]{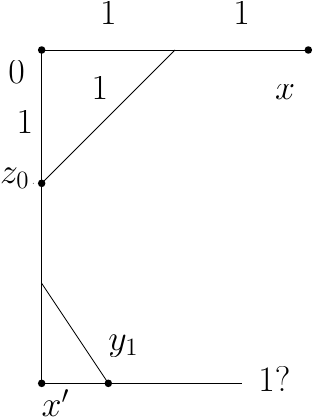}
    \end{minipage}\
\]
which is a sum of 4 standard diagrams. One should notice that for this equality to hold it is important to split the labels of the edge $\{0,x\}$ when the connection $1?$ in the first picture gets connected to it. 

\begin{remark}\label{rem:Wsum} It will be convenient to extend this graphical representation to cover the notation introduced in Remark~\ref{gen_W}. We will only present an example as the notation is fairly self-explanatory in light of everything discussed above. The generalized diagrammatic graphs $(\{0,x'\} \uplus [u,W])\uplus \{0,x\}$ and $((\{0,x'\} \uplus [u,W])\uplus \{0,x\}) \cup ((\{0,x'\} \circledast x )\uplus \{0,x\})$ will be represented as
\begin{align*}
       \begin{minipage}{0.3\textwidth}
      \includegraphics[width=0.4\linewidth]{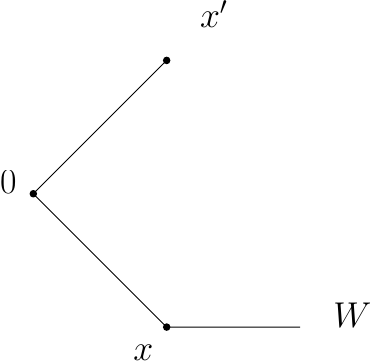}
    \end{minipage}\
 &\leq \sum_{w\in W}
            \begin{minipage}{0.3\textwidth}
      \includegraphics[width=0.4\linewidth]{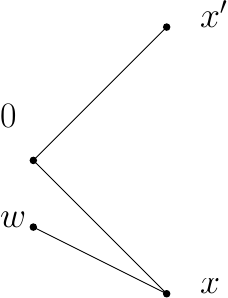}
    \end{minipage}\ \\
       \begin{minipage}{0.3\textwidth}
      \includegraphics[width=0.4\linewidth]{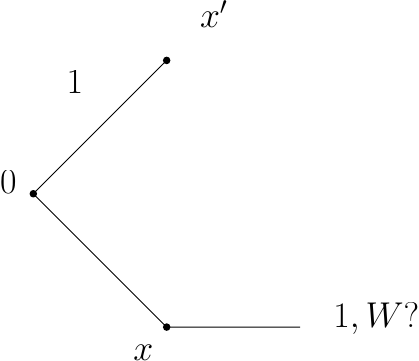}
    \end{minipage}\
 &\leq 
            \begin{minipage}{0.3\textwidth}
      \includegraphics[width=0.5\linewidth]{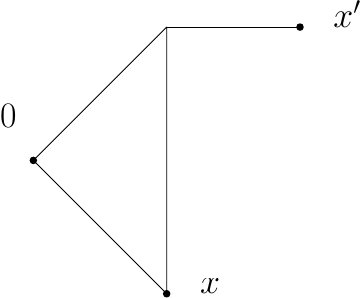}
    \end{minipage}\
    +\sum_{w\in W}
                \begin{minipage}{0.3\textwidth}
      \includegraphics[width=0.4\linewidth]{part0b-eps-converted-to.pdf}
    \end{minipage}\
\end{align*}

\end{remark}

\begin{proposition}\label{neglect_w2}
Let $W \subset \mathbb{Z}^d$ be a finite set containing $0$ and $u\in \mathbb{Z}^d$. There exists a constant $C_W$ such that for any diagrammatic graph ${\bf G}=(G,S,\ell)$ with an edge adjacent to $0$ (i.e.~$\ell(x)=0$ for some $\{x,y\} \in E(G)$), we have 
\[
\sum_{{\bf G}'\subset {\bf G}\uplus [u,W] } \operatorname{Diag}({\bf G'}) \leq C_W \sum_{{\bf G}'\subset {\bf G} \circledast u}\operatorname{Diag}({\bf G'}).
\]
\end{proposition}

\begin{proof}
We only need to prove the result for the case where ${\bf G}$ is fully labelled, as the general case can then be deduced by summing over all unlabelled points on both sides. 

For a fixed $w\in W$, we know that $\frac{\tau(x,w) }{ \tau(x,0)}$ is bounded by~\eqref{eq:tauAsy} since its limit is $1$ as $|x|\to\infty$. This means that since $W$ is a finite set, we have that 
\eqn{eq:tauXZ}{
\sup_{w\in W}  \sup_{x\in \mathbb{Z}^d} \frac{\tau(x,w) }{ \tau(x,0)} =C_W<\infty.
}
We may hence bound for any diagrammatic graph ${\bf G}=(G,S,\ell)$,
\begin{align}\label{truck}
\sum_{{\bf G}'\subset {\bf G}\uplus [u,W]}  \text{Diag}({\bf G}')& =\sum_{w\in W} \tau(u,w)  \sum_{f\in\LG}  \prod_{\{x,y\}\in E(G)}  \tau(f(x),f(y)) \\ \nonumber
  & \leq C_W |W| \tau(u,0)  \sum_{f\in\LG}  \prod_{\{x,y\}\in E(G)}  \tau(f(x),f(y)).
\end{align}
Summing over all the potentiel target edges $\{x_*,y_*\}$ for the connecting edge in $ {\bf G} \circledast u$ we have
\begin{align*}
 \sum_{{\bf G}'\subset {\bf G} \circledast u}  \text{Diag}({\bf G}')
=& \sum_{\{x_*,y_*\} \in E(G)}\sum_{f\in\LG} \sum_{u_*\in \mathbb{Z}^d} \tau(u,u_*) \,\tau(f(x_*),u_*) \,\tau(f(y_*),u_*) \\
& \qquad \qquad \qquad \times    \prod_{\{x,y\}\in E(G)\setminus \{x_*,y_*\}  } \tau(f(x),f(y)),
\end{align*}
and in this double sum, the term coming from an edge $\{a,b\}\in E(G)$ with $f(x)=0$   (which exists by hypothesis) and $u_*=0$ is 
\begin{align*}
& \tau(u,0)\tau(0,0)\tau(f(a),f(b))\prod_{\{x,y\}\in E(G)\setminus \{a,b\}  } \tau(f(x),f(y))\\
=&\tau(u,0)  \sum_{f\in\LG}  \prod_{\{x,y\}\in E(G)}  \tau(f(x),f(y)),
\end{align*}
since $\tau(0,0)=1$. Hence, that term alone is enough to bound the product in~\eqref{truck}. This yields the result.
\end{proof}

\subsection{A method to reduce diagrams and generalized diagrams}\label{sec:reduction}
Let us prove a key tool for us to upper-bound diagrams called the H-reduction.

\begin{definition}[H-reduced graph]
Let $G$ be a multigraph, we say that an edge $\{x,y\}$ is \emph{H-reducible} if $x$ and $y$ both have an edge degree at least 3 in $G$ and $x$, $y$ are only connected by one edge. 

We define a multigraph $\tilde{G}$, called \emph{H-reduction of $G$ by $\{x,y\}$}, in the following manner:
\begin{itemize}
\item if $x$ and $y$ have degree 3, then $V(\tilde{G})=V(G)\setminus \{x,y\}$ and $E(\tilde{G})=E(G) \setminus (\{\{x,z\}; x\sim_G z\} \cup \{\{y,z\}; y\sim_G z\}) \cup \{e_{x\setminus y}, e_{y \setminus x}\}$ with $e_{x\setminus y}=\{z_1, z_2\}$ where $z_1,z_2$ are the neighbours of $x$ which are not $y$ (and a symmetrical definition of $e_{y\setminus x}$)
\item if $x$ has degree 3, $y$ has degree at least 4, then then $V(\tilde{G})=V(G)\setminus \{x\}$ and $E(\tilde{G})=E(G) \setminus (\{\{x,z\}; z\sim_G z\} ) \cup \{e_{x\setminus y}\}$.
\item if $x$ and $y$ both have degree at least 4, then  $V(\tilde{G})=V(G)$ and $E(\tilde{G})=E(G)\setminus \{x,y\}$.
\end{itemize}
\end{definition}

\begin{proposition}[H-reduction]\label{prop_Hred}
There exists an absolute constant $C<\infty$, such that for any diagrammatic graph ${\bf G}=(G,S,\ell)$ with a finite number of edges and vertices, any $x,y\in V(G)\setminus S$ such that $\{x,y\}$ is H-reducible in $G$, we have 
\[
\operatorname{Diag}((G,S,\ell)) \leq C \operatorname{Diag}((\tilde{G},S,\ell)),
\]
where   $\tilde{G}$ is the H-reduction of $G$ by $\{x,y\}$.
\end{proposition}

\begin{proof}
We proceed by cases. We first treat the case when
$x$ and $y$ both have degree at least 4:
In this case $E(\tilde{G})=E(G)\cup\{\{x,y\}\}$ and therefore
\[ \prod_{\{z,w\}\in E(G)} \tau(f(z),f(w))\leq \prod_{\{z,w\}\in E(\tilde{G})} \tau(f(z),f(w))\]
simply because $\tau(f(x),f(y))\leq 1$.

\begin{figure}[ht]
\includegraphics[height=.2\textwidth]{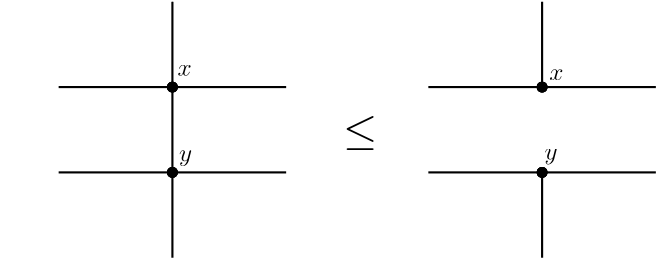}
\caption{Diagrammatic estimates for the proof of the Proposition \ref{prop_Hred} when both $x$ and $y$ have degree larger than $3$.}
\end{figure}
This immediately yields the claim of the lemma.

Now we treat the case when either $x$ or $y$ has degree exactly $3$.
Let $E^*(G)$ be the edges of $G$ which do not have any vertex of $V(G)\setminus V(\tilde{G})$ has one of its endpoints. Let $\mathcal{L}_{\bold {G\setminus\tilde{G}}}:=\{f:V(G)\setminus V(\tilde{G})\to\mathbb{Z}^d\}$. Given $f_1\in\mathcal{L}_{\tilde {\bold G}}$ and $f_2\in\mathcal{L}_{\bold {G\setminus\tilde{G}}}$, let $f_{1,2}\in\LG$, be the function that coincides with $f_1$ in $V(\tilde G)$ and coincides with $f_2$ in  $V(G)\setminus V(\tilde{G})$.
We can write
\eqnsplst{\operatorname{Diag}&((G,S,\ell))=\sum_{f\in\LG}  \prod_{\{z,w\}\in E(G)} \tau(f(z),f(w)).\\
=&
\sum_{f_1\in\mathcal{L}_{\tilde {\bold G}}} \sum_{f_2\in\mathcal{L}_{\bold {G\setminus\tilde{G}}}} \prod_{\{z,w\}\in E(G)} \tau(f_{1,2}(z),f_{1,2}(w))\\
=&
\sum_{f_1\in\mathcal{L}_{\tilde {\bold G}}} \prod_{\{z,w\}\in  E^*(G)} \tau(f_{1,2}(z),f_{1,2}(w))
\sum_{f_2\in\mathcal{L}_{\bold {G\setminus\tilde{G}}}} \prod_{\{z,w\}\in E(G)\setminus E^*(G)} \tau(f_{1,2}(z),f_{1,2}(w))
,}
where, in the last step, the factorization of the term $\prod_{\{z,w\}\in E^*(G)} \tau(f_{1,2}(z),f_{1,2}(w))$ is justified because it does not depend on $f_2$ (indeed, for all $\{z,w\}\in E^*(G)$ it holds that  $f_{1,2}(z)=f_1(z)$ and $f_{1,2}(w)=f_1(w)$).

Also, since in this case we have $E^*(G)\subseteq E(\tilde{G})$, we can write
\eqnsplst{\operatorname{Diag}((\tilde{G},S,\ell))=&
\sum_{f\in\mathcal{L}_{\tilde {\bold G}}} \prod_{\{z,w\}\in E^\ast(G)} \tau(f(z),f(w))
\prod_{\{z,w\}\in E(\tilde{G})\setminus E^*(G)} \tau(f(z),f(w)).
}
Therefore, to obtain the claim of the lemma, it suffices to show that, for all $f_1\in\mathcal{L}_{\tilde {\bold G}}$
\begin{equation}\label{eq:hredisolated}
\sum_{f_2\in\mathcal{L}_{\bold {G\setminus\tilde{G}}}} \prod_{\{z,w\}\in E(G)\setminus E^*(G)} \tau(f_{1,2}(z),f_{1,2}(w))\leq C
\prod_{\{z,w\}\in E(\tilde{G})\setminus E^*(G)} \tau(f_{1,2}(z),f_{1,2}(w)).
\end{equation}
We now treat 2 sub-cases:

Sub-case 1 ($x$ and $y$ have degree 3): Let $\{w,z\}$ be the neighbors of $x$ which are not $y$ and $\{t,s\}$ the neighbors of $y$ which are not $x$. Therefore
\eqnsplst{
&\sum_{f_2\in\mathcal{L}_{\bold {G\setminus\tilde{G}}}} \prod_{\{u,v\}\in E(G)\setminus E^*(G)} \tau(f_{1,2}(u),f_{1,2}(v))\\
&=\sum_{f_2(x),f_2(y)\in\Zd} \tau(f_1(z),f_2(x))\tau(f_2(x),f_1(w)) \tau(f_1(s),f_2(y))\tau(f_2(y),f_1(t))\tau(f_2(x),f_2(y))
}
and
\[\prod_{\{u,v\}\in E(\tilde{G})\setminus E^*(G)} \tau(f_{1,2}(u),f_{1,2}(v))=\tau(f_1(z),f_1(w))\tau(f_1(s),f_1(t)).\]

We notice that either $|f_1(z)-f_2(x)|\geq\frac{1}{2}|f_1(z)-f_1(w)|$ or $|f_2(x)-f_1(w)|\geq\frac{1}{2}|f_1(z)-f_1(w)|$. By \eqref{eq:tauAsy}, this implies that there exists $C>0$ such that 
\eqnsplst{&\tau(f_1(z),f_2(x))\tau(f_2(x),f_1(w))\\
&\leq C\left(\tau(f_1(z),f_1(w))\tau(f_2(x),f_1(w))+\tau(f_2(z),f_2(w))\tau(f_2(w),f_1(x))\right)} and thus 

\eqnsplst{ &\sum_{f_2(x),f_2(y)\in\Zd} \tau(f_1(z),f_2(x))\tau(f_2(x),f_1(w)) \tau(f_1(s),f_2(y))\tau(f_2(y),f_1(t))\tau(f_2(x),f_2(y))\\
	 &\leq C\sum_{f_2(x),f_2(y)\in\Zd}\tau(f_1(z),f_1(w))\tau(f_2(x),f_1(w)) \tau(f_1(s),f_2(y))\tau(f_2(y),f_1(t))\tau(f_2(x),f_2(y))+\\
	&\quad+C\sum_{f_2(x),f_2(y)\in\Zd} \tau(f_1(z),f_2(x))\tau(f_1(z),f_1(w)) \tau(f_1(s),f_2(y))\tau(f_2(y),f_1(t))\tau(f_2(x),f_2(y)).}
	
	\begin{figure}[ht]
	\includegraphics[height=.25\textwidth]{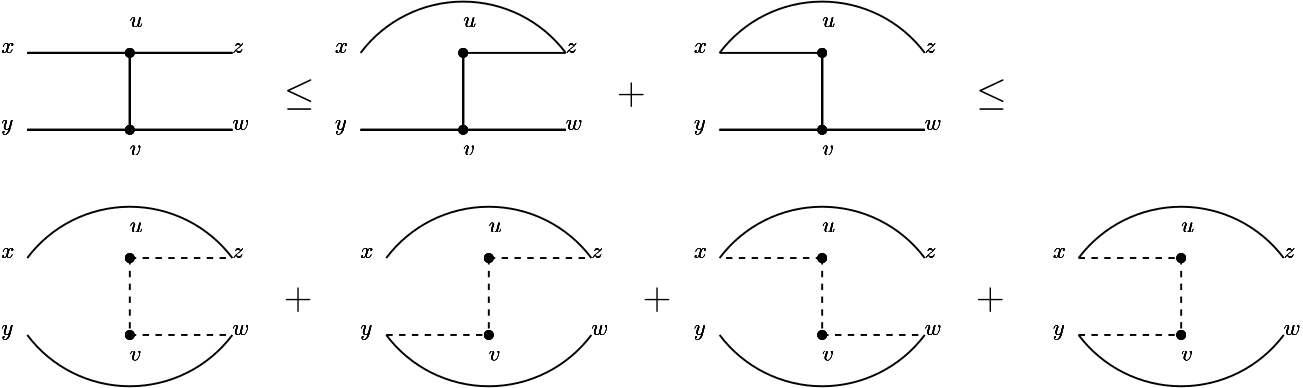}
	\caption{Diagrammatic estimates for the proof of the first case of Proposition \ref{prop_Hred}. Dotted edges contribute only a constant factor and therefore can be erased.}
	\label{fig:HT1}
\end{figure}

Repeating the strategy on each of the terms we get that the display above is bounded by
	\eqnsplst{C \tau(f_1(z),f_1(w))\tau(f_1(s),f_1(t))\Big[ &\sum_{f_2(x),f_2(y)\in\Zd}\tau(f_2(x),f_1(w)) \tau(f_2(y),f_1(t))\tau(f_2(x),f_2(y))+\\
	\quad+&\sum_{f_2(x),f_2(y)\in\Zd} \tau(f_2(x),f_1(w)) \tau(f_1(s),f_2(y))\tau(f_2(x),f_2(y))+
\\
	\quad+&\sum_{f_2(x),f_2(y)\in\Zd}\tau(f_1(z),f_2(x)) \tau(f_2(y),f_1(t))\tau(f_2(x),f_2(y))+\\
	\quad+&\sum_{f_2(x),f_2(y)\in\Zd}\tau(f_1(z),f_2(x))\tau(f_1(s),f_2(y))\tau(f_2(x),f_2(y))\Big]}

	Using that
	 \[\Delta=\sup_{f_1(u),f_1(v)\in\Zd}\sum_{f_2(x),f_2(y)\in\Zd}\tau(f_1(u),f_2(x))\tau(f_2(x),f_2(y))\tau(f_2(y),f_1(v))<\infty\]
we get that each of the 4 terms above is bounded by  $ C \tau(f_1(z),f_1(w)) \tau(f_1(s),f_1(t))$, thus obtaining \eqref{eq:hredisolated}.
A pictorial representation of these bounds is given in Figure \ref{fig:HT1}:


Sub-case 2 ($y$ has degree 3 and $x$ has degree at least 4): Let $z,w$ be the neighbors of $y$ which are not $x$. 
In this case
\eqnsplst{
&\sum_{f_2\in\mathcal{L}_{\bold {G\setminus\tilde{G}}}} \prod_{\{u,v\}\in E(G)\setminus E^*(G)} \tau(f_{1,2}(u),f_{1,2}(v))\\
&=\sum_{f_2(y)\in\Zd} \tau(f_1(z),f_2(y))\tau(f_2(y),f_1(x)) \tau(f_2(y),f_1(w))}
and
\[\prod_{\{u,v\}\in E(\tilde{G})\setminus E^*(G)} \tau(f_{1,2}(u),f_{1,2}(v))=\tau(f_1(z),f_1(w)) .\]

As in the proof of the first case, we have
\[\tau(f_1(z),f_2(y)) \tau(f_2(y),f_1(w))\\
 \leq C(\tau(f_1(z),f_1(w))\tau(f_2(y),f_1(w))+\tau(f_1(z),f_1(w))\tau(f_1(z),f_2(y))).\]
 
  Therefore
\eqnsplst{&\sum_{f_2(y)\in\Zd}  \tau(f_1(z),f_2(y))\tau(f_2(y),f_1(x)) \tau(f_2(y),f_1(w))\\
	& \leq C\tau(f_1(z),f_1(w))\Big[\sum_{f_2(y)\in\Zd} \tau(f_2(y),f_1(x))\tau(f_2(y),f_1(w))+
	\\&\quad+C\sum_{f_2(y)\in\Zd}\tau(f_2(y),f_1(x))\tau(f_1(z),f_2(y))\Big]\\
	&\leq C\tau(f_1(z),f_1(w)),
}
where in the last inequality we have used that \[\sup_{f_1(x),f_1(v)\in\mathbb{Z}^d}\sum_{f_2(y)\in\mathbb{Z}^d}\tau(f_1(x),f_2(y))\tau(f_2(y),f_2(v))<\infty.\]
A pictorial representation of these bound is given here: 
\begin{figure}[ht]
\includegraphics[height=.2\textwidth]{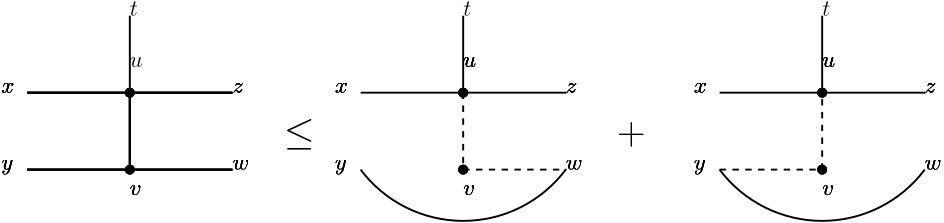}
\caption{Diagrammatic estimates for the proof of the second sub-case of Proposition \ref{prop_Hred}. Dotted edges contribute only a constant factor and therefore can be erased.}
\end{figure}

\end{proof}

\begin{remark}\label{rem_reductions}
From a diagrammatic perspective the previous lemma means that any H-reducible edge with unlabelled endpoints can be erased from a diagram at the cost of a multiplicative constant. This can be represented as follows (where the H-reducible edge has been highlighted in bold)
\[
        \begin{minipage}{0.4\textwidth}
      \includegraphics[width=0.5\linewidth]{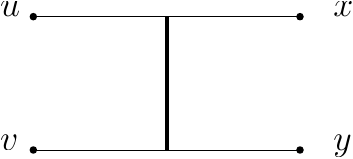}
    \end{minipage}\
    \lesssim
            \begin{minipage}{0.4\textwidth}
      \includegraphics[width=0.5\linewidth]{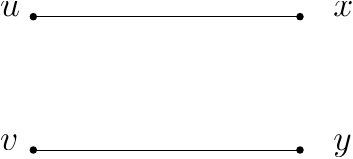}
    \end{minipage}\
\]
where   $\lesssim$ is an inequality up to multiplicative constant.
\medskip

We can note that a particular case of the $H$-reduction is when $x=y$ in the previous figure, we will refer to it as a triangle reduction, illustrated in the following picture (where we highlight the triangle)
 \[
        \begin{minipage}{0.4\textwidth}
      \includegraphics[width=0.4\linewidth]{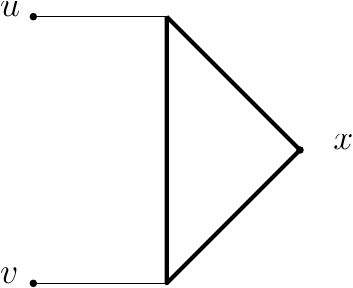}
    \end{minipage}\
    \lesssim
            \begin{minipage}{0.4\textwidth}
      \includegraphics[width=0.4\linewidth]{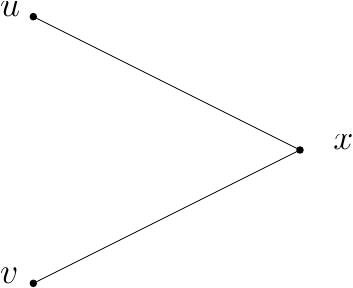}
    \end{minipage}\
\]
\end{remark}

\begin{remark}
Since generalized diagrams are sums of diagrams, we can apply an $H$-reduction to a generalized diagram $\mathcal{G}$ if that operation would work on all its constituting diagrams. 

In particular, we say that a non-connecting edge $e=\{x,y\}$ of generalized diagram is H-reducible, if its endpoints have at least an edge degree of 3 (including connecting edges) and if $e$ does not have a label (which means that this edge is present in all the diagrams of $\mathcal{G}$). A connecting edge $e$ with label A of generalized diagram is called H-reducible if in addition we require all edges neighbouring $e$ to not be labelled $A$.

We can use an H-reduction inequality on a generalized diagram $\mathcal{G}$ when its reduced version $\tilde{\mathcal{G}}$ has the same edge labeling as $\mathcal{G}$ and where the (potentially) newly created edges (e.g.~$e_{x\setminus y}$) inherits the labels of both the two edges that it replaced (e.g.~$\{x,z_1\}$, $\{x,z_2\}$ where $z_1,z_2$ are the unique two neighbours of $x$ which are not $y$). Notice that the new edge needs both labels since the H-reduction provides an upper bound. Here is an example of a part of a generalized graph that can be H-reduced and its result post H-reduction
 \[
        \begin{minipage}{0.4\textwidth}
      \includegraphics[width=0.4\linewidth]{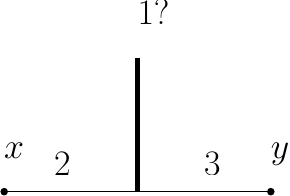}
    \end{minipage}\
    \lesssim \qquad 
            \begin{minipage}{0.4\textwidth}
      \includegraphics[width=0.4\linewidth]{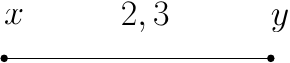}
    \end{minipage}\
\]
and notice that we could not have applied the H-reduction if the edge labelled 2 was instead labelled 1, or if connecting edge (in bold) had any label (say 5).
\end{remark}
\medskip

\subsection{A general graph theoretical result on H-reducibility}\label{sect_graph_theory}

{
In this section we prove a graph theoretical result that will allow us to deal with all arising diagrams
in a very efficient way. 
{\colAJ From now on we limit our consideration to simple graphs, not multi-graphs.}

\begin{definition} \ \\
(i) Let $G$ be a finite connected (not necessarily simple) graph, and $S \subset V(G)$, $|S| = 3$. 
We say that $x \in V(G) \setminus S$ is 3-connected to $S$ if there exist 3 edge-disjoint paths between $x$ and 
each element of $S$, respectively. The graph $G$ is said to be 3-connected to $S$ if $x$ is 3-connected to $S$ 
for any $x\in V(G) \setminus S$. \\
(ii) When $G$ is also simple, we call an edge $\{x,y\}$ strongly H-reducible (for the graph $G$ and 
the set $S$), if $x,y\in V(G) \setminus S$, $\{x,y\}$ is H-reducible and the graph $G$ H-reduced by $\{x,y\}$ 
is still 3-connected to $S$.
\end{definition}
 
 \begin{lemma}
 \label{fix_path-new}
 Let $G$ be a finite connected simple graph, $S \subset V(G)$, $|S| = 3$. 
 Assume that $G$ is 3-connected to $S$, all $v \in V(G) \setminus S$ have degree $3$, and all elements of 
 $S$ have degree at most $2$.
 \begin{itemize}
 \item[(i)] If $x, y \in V(G) \setminus S$, $\{x,y\} \in E(G)$, then $G$ is H-reducible by $\{x,y\}$, and its
 H-reduction by $\{x,y\}$ is a finite connected graph in which all vertices not in $S$ have degree $3$.
 \item[(ii)] If $\{x,y\}$ is strongly H-reducible, then the H-reduction of $G$ by $\{x,y\}$ is also a 
 simple graph (and 3-connected to $S$).
 \item[(iii)] If $x, y \in V(G) \setminus S$, $x \not= y$, and $\mathcal{P}$ is a simple path in $G$ which 
 does not use $\{x,y\}$ which starts and ends in $V(G)\setminus \{x,y\}$ then there is a simple path 
 $\mathcal{P}'$ in the H-reduction of $G$ by $\{x,y\}$, that has the same start and end points as $\mathcal{P}$
 and coincides with $\mathcal{P}$ on edges that belong to both graphs. 
 \end{itemize}
 \end{lemma}
 
\begin{proof} 
(i) Since $G$ is simple, there is no double edge between $x$ and $y$, and hence $G$ is H-reducible by
$\{x,y\}$. The H-reduction does not change the degrees of vertices in $V(G) \setminus (S \cup \{x,y\})$,
so these are still $3$. To see connectedness, note that $\{x,y\}$ cannot be a bridge in $G$, due to the
definition of 3-connectivity of $G$.

(ii) Since $G$ has no double edges, no loops can arise via the H-reduction. Assume that a double edge arises
after H-reduction. If neither endpoint of the double edge is an element of $S$, then the two edges adjacent 
to the double edge at either end form a cut set of size $2$ separating the double edge from $S$, and this 
cannot happen in the case of strong H-reduction. (Here we use that both ends of the double edge have degree 
$3$ by (i).) Assume now that the double edge that arose is of the form $\{s,x\}$ with $s \in S$, 
$x \not\in S$. Since $s$ has degree $2$ in the reduced graph, and two of the three edges out of $x$ lead
to $s$, the third edge out of $x$ forms a cut set of size $1$ separating $x$ from $S \setminus \{s\}$,
contradicting the assumption of strong H-reduction. Finally, the double edge cannot be of the form
$\{ s_1, s_2 \}$ with $s_1, s_2 \in S$, as this would again contradict the assumption of strong 
H-reduction. It follows that the reduced graph is simple.

(iii) Should the path $\mathcal{P}$ in $G$ pass through $x$, say, it necessarily uses both edges 
incident with $x$ that are different from $\{x,y\}$, and uses them consecutively. Thus we can replace
these edges with the new edge present after the H-reduction of $G$. 
\end{proof}

\begin{proposition}
\label{prop:strong-H-reduce}
Let $G$ be a finite connected simple graph, $S \subset V(G)$, $|S| = 3$. 
Assume that $G$ is 3-connected to $S$, all $v \in V(G) \setminus S$ have degree $3$, and all $s \in S$
have degree at most $2$. If $V(G)\setminus S$ contains two neighbouring vertices, then it contains a strongly 
H-reducible edge.
\end{proposition}

 \begin{remark}\label{counterex}
 We wrote the previous proposition with three simplifying hypotheses: $|S|=3$, 
 all vertices of $V(G)\setminus S$ have degree 3, and the vertices of $S$ have degree at most 2. 
 However, our proof can be generalized with minor adaptations to avoid these assumptions at the cost of some complications:
 \begin{enumerate}
 \item {\colAJ if $|S| \ge 3$ and vertices of $S$ are allowed to have any degree, then we can show that any 
 simple graph that has an edge $\{x,y\}$ where $x$ and $y$ are adjacent can be strongly H-reduced by some edge.}
 \item if vertices of $V(G)\setminus S$ are allowed to have any degree, then the statement is false because of the graph below. However, the theorem remains true if we slightly strengthen the connectivity property that $G$ is 3-connected to $S$, by assuming that for any $x\in V(G)\setminus S$  there exists 3 edge-disjoint paths from $x$ to $S$ such that the first vertex (after $x$) visited by any path cannot belong to vertices visited by the other two paths. 
\end{enumerate}

In the following graph the reduction by any H-reducible edge will result in the creation of a double edge.
 \[
     \begin{minipage}{0.3\textwidth}
      \includegraphics[width=0.7\linewidth]{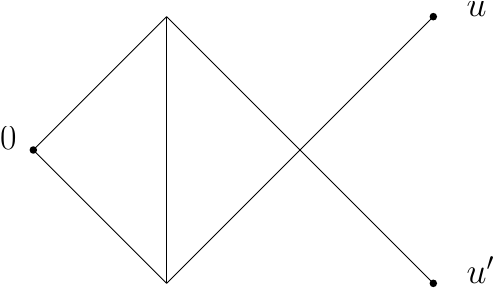}
    \end{minipage}
   \]
 \end{remark}

The proof of Proposition \ref{prop:strong-H-reduce} will rely on the following lemma.

\begin{lemma}
\label{lem:reduction}
Let $G$ be a finite connected simple graph, $S \subset V(G)$, $|S| = 3$. 
Assume that $G$ is 3-connected to $S$, all $v \in V(G) \setminus S$ have degree $3$, and all elements of
$S$ have degree at most $2$.
If $x,y\in V(G)\setminus S$, $\{x,y\} \in E(G)$ but $\{x,y\}$ is not strongly H-reducible, then there exists
a subgraph $G'$ of $G$ and a subset $S' \subset V(G')$ such that the following hold.
\begin{itemize}
    \item[(i)] $G'$ has fewer edges than $G$, and $|S'| = 3$.
    \item[(ii)] Either $y \in S'$ and $x \in V(G') \setminus S'$, or vice versa.
    \item[(iii)] $G'$ is simple, connected, all its vertices not in $S'$ have degree $3$, and all 
    $s' \in S'$ have degree $1$ in $G'$.
    \item[(iv)] $G'$ is 3-connected to $S'$.
    \item[(v)] $V(G') \setminus S'$ contains at least $2$ neighbouring vertices.
\end{itemize}
\end{lemma}

\begin{proof}
Let $\tilde{G}$ denote the H-reduction of $G$ by $\{x,y\}$, and let $g_x$ and $g_y$ be the new edges in
$\tilde{G}$ that were formed by merging edges at $x$ and $y$, respectively. Since this is not a strong H-reduction,
$V(G) \setminus S$ is strictly larger than $\{x,y\}$, and there exists a vertex $z$ in
$V(G) \setminus (S \cup \{x,y\})$ such that $z$ is not 3-connected to $S$ in $\tilde{G}$.
Using Menger's theorem, let $e, e' \in E(\tilde{G})$ be two edges that form a cutset between 
$z$ and $S$ in $\tilde{G}$. 

We first claim that both $e$ and $e'$ are distinct from $g_x$ and $g_y$
(and hence can be identified with edges in $E(G)$). Let $\mathcal{P}_i$, $i=1,2,3$, be edge-disjoint paths
from $z$ to $S$ in $G$. If none of these uses $\{x,y\}$, then by Lemma \ref{fix_path-new}(iii), we get three
(edge-disjoint) paths $\mathcal{P}'_i$, $i=1,2,3$, in $\tilde{G}$ from $z$ to $S$ in $\tilde{G}$, contradicting 
the choice of $z$. Let one of the paths, say $\mathcal{P}_1$, use $\{x,y\}$. Then Lemma \ref{fix_path-new}(iii)
gives paths $\mathcal{P}'_2$, $\mathcal{P}'_3$ in $\tilde{G}$ that are edge-disjoint, and cannot use the edges
$g_x$ and $g_y$. But they must pass through $e$ and $e'$, respectively, by the choice of $z$, and the claim follows.

Let now $G_0$ denote the connected component of $z$ in $E(G) \setminus \{e, e', \{x,y\}\}$, and let 
$G'$ denote the graph obtained from $G_0$ by adding the edges $e, e', \{x,y\}$. Using the paths in the 
previous paragraph, it is easy to see that $G'$ is connected. Also, exactly one of $x$ and $y$
is a vertex of $G_0$. For the remainder assume that $\mathcal{P}_1$ reaches $x$ before $y$, so that
$x \in V(G_0)$. Call $x_1 := x$, $y_1 := y$, and let us also write $e = \{ x_2, y_2 \} \in \mathcal{P}_2$, 
$e' = \{ x_3, y_3 \} \in \mathcal{P}_3$, with $x_2, x_3 \in V(G_0)$. Put $S' = \{y_1,y_2,y_3\}$.
Edge-disjointness of the paths $\mathcal{P}_i$, $i=1,2,3$, implies that the $y_i$'s are distinct, and 
hence $|S'|=3$. (It may happen, but is not harmful, that $x_2 = x_3$, in which case they necessarily also
equal $z$.)

We verify the properties claimed about $G', S'$. Two edges incident with $y$ are not in $G'$, and therefore
(i) follows. Property (ii) holds by construction. Simpleness of $G'$ follows since it is a subgraph of $G$, 
and we noted already that it is connected. Since $G_0$ was a connected component, degrees of vertices other 
than in the set $\{x_1,x_2,x_3,y_1,y_2,y_3\}$ remain $3$. The construction also implies that $x_i$, $i=1,2,3$, 
have degree $3$. Hence (iii) follows.

Let now $w \in V(G') \setminus S' = V(G_0)$. Since $w$ is 3-connected to $S$ in $G$, there are edge-disjoint
paths $\mathcal{P}^w_i$, $i=1,2,3$, from $w$ to each element of $S$. Since $S \cap V(G_0) = \es$, these paths
must use the edges $\{x_i,y_i\}$, $i=1,2,3$, and the claim (iv) follows.

Finally, since $x \in V(G') \setminus S'$, either neighbour of $x$ that is not $y$ shows that claim (v) holds.
\end{proof}

\begin{proof}[Proof of Proposition \ref{prop:strong-H-reduce}.]
If all pairs of neighbouring vertices in $V(G) \setminus S$ form a strongly H-reducible edge, there is 
nothing to prove. Otherwise let $x, y \in V(G) \setminus S$ be neighbours that do not form a strongly 
H-reducible edge, and use Lemma \ref{lem:reduction} to find a pair $G'$, $S'$ with the properties in 
that lemma. We will assume that $\{x,y\}$ and $G', S'$ have been chosen in such a way that $G'$ has
as few edges as possible. 

We first claim that all edges $\{u,v\}$ with $u,v \in V(G') \setminus S'$ 
are strongly H-reducible for $G', S'$. (We know this set is non-empty by property (v) of 
Lemma \ref{lem:reduction}.) Indeed, suppose not. Then applying Lemma \ref{lem:reduction} to 
$G', S'$ and an edge $\{u,v\}$ that is not strongly H-reducible for $G', S'$, we obtain a pair
$G'', S''$ where $G''$ has fewer edges than $G'$, by property (i) of the lemma, and this contradicts the 
choice of $\{x,y\}, G', S'$.

Select any edge $\{u,v\}$ with $u,v \in V(G') \setminus S'$ that, by the previous paragraph,
is necessarily strongly H-reducible for $G', S'$. Let $\tilde{G}'$ denote the H-reduction of
$G'$ by $\{u,v\}$. We claim that $\{u,v\}$ is strongly H-reducible for $G, S$
as well, which will complete the proof. First, since $V(G') \setminus S' \subset V(G) \setminus S$,
and $u,v$ have degree $3$ in $G$, the edge $\{u,v\}$ is H-reducible in $G, S$. Let 
$\tilde{G}$ denote the H-reduction of $G$ by $\{u,v\}$. Let $w \in V(G) \setminus S$,
$w \not= u,v$. We have to show that $w$ is 3-connected to $S$ in $\tilde{G}$. 

\emph{Case (i). $w \in V(G') \setminus S'$.} Using that $w$ is 3-connected to $S$ in $G$, we have 
three edge-disjoint paths $\mathcal{P}^w_i$, $i=1,2,3$, from $w$ to $S$. Since $\{ \{x_i,y_i\} : i=1,2,3 \}$
is a cutset between $w$ and $S$, we have that the paths use these three edges, and we may assume that
$\mathcal{P}^w_i$ passes through $\{x_i,y_i\}$, $i=1,2,3$. Let $\pi^w_i$ denote the portion of
$\mathcal{P}^w_i$ between $y_i$ and $S$, $i=1,2,3$. Note that none of the $\pi^w_i$'s uses $\{u,v\}$. 
Since $w$ is 3-connected to $S'$ in $\tilde{G}'$, we can find edge-disjoint paths $\rho^w_i$, $i=1,2,3$,
from $w$ to $y_i$, $i=1,2,3$, respectively, in $\tilde{G}'$. Concatenation of $\rho^w_i$ and
$\pi^w_i$, $i=1,2,3$, completes the proof for this case.

\emph{Case (ii). $w \in (V(G) \setminus V(G')) \cup S'$.} Again we use that $w$ is 3-connected to
$S$ in $G$ to select edge-disjoint paths $\mathcal{P}^w_i$, $i=1,2,3$, from $w$ to $S$. 
If none of the paths enters the set $V(G') \setminus S'$, we are done, as these remain paths in
$\tilde{G}$. Since any of these three paths that enter the set must do so via 
one of the edges $\{x_i,y_i\}$, and then exit via another of these three edges, at most one of the paths 
can enter $V(G') \setminus S'$. We may assume it is the path $\mathcal{P}^w_1$ and it enters via, say,
$(y_1,x_1)$ and exits via $(x_2,y_2)$. We distinguish the following subcases.
\begin{itemize}

\item[(ii)(a)] $\{u,v\} \cap \{x_1,x_2\} = \es$. In this case it is sufficient to use that
$\tilde{G}'$ is connected to get a path from $x_1$ to $x_2$ and replace the portion of $\mathcal{P}^w_1$
inside $V(G') \setminus S'$ by this path to obtain a path in $\tilde{G}$ edge-disjoint from 
$\mathcal{P}^w_2$, $\mathcal{P}^w_3$.

\item[(ii)(b)] $\{u,v\} = \{x_1,x_2\}$. In this case $y_1$ and $y_2$ get connected by an edge 
in $\tilde{G}$, and we may omit any portion of $\mathcal{P}^w_1$ inside $V(G') \setminus S'$
to get a path in $\tilde{G}$.

\item[(ii)(c)] $\{u,v\} \cap \{x_1,x_2\} = \{ x_1 \}$. We may assume $x_1 = u$. Let $t$ be the neighbour in $G'$
of $x_1 = u$ that is different from both $y_1$ and $v$. Then $t \in V(\tilde{G}') \setminus S'$, and hence
it has three edge-disjoint paths to $\{y_1,y_2,y_3\}$ in $\tilde{G}'$. One of these paths is necessarily 
the edge $(t,y_1)$ created by the H-reduction. Another will connect $t$ to $y_2$, and we can use the 
concatenation of these two to replace the portion of $\mathcal{P}^w_1$ inside $V(G') \setminus S'$.

\item[(ii)(d)] $\{u,v\} \cap \{x_1,x_2\} = \{ x_2 \}$. The proof is analogous to case (ii)(c).

\end{itemize}
The above completes the proof of strong H-reduction, and the proposition follows.
\end{proof}

Assume now that $G, S$ satisfies the assumptions of Proposition \ref{prop:strong-H-reduce}.
Let us do the following procedure.
\begin{enumerate}
\item If $G$ does not have a strongly H-reducible edge then set $G_{\mathcal{R}}:=G$. Otherwise move to step 2.
\item Consider a strongly H-reducible edge $e$. We construct $\tilde{G}$, the H-reduction of $G$ by $e$, and restart the process with $\tilde{G}$ in place of $G$.
\end{enumerate}
 
 Note that at each iteration $\tilde{G}$ again satisfies the assumptions of Proposition \ref{prop:strong-H-reduce} 
 and has three edges less than $G$. This means that the procedure finishes, since we start with a finite graph. 
 Any resulting graph $G_{\mathcal{R}}$ is called a reduced graph of $G$. Note that the order in which we choose 
 to do the reductions could lead to different reduced graphs starting from the same initial graph. The following
 proposition holds for any admissible construction.
} 

We say that a {\colAJ graph}
$G$ can be H-reduced to a graph $G'$ if there exists a sequence of successive H-reduction that transform $G$ into $G'$.

\begin{proposition}\label{super_reduc}
Let $G$ be a finite connected simple graph 
and let $S = \{ x,y,z\} \subset V(G)$ be three distinct vertices. 
Assume that $G$ is 3-connected to $\{x,y,z\}$, that any $v\in V(G) \setminus \{x,y,z\}$ has degree 3 and that $x$, $y$ and $z$ have exactly degree two. Then $G$ can be H-reduced to one of the following two graphs
\[
\begin{minipage}{0.3\textwidth}
      \includegraphics[width=0.7\linewidth]{Lem_triangleb-eps-converted-to.pdf}
    \end{minipage}
    \qquad
    \begin{minipage}{0.3\textwidth}
      \includegraphics[width=0.7\linewidth]{squareb-eps-converted-to.pdf}
    \end{minipage}
\]
\end{proposition}

\begin{proof}
The {\colAJ reduced} graph $G_{\mathcal{R}}$ does not have any strongly H-reducible edges by construction.
 
 We can see that any $v,v'\in V(G_{\mathcal{R}})\setminus \{0,u,u'\}$ cannot be adjacent, otherwise the edge $\{v,v'\}$ would be H-reducible by Proposition~\ref{prop:strong-H-reduce}.
 
{\colAJ We have by Lemma \ref{fix_path-new}(i) that any $v\in V(G_{\mathcal{R}})\setminus \{x,y,z\}$ has 
degree $3$.}
In the construction of $G_{\mathcal{R}}$, no reduction occurs on any edge adjacent to $\{x,y,z\}$, 
and this means that $x$, $y$ and $z$ have exactly degree two in $G_{\mathcal{R}}$ since that is the case in $G$. Furthermore, since $G_{\mathcal{R}}$ is a simple graph by construction, we know that  $x$, $y$ and $z$  have exactly two neighbours in $G_{\mathcal{R}}$.
 
Let us do a case distinction 
\begin{enumerate}
\item If $|V(G_{\mathcal{R}}) \setminus \{x,y,z\}|=0$, then $G_{\mathcal{R}}$ is a triangle (since $x$, $y$ and $z$ all have 2 neighbours within $\{x,y,z\}$). 
\item If $|V(G_{\mathcal{R}}) \setminus \{x,y,z\}|= 1$, let $v\in V(G_{\mathcal{R}})\setminus \{x,y,z\}$. We know that $v$ has {\colAJ 3 neighbours} since $v$ is 3-connected to $\{x,y,z\}$. This means that $v$ is connected to $x$, $y$ and $z$ since those are the only vertices different from $v$ ($v$ cannot have 2 edges to one of these vertices because they have degree 2 and 2 neighbours). Since $x$ is of degree $2$, it has to have another neighbour in addition to $v$ and the only available vertices are $y$ or $z$. Let's assume that the neighbour is $y$ (the other case is similar). Then $x$ and $y$ have already degree 2 and thus cannot be neighbours of $z$.  Furthermore, $v$ is already a neighbour of $z$ and $G_{\mathcal{R}}$ contains no other vertices than $x,y,z,v$ which means $z$ is of degree 1. This is a contradiction.
\item If $|V(G_{\mathcal{R}}) \setminus \{0,u,u'\}|= 2$, then we know there exists $v,v'\in V(G_{\mathcal{R}})\setminus \{x,y,z\}$ and, since we proved that they cannot be adjacent, they both have to be adjacent to $x$, $y$, and $z$. This gives the second graph in the lemma.
 \item If $|V(G_{\mathcal{R}}) \setminus \{x,y,z\}|> 2$, then we know there exists $v,v',v''\in V(G_{\mathcal{R}})\setminus \{x,y,z\}$ and, since we proved that they cannot be adjacent, they  all have to be adjacent to $x$, $y$, and $z$. This means that $x$, $y$, and $z$ have degree at least 3, which is a contradiction.
 \end{enumerate}
 This proves the proposition.
\end{proof}
 
The previous result, which is purely graph theoretical, is highly useful for bounding diagrams.

\subsection{General diagrammatic bound for 3-connected graphs}
Before moving onto our main bound, we will prove an interesting preparatory lemma. Since, under certain hypotheses, we can H-reduce a graph to a graph in which unlabelled vertices are not neighbours, it means that in the resulting graph the only remaining unlabelled points are adjacent to labelled points. This means the graph is built from certain specific pieces and in the following lemma we give a bound on such pieces (in the case of a graph of degree at most 3).

\begin{lemma}\label{holder}
For any $u,u'\in \mathbb{Z}^d$,
\[
\begin{minipage}{0.3\textwidth}
      \includegraphics[width=0.4\linewidth]{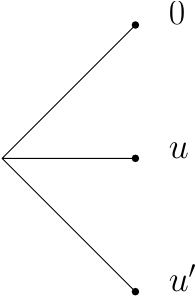}
    \end{minipage}
    \lesssim 
\left(
    \begin{minipage}{0.3\textwidth}
\begin{center}      \includegraphics[width=0.6\linewidth]{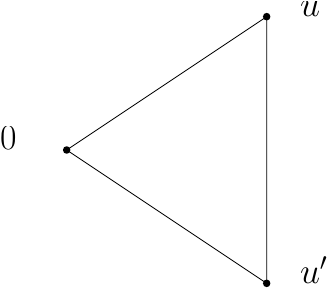}\end{center}
    \end{minipage}\right)^{1/2}
\]
\end{lemma}

\begin{proof}
The diagram on the left-hand side corresponds to 
  \begin{align*}
& \sum_{x\in  \mathbb{Z}^d}\tau(x)\tau(x-u')\tau(x-u) \\
=& \sum_{x\in \mathbb{Z}^d} (\tau(x)\tau(x-u'))^{1/2}(\tau(x)\tau(x-u))^{1/2}(\tau(x-u')\tau(x-u))^{1/2} \\
\leq & \Bigl (\sum_{x\in \mathbb{Z}^d} (\tau(x)\tau(x-u'))^{3/2}\Bigr)^{1/3} \Bigl (\sum_{x\in \mathbb{Z}^d} (\tau(x)\tau(x-u))^{3/2}\Bigr)^{1/3} \Bigl (\sum_{x\in \mathbb{Z}^d} (\tau(x-u')\tau(x-u))^{3/2}\Bigr)^{1/3}
\end{align*}  
where we used Hölder's inequality. Using Lemma~\ref{lem:ConvBd} and $\tau(x) \lesssim (1+\abs{x})^{-(d-2)}$ we see that (using $\frac 32(d-2) >d$ whenever $d<6$)
\[
\sum_{x\in \mathbb{Z}^d} (\tau(x)\tau(x-u'))^{3/2} =(\tau^{3/2} \ast \tau^{3/2})(u')  \lesssim (1+\abs{u'})^{-\frac 32(d-2)},
\]
and similarly 
\[
\sum_{x\in  \mathbb{Z}^d} (\tau(x)\tau(x-u))^{3/2} \lesssim (1+\abs{u})^{-\frac 32(d-2)} \text{ and } \sum_{x\in \mathbb{Z}^d} (\tau(x-u')\tau(x-u))^{3/2} \lesssim (1+\abs{u-u'})^{-\frac 32(d-2)}.
\]

Putting the last three equations together yields
\[
\sum_{x \in \mathbb{Z}^d} \tau(x)\tau(x-u')\tau(x-u) \lesssim (1+\abs{u})^{-\frac{d-2}2} (1+\abs{u'})^{-\frac{d-2}2} (1+\abs{u-u'})^{-\frac{d-2}2},
\]
which is lower than $C(\tau(0,u)\tau(u,u')\tau(0,u'))^{1/2}$.
\end{proof}

\begin{proposition}\label{kill_bill}
Let $G$ be a finite simple connected  graph and denote $x,y,z\in V(G)$ three distinct vertices in $G$. Assume that $G$ is 3-connected to $\{x,y,z\}$, that any $v\in V(G) \setminus \{x,y,z\}$ has degree 3 and that $x$, $y$ and $z$ have exactly degree two. Assign labels $\ell(x)=0$, $\ell(y)=u$ and $\ell(z)=u'$ for some $u,u' \in \mathbb{Z}^d$. We have
 \[
\operatorname{Diag}((G,\{x,y,z\},\ell))\leq C_{|E(G)|} \tau(0,u)\tau(u',u)\tau(u',0),
  \]
where the constant  $C_{|E(G)|}<\infty$ only depends on the edge cardinality of $G$. 
   \end{proposition}
 \begin{proof}
 By Proposition~\ref{super_reduc}, the graph $G$ can be reduced, after at most $|E(G)|$ H-reductions, to the either of the two graphs shown in said proposition. This means that 
\[
\text{Diag}((G,\{x,y,z\},\ell)) \leq (C_H)^{|E(G)|} \left(\begin{minipage}{0.3\textwidth}
      \includegraphics[width=0.55\linewidth]{Lem_triangle-eps-converted-to.pdf}
    \end{minipage}
+ 
    \begin{minipage}{0.3\textwidth}
      \includegraphics[width=0.55\linewidth]{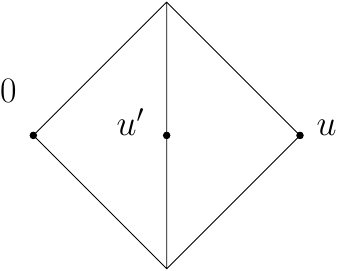}
    \end{minipage}\right)
\]
   where $C_H$ is the constant appearing in Proposition \ref{prop_Hred}.

 The first diagram is equal to $\tau(0,u)\tau(u,u')\tau(0,u')$. The second one is
\[
   \sum_{x,y\in \mathbb{Z}^d} \tau(x)\tau(x-u')\tau(x-u)\tau(y)\tau(y-u')\tau(y-u) = \Bigl(\sum_{x\in  \mathbb{Z}^d}\tau(x)\tau(x-u')\tau(x-u)\Bigr)^2 ,
   \]
   where the right-hand side can be bounded by Lemma~\ref{holder}. This proves the Proposition.
 \end{proof}

\subsection{Building a large class of 3-connected graphs}\label{sect_big_class}

 Let $G$ be a connected simple graph with a finite number of edges and denote $0,u,u'\in V(G)$ three distinct vertices in $G$.

 \begin{remark}\label{ex_3strong}
 The following two graphs are  3-connected to $\{0,u,u'\}$, all vertices in $V(G)\setminus \{0,u,u'\}$ have degree 3 and $0$, $u$, $u'$ all have degree 2
\[
 \begin{minipage}{0.3\textwidth}
      \includegraphics[width=0.5\linewidth]{Lem_triangle-eps-converted-to.pdf}
    \end{minipage}\
      \qquad 
        \begin{minipage}{0.3\textwidth}
      \includegraphics[width=0.8\linewidth]{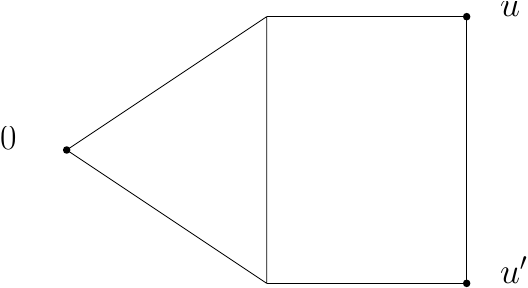}
    \end{minipage}\
\]
\end{remark}
 
 The following proposition applied recursively will allow to build a very large class of 
 3-connected graphs.

 \begin{proposition}\label{expand_class}
Let $G$ be a connected  simple graph with  three labelled vertices $0,u,u'$. Assume that $G$ is 3-connected to $\{0,u,u'\}$, that all vertices in $V(G)\setminus \{0,u,u'\}$ have degree 3 and that $0$, $u$, $u'$ all have degree 2. Then, all the graphs in the generalized diagrammatic graphs $(G\circledast u^+) \uplus G^{(1)}$,  $(G\circledast y) \uplus G^{(2)}$ and $(G\circledast^2 y) \uplus G^{(3)}$ are  3-connected to $\{0,u,u^+\}$ where $y \notin \{0,u,u'\}$ the graphs $G^{(1)}$, $G^{(2)}$ and $G^{(3)}$ are respectively
 \[
\begin{minipage}{0.3\textwidth}
      \includegraphics[width=0.5\linewidth]{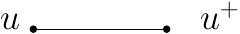}
    \end{minipage}
    \begin{minipage}{0.3\textwidth}
      \includegraphics[width=0.6\linewidth]{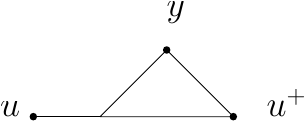}
    \end{minipage}
        \begin{minipage}{0.3\textwidth}
      \includegraphics[width=0.6\linewidth]{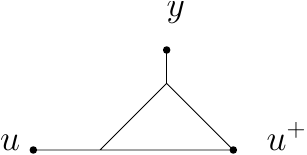}
    \end{minipage}
\]
 \end{proposition}

 \begin{remark}\label{cone}
 Note that in principle, we only defined the operation $\uplus$ on fully labelled graphs. However in the previous proposition is a statement on the graph structure of, for example, $(G\circledast u^+) \uplus G^{(1)}$ so in order to keep the notations lighter we avoided introducing artificial labels for all the vertices of graphs  $G^{(1)}$, $G^{(2)}$ and $G^{(3)}$.
 \end{remark}

 \begin{remark}\label{chess}
In generalized diagrams notation, this means that we can merge the   following graphs with $G$ at the point $u$
 \[
\begin{minipage}{0.3\textwidth}
      \includegraphics[width=0.5\linewidth]{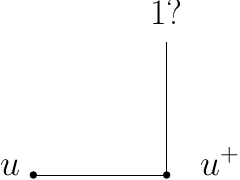}
    \end{minipage}
    \begin{minipage}{0.3\textwidth}
      \includegraphics[width=0.6\linewidth]{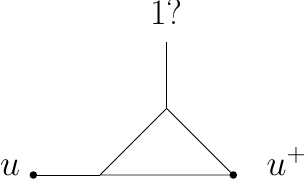}
    \end{minipage}
        \begin{minipage}{0.3\textwidth}
      \includegraphics[width=0.6\linewidth]{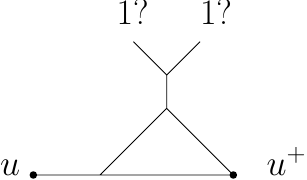}
    \end{minipage}
\]
and the connecting edge \lq\lq1?\rq\rq~gets connected to some edge in $G$.
 \end{remark}

 \begin{remark}\label{tuli}
 Note that if  $G$ is a connected 
 graph 3-connected to $\{0,u,u'\}$ then
 \begin{enumerate}
 \item  for any two vertices $x\neq y\in V(G)\setminus \{0,u,u'\}$ there exists two paths which are edge-disjoint paths starting from $x$ and $y$ with one ending at $0$ and the other at $u'$. This can be seen as follows: since $x$ is 3-connected to $\{0,u,u'\}$ there exists 2 simple paths $\mathcal{P}_1$ and $\mathcal{P}_2$ edge-disjoint starting from $x$ with the first ending at $0$ and the other at $u'$. Since $y$ is connected to $0$ by a simple path $\mathcal{P}_3$ we can look at the first vertex $z$ visited by $\mathcal{P}_3$ which is a vertex of  $\mathcal{P}_1$ or $\mathcal{P}_2$ (that exists since $0$ is part of $\mathcal{P}_3$  and either $\mathcal{P}_1$ or $\mathcal{P}_2$). By following $\mathcal{P}_3$ until $z$ and then the path it intersects, say $\mathcal{P}_1$ (or $\mathcal{P}_2$), we create a path $\mathcal{P}_4$ from $y$ to $0$ which is disjoint from $\mathcal{P}_2$ (or $\mathcal{P}_1$) and the pair $(\mathcal{P}_4, \mathcal{P}_2)$ (or $(\mathcal{P}_4, \mathcal{P}_1)$) provides the two paths we were looking for.
 \item for any  $x\in V(G)\setminus \{u\}$, there exists two paths which are edge-disjoint paths starting from $x$ and $u$ with one ending at $0$ and the other at $u'$. The case where $x\in \{0,u'\}$ is trivial. Otherwise, notice that the only property of $y$ we used in the previous is that $y$ was connected to $0$, so the same proof can be repeated with $u$ in place of $y$ since $G$ is connected proving the result for $x\in V(G)\setminus  \{0,u,u'\}$.
 \end{enumerate} 
 \end{remark}
 
\begin{proof}[Proof of Proposition \ref{expand_class}]
The graphs $H^{(1)}$, $H^{(2)}$ and $H^{(3)}$ are the 3 graphs drawn in Remark~\ref{chess} where the endpoints of connecting edges labelled \lq\lq 1?\rq\rq~are considered to be vertices (in the third case both vertices \lq\lq 1?\rq\rq~are considered to be distinct). We deal with the various cases separately.

\medskip 
\noindent{\it Case 1:  graphs arising from $(G\circledast u^+) \uplus G^{(1)}$ and  $(G\circledast y) \uplus G^{(2)}$.} 

 We will do the proof for $(G\circledast y) \uplus G^{(2)}$, as the one for $(G\circledast y) \uplus G^{(1)}$ can be done by replacing $H^{(2)}$ {\colAJ{with}} $H^{(1)}$.
 
Fix $u_*\notin V(G)$ and $\{x,y\} \in E(G)$. We denote $G^{(u_*)}$ the graph $(V(G)\cup \{u_*\}, (E(G)\setminus \{x,y\})\cup \{\{x,u_*\}, \{y,u_*\}\})$, where $u_*\notin V(G)$ and $\{x,y\} \in E(G)$.  
We denote $G^{(u_*)}\circledast H^{(2)}$ the graph resulting from identifying $u_*$ in $G^{(u_*)}$ with the vertex \lq\lq 1?\rq\rq~in $H^{(2)}$.
 
 Let us show that $G^{(u_*)}\circledast H^{(2)}$ is 3-connected to $\{0,u_+,u'\}$. For this let us take $v\in V(G^{(u_*)}\circledast H^{(2)})$, there are two cases
 \begin{itemize}
 \item if $v\in V(G^{(u_*)})\setminus u_*$, then we know there {\colAJ{exist}} 3 edge-disjoint paths $(p_i)_{i=1,2,3}$ from $v$ to $0$, $u$, $u'$ in $G$, which induces similar paths in $G^{(2)}$ by replacing the potential transition of the edge $\{x,y\}$ by two transitions: one from $x$ to $u_*$ and the other from $u_*$ to $y$ (or the reverse order if needed). The path leading to $u$ (assume without loss of generality that it is $p_3$), can then be completed using only edges and vertices of $H^{(2)}$ with a path leading to $u_+$ in $G^{(u_*)}\circledast H^{(2)}$ . If this new path is called $\tilde{p}_3$, we can see that the 3 paths $p_1$, $p_2$ and $\tilde{p}_3$ are 
 {\colAJ{edge disjoint}} and connect $v$ to the 3 points $0$, $u_+$ and $u'$. Hence $v$ is 3-connected to $\{0,u,u'\}$ in $G^{(2)}\circledast H^{(2)}$. 
 \item if $v\in V(H^{(2)})\setminus \{u,1?\}$ , we can easily see that $v$ can be connected with 3  paths in $H^{(2)}$ (which are edge-disjoint) ending in $u_+$, $u$ and \lq\lq 1?\rq\rq (the last being identified with $u_*$ in $G^{(u_*)}\circledast H^{(2)}$ ). We know that $u_*$ is adjacent to a vertex 
 {\colAJ{$v_*\in V(G)\setminus \{u\}$}} and using the second point of Remark~\ref{tuli}, we can find two edge-disjoint paths in $G$ from $v_*$ and $u$ with one leading to $0$ and the other to $u'$. Adding a transition from $u_*$ to $v_*$ we find two edge-disjoint paths from $u_*$ and $u$ with one leading to $0$ and the other to $u'$. We can then concatenate those two paths found with the 2 paths we built in $H^{(2)}$ which ended in $u$ and $u_*$ to build two paths $\tilde{p_1}$ and $\tilde{p_2}$. We can see that  $\tilde{p_1}$,  $\tilde{p_2}$ and the path $\tilde{p_3}$ from $v$ to $u_+$ in $H^{(2)}$ provides us with three edge-disjoint paths  in $G^{(u_*)}\circledast H^{(2)}$ from $v$ to $\{0,u_+,u'\}$. This is illustrated in the following picture, where the point $v$ is circled.
 \[
           \begin{minipage}{0.4\textwidth}
      \includegraphics[width=0.9\linewidth]{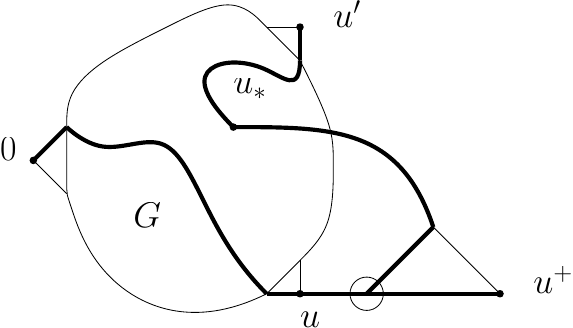}
    \end{minipage}\
    \]
 \item if $v=u$ or $v=u_*$ (the latter being identified with $u_*$ in $G^{(u_*)}\circledast H^{(2)}$). We know that $v$ can be connected with two edge-disjoint paths in $G$, and thus in $G^{(u_*)}$ to $0$ and $u'$ (this can be proved by using the first part of Remark~\ref{tuli} applied to the two neighbours of $v$ in $G$). We can complete this with a path from $v$ to $u_+$ within $H^{(2)}$ to provide the three edge-disjoint paths  in $G^{(u_*)}\circledast H^{(2)}$ from $v$ to $\{0,u_+,u'\}$. This is illustrated in the following picture.
   \[
              \begin{minipage}{0.4\textwidth}
      \includegraphics[width=0.9\linewidth]{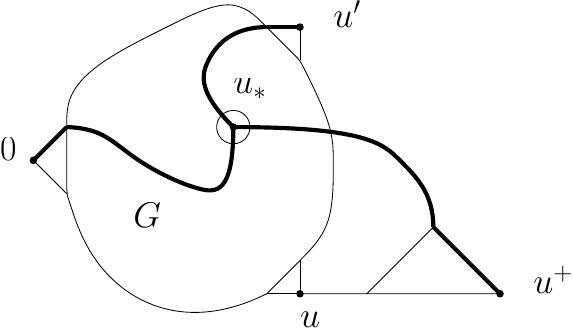}
    \end{minipage}\
  \]
 
 \end{itemize}
 
 Hence $G^{(u_*)}\circledast H^{(2)}$ is 3-connected to $\{0,u_+,u'\}$. Since any graph in the generalized diagrammatic graph $(G\circledast u^+) \uplus G^{(2)}$ can be written as $G^{(u_*)}\circledast H^{(2)}$ for some $u_*\notin V(G)$ and $\{x,y\} \in E(G)$, this proves the result for  $(G\circledast u^+) \uplus G^{(2)}$ .

\medskip
\noindent{\it Case 2: graphs arising from $(G\circledast^2 y) \uplus G^{(3)}$. }

This part has to be {\colAJ{divided according to}} whether the double connection from $\circledast^2 y$ lands on different edges of $G$ or not.

Let us start by the first case. Fix $u_*\neq u_{**} \notin V(G)$ and $\{x_1,y_1\}\neq \{x_2,y_2\} \in E(G)$, we denote $G^{(u_*,u_{**})}$ the graph 
\[(V(G)\cup \{u_*,u_**\}, (E(G)\setminus \{\{x_1,y_1\},\{x_2,y_2\}\})\cup \{\{x_1,u_*\}, \{y_1,u_*\},\{x_2,u_{**}\},\{y_2.u_{**}\}\}).\] 
We denote $G^{(u_*,u_{**})}\circledast H^{(3)}$ the graph resulting from identifying $u_*$ and~$u_{**}$ in $G^{(u_*,u_{**})}$ with the leftmost and the rightmost vertex \lq\lq 1?\rq\rq~in $H^{(3)}$.

 Let us show that $G^{(u_*,u_{**})}\circledast H^{(3)})$ is 3-connected to $\{0,u_+,u'\}$. Denote $a$ the vertex in $H^{(3)}$ which is adjacent to both \lq\lq 1?\rq\rq. Take $v\in V(G^{(u_*,u_{**})}\circledast H^{(3)}$, there are three cases
 \begin{itemize}
 \item if $v\in V(G^{(u_*,u_{**})})\setminus \{u_*,u_{**}\}$, then we know there exists 3 edge-disjoint paths $(p_i))_{i=1,2,3}$ from $v$ to $0$, $u$, $u'$ in $G$, which induces similar paths in $G^{(u_*,u_{**})}$ by replacing the potential transition along the edge $\{x_1,y_1\}$ (resp.~$\{x_2,y_2\}$) by two transitions one from $x_1$ (resp.~$x_2$) to $u_*$ (resp.~$u_{**}$) and the other from $u_*$ (resp.~$u_{**}$)  to $y_1$ (or $y_2$) (or the reverse order if needed). The path leading to $u$ (assume without loss of generality that it is $p_3$), can then be completed using only edges and vertices of $H^{(3)}$ with a path leading to $u_+$ in $G^{(u_*,u_{**})}\circledast H^{(3)}$ . If this new path is called $\tilde{p}_3$, we can see that the 3 paths $p_1$, $p_2$ and $\tilde{p}_3$ are vertex and edge disjoint and connect $v$ to the 3 points $0$, $u_+$ and $u'$. Hence $v$ is 3-connected to $\{0,u,u'\}$ in $G^{(u_*,u_{**})}\circledast H^{(3)}$.
 \item if $v\in V(H^{(3)})\setminus \{a,u,1?,1?\}$ (by this we mean that we do not consider both 1? points). We can easily see that $v$ can be connected with 3  paths in $H^{(3)}$ (which are edge-disjoint) one ending in $u_+$, one in $u$ and  the last one in the leftmost \lq\lq 1?\rq\rq (identified with $u_*$  in $G^{(u_*,u_{**})}\circledast H^{(3)}$ ). Essentially, repeating the argument from the corresponding case for $G^{(u_*)}\circledast H^{(2)}$ we can create three edge-disjoint paths  in $G^{(u_*,u_{**})}\circledast H^{(3)}$ from $v$ to $\{0,u_+,u'\}$.
 \item if $v=a$, then $v$ can be connected disjointly with 3 paths $(p_i)_{i=1,2,3}$ going respectively to $u_+$ the leftmost and the rightmost vertex \lq\lq 1?\rq\rq in $H^{(3)}$ (those being identified with  $u_*$ and $u_{**}$ in $G^{(u_*,u_{**})}\circledast H^{(3)}$). Whichever the edges of $G$ on which $u_*$ and $u_{**}$ are located, we can find $z_1 \neq z_2\in V(G) \setminus\{u\}$ such that  $z_1$ and $z_2$ are adjacent to $u_*$ and $u_{**}$ respectively. (otherwise those edges would have a endpoints $u$ and $z_1=z_2$ which means we have a double edge, contradicting the fact that $G$ is a simple graph). The first point of Remark~\ref{tuli} guarantees the existence of two paths in $G$, $\tilde{p_1}$ and $\tilde{p_2}$, which are vertex and edge disjoint  from $z_1$ and $z_2$ with one ending at $0$ and the other at $u'$. Adding a transition from $u_*$ to $z_1$ (resp.~$u_{**}$ to $z_2$) at the start of $\tilde{p_1}$ (resp.~$\tilde{p_2}$) yields two new edge-disjoint paths, $\tilde{p}_3$ and $\tilde{p}_4$, in $G^{(2)}$ going from $u_*$ and $u_{**}$ with one ending at $0$ and the other at $u'$. Now, we concatenate $p_2$ and $\tilde{p}_3$ (resp.~$p_3$ and $\tilde{p}_4$) to create two vertex and edge disjoint paths in $G^{(u_*,u_{**})}\circledast H^{(3)}$ from $v$  with one ending at $0$ and the other at $u'$ and those two paths also being disjoint from $p_1$ connecting $0$ to $u_+$. Showing that $v$ is 3-connected to $\{0,u_+,u'\}$ in  $G^{(u_*,u_{**})}\circledast H^{(3)}$.
 \item if $v=u$ or $v$ is one of the 1? vertices. Then, just as in the case $G^{(u_*,u_{**})}\circledast H^{(2)}$, we can connect $v$ to $0$ and $u'$ using disjoint paths within $G^{(u_*,u_{**})}$ and we can furthermore find another disjoint path within $H^{(3)}$ which connects $v$ to $u_+$. Those three paths show that $v$ is 3-connected to $\{0,u_+,u'\}$ in  $G^{(u_*,u_{**})}\circledast H^{(3)}$. \end{itemize}
 This case distinction shows that  $G^{(u_*,u_{**})}\circledast H^{(3)}$  is 3-connected to $\{0,u_+,u'\}$.
 
Finally, we deal with the case where the double connection lands on the same edge of $G$.  Fix $u_*\neq u_{**} \notin V(G)$ and $\{x,y\} \in E(G)$, we denote $\tilde{G}^{(u_*,u_{**})}$ the graph $(V(G)\cup \{u_*,u_**\}, (E(G)\setminus \{\{x,y\})\cup \{\{x,u_*\}, \{u_{**},u_*\},\{y,u_{**}\}\})$.  We denote $\tilde{G}^{(u_*,u_{**})}\circledast H^{(3)}$ for the graph resulting from identifying $u_*$ and $u_{**}$ in $\tilde{G}^{(u_*,u_{**})}$ with the leftmost and the rightmost vertex \lq\lq 1?\rq\rq~in $H^{(3)}$.

  Let us show that $\tilde{G}^{(u_*,u_{**})}\circledast H^{(3)}$ is 3-connected to $\{0,u_+,u'\}$. Denote $a$ the vertex in $H^{(3)}$ which is adjacent to both \lq\lq 1?\rq\rq. Take $v\in V(\tilde{G}^{(u_*,u_{**})}\circledast H^{(3)})$, there are three cases
 \begin{itemize}
 \item if $v\in V(\tilde{G}^{(u_*,u_{**})})\setminus \{u_*,u_{**}\}$, we can follow the same proof as in the case ${G}^{(u_*,u_{**})}\circledast H^{(3)}$.
  \item if $v\in V(\tilde{G}^{(u_*,u_{**})})\setminus \{a,u,1?,1?\}$,  we can follow the same proof as in the case ${G}^{(u_*,u_{**})}\circledast H^{(3)}$.
 \item if $v=a$, we see  that $u_*$ and $u_{**}$ each have a neighbour in $V(G)$, we denote them $z_1$, $z_2$ respectively (they have to be different since those were the endpoint of an edge in $G$). We can then follow the same proof as in the case ${G}^{(u_*,u_{**})}\circledast H^{(3)}$.
 \item if $v=u$ or $v$ is one of the 1? vertices, we can use the same proof as in the case ${G}^{(u_*,u_{**})}\circledast H^{(3)}$.
 \end{itemize}
This shows that  $\tilde{G}^{(u_*,u_{**})}\circledast H^{(3)}$ is 3-connected to $\{0,u_+,u'\}$.

Any graph in the generalized diagrammatic graph $(G\circledast^2 y) \uplus G^{(3)}$ can be written as $G^{(u_*,u_{**})}\circledast H^{(3)}$ for some $u_*\neq u_{**} \notin V(G)$ and $\{x_1,y_1\}\neq \{x_2,y_2\} \in E(G)$ or as $\tilde{G}^{(u_*,u_{**})}\circledast H^{(3)}$ for some $u_*\neq u_{**} \notin V(G)$ and $\{x,y\} \in E(G)$. This proves the result for $(G\circledast^2 y) \uplus G^{(3)}$. 
\end{proof}

\subsection{Diagrams appearing in the expansion of the bi-infinite incipient cluster}
In this section we shall give several diagrammatic bounds which will be needed in order to bound the expansion terms. 
We start by a simple diagrammatic estimate 
\begin{lemma}\label{lem_triangle_plus}
There exists $C>0$ such that for all $x,x'\in\Zd$ we have 
\[
\sum_{y} \tau(x) \tau(y) \tau(x,y) \tau(x',y) = C\tau(x)\tau(x')  \min\{|x|,|x-x'|\}^{-(d-4)} .
\]
\end{lemma}
\begin{proof}
Notice that either $|y|\geq |x'|/2$ or $|x'-y| \geq |x'|/2$, hence, by~\eqref{eq:tauAsy},  $\min( \tau(y), \tau(x',y)) \leq C \tau(x')$
\[
\sum_{y} \tau(x) \tau(y) \tau(x,y) \tau(x',y) \leq C \tau(x)\tau(x') \Bigl(\sum_{y}  \tau(y) \tau(x,y) +\sum_{y}  \tau(y) \tau(x',y)\Bigr),
\]
but now, for the first sum appearing on the right side we can see that $|y|\geq |x'|/2$ or $|x'-y| \geq |x'|/2$, thus $\min( \tau(y), \tau(x',y)) \leq C \tau(x')$ and with a similar reasoning for the second sum
\[
\sum_{y} \tau(x) \tau(y) \tau(x,y) \tau(x',y) \leq C \tau(x)\tau(x') \Bigl(\sum_y \tau(x,y)\tau(y)+ \sum_y \tau(x',y)\tau(y,x)\Bigr),
\]
and by \eqref{eq:TauTauBd} we know that $\sum_y \tau(x,y)\tau(y)=O(|x|^{-(d-4)})$ and $\sum_y \tau(x',y)\tau(y,x)=O(|x-x'|^{-(d-4)})$. The lemma follows.
\end{proof}

Apart from the previous elementary diagram, we can bound all the diagrams using the machinery developed so far in this section.
\begin{lemma}\label{extra_diag}
There exists $C>0$ such that for all $x,x'\in\Zd$ we have 
      \[
        \begin{minipage}{0.3\textwidth}
      \includegraphics[width=0.8\linewidth]{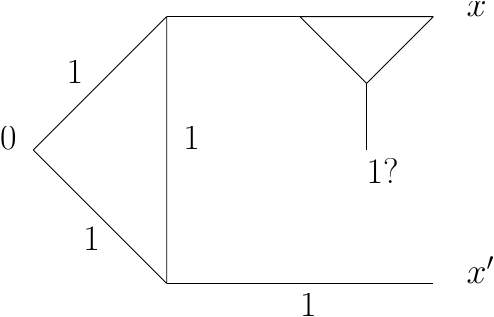}
    \end{minipage}\
=
\tau(x)\tau(x') \min\{|x|,|x-x'|\}^{-(d-4)} 
    \]
    \end{lemma}
    \begin{proof}
We have 
          \begin{align*}
 \begin{minipage}{0.3\textwidth}
      \includegraphics[width=0.8\linewidth]{last1-eps-converted-to.pdf}
    \end{minipage}\
      &=
        \begin{minipage}{0.3\textwidth}
      \includegraphics[width=0.8\linewidth]{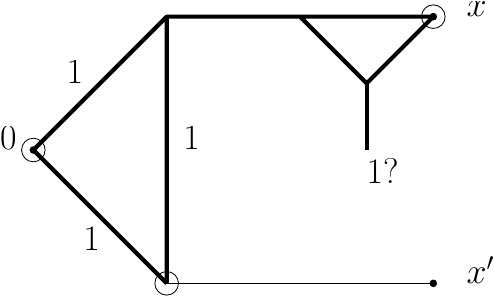}
    \end{minipage}\
    +
         \begin{minipage}{0.3\textwidth}
      \includegraphics[width=0.8\linewidth]{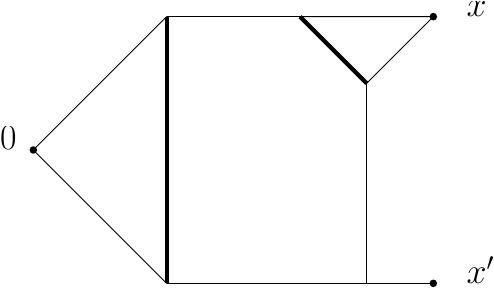}
    \end{minipage}\ \\
    & \lesssim   \begin{minipage}{0.3\textwidth}
      \includegraphics[width=0.8\linewidth]{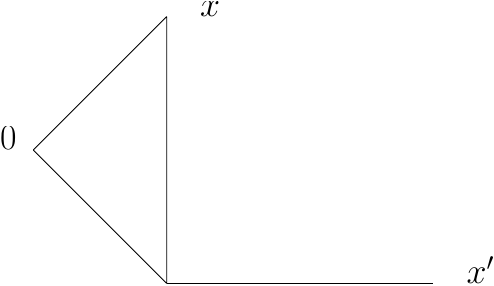}
    \end{minipage}
     \end{align*}
where we simplified the first diagram by noticing that the bold part (which is 3-connected to the circled points, by Proposition~\ref{expand_class} and Remark~\ref{ex_3strong}) reduces to a triangle by Proposition~\ref{kill_bill} and the second diagram is easily reduced by the two H-reductions highlighted in bold. The result follows from Lemma~\ref{lem_triangle_plus}.
    \end{proof}

We will now list several lemmas. Their proof is the same: the diagrams to be bounded can be obtained from the two base diagrams given in  Remark~\ref{ex_3strong} by gluing onto them one or several of the graphs 3 graphs in Remark~\ref{cone}. This means by Proposition~\ref{expand_class} that those graphs are 3-connected to the highlighted points $\{0,u,u'\}$ (or sometimes $\{0,u,x'\}$). In particular, 
by Proposition~\ref{kill_bill} it means that those diagrams are bounded by a triangle.

\begin{lemma}\label{please_stop}
We have
\[
                \begin{minipage}{0.3\textwidth}
      \includegraphics[width=0.6\linewidth]{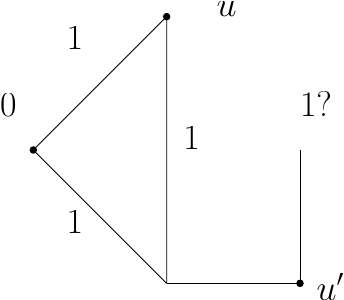}
    \end{minipage}\ 
    \lesssim 
\tau(u)\tau(u')\tau(u-u')
\]
and 
\[
                \begin{minipage}{0.3\textwidth}
      \includegraphics[width=\linewidth]{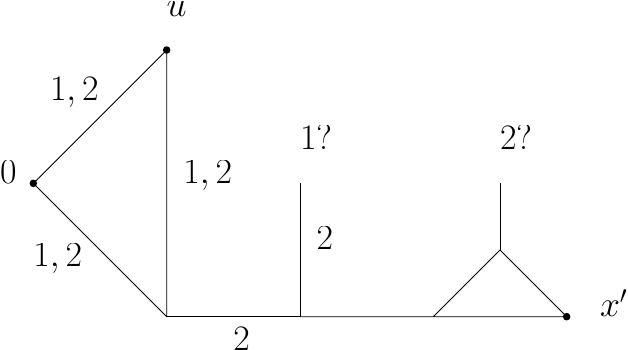}
    \end{minipage}\ 
    \lesssim 
\tau(u)\tau(u')\tau(u-u')
\]
\end{lemma}

\begin{lemma}\label{diag10}
The following four diagrams are bounded by $C\tau(u)\tau(x')\tau(u-x')$ for some constant $C<\infty$:
\begin{align*}
 \begin{minipage}{0.3\textwidth}
      \includegraphics[width=0.7\linewidth]{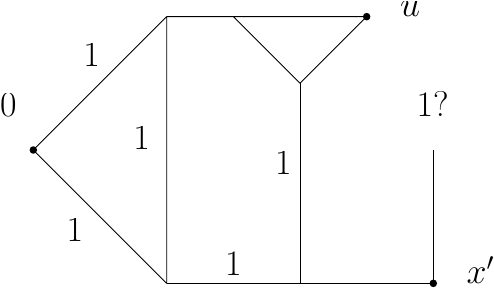} 
    \end{minipage}
& 
\begin{minipage}{0.3\textwidth}
      \includegraphics[width=0.7\linewidth]{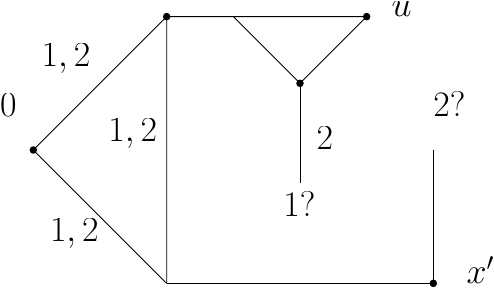}
    \end{minipage} \\
    \begin{minipage}{0.3\textwidth}
      \includegraphics[width=0.7\linewidth]{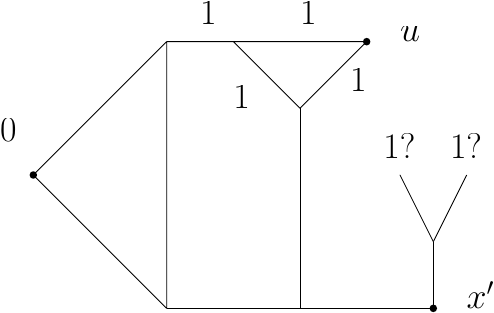} 
    \end{minipage} 
& 
\begin{minipage}{0.3\textwidth}
      \includegraphics[width=0.7\linewidth]{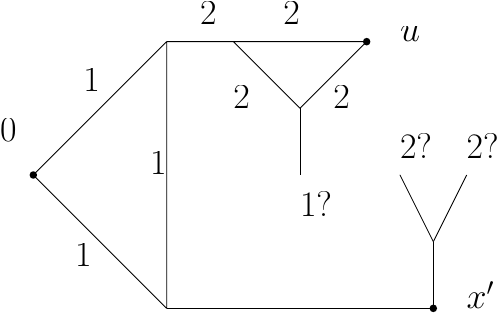}
    \end{minipage}
    \end{align*}
\end{lemma}

\begin{lemma}\label{diag11}
The following eight diagrams are bounded by $C\tau(u)\tau(u')\tau(u-u')$  for some constant $C<\infty$. 
\begin{align*}
 \begin{minipage}{0.3\textwidth}
      \includegraphics[width=0.7\linewidth]{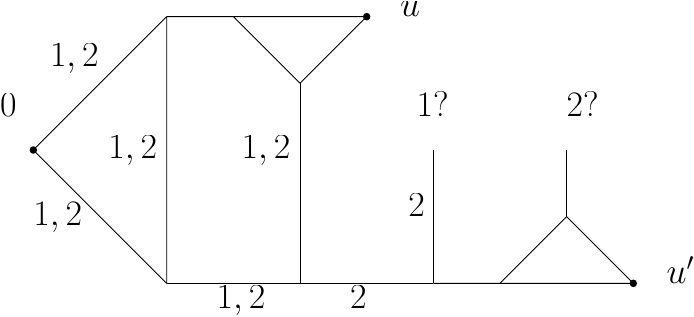} 
    \end{minipage}
& 
\begin{minipage}{0.3\textwidth}
      \includegraphics[width=0.7\linewidth]{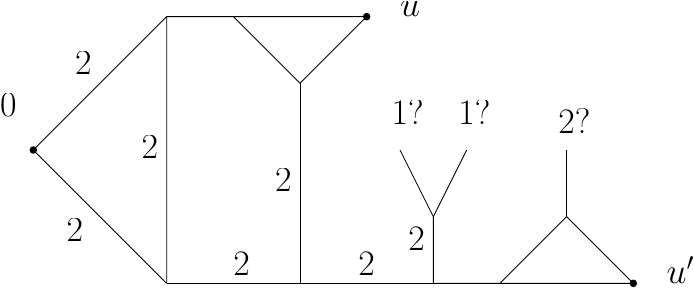}
    \end{minipage}  \\
    \begin{minipage}{0.3\textwidth}
      \includegraphics[width=0.7\linewidth]{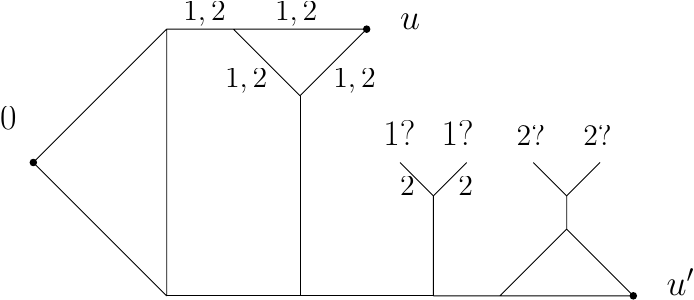} 
    \end{minipage} 
& 
\begin{minipage}{0.3\textwidth}
      \includegraphics[width=0.7\linewidth]{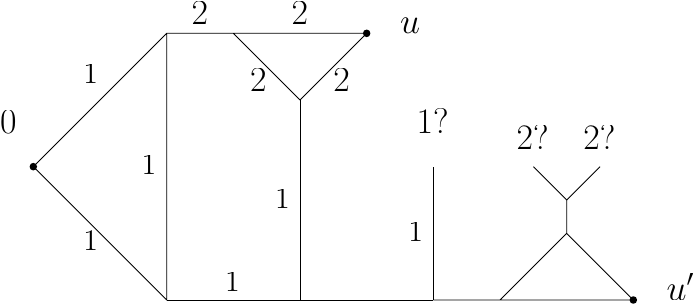}
    \end{minipage}  \\
    \begin{minipage}{0.3\textwidth}
      \includegraphics[width=0.7\linewidth]{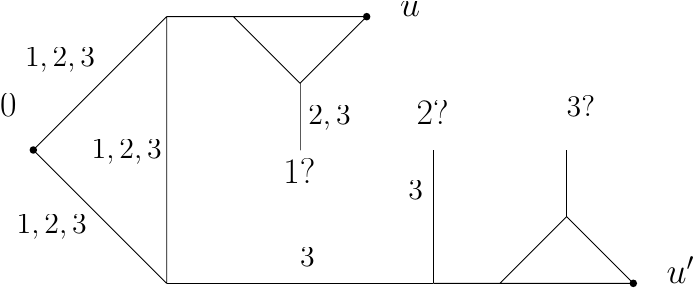} 
    \end{minipage} 
& 
\begin{minipage}{0.3\textwidth}
      \includegraphics[width=0.7\linewidth]{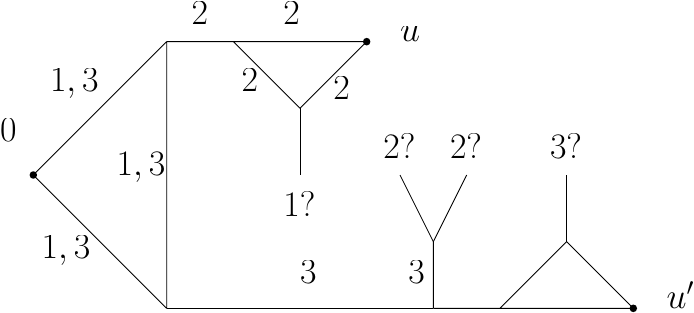}
    \end{minipage}  \\
    \begin{minipage}{0.3\textwidth}
      \includegraphics[width=0.7\linewidth]{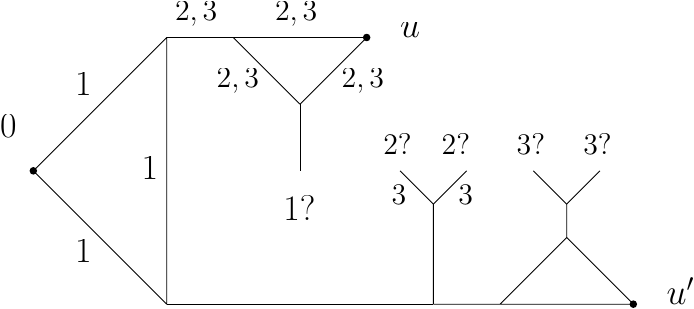} 
    \end{minipage} 
& 
\begin{minipage}{0.3\textwidth}
      \includegraphics[width=0.7\linewidth]{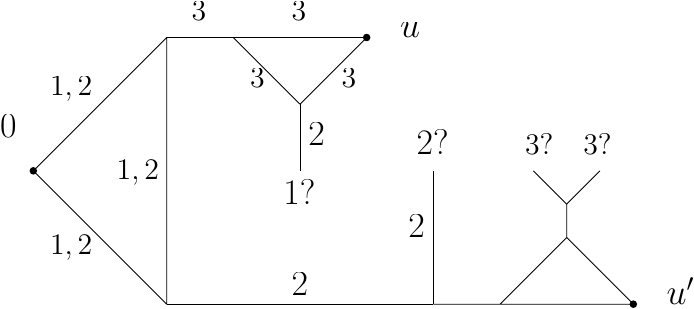}
    \end{minipage}
    \end{align*}   
\end{lemma}

\begin{lemma}\label{diag12}
The following four diagrams are bounded by $C\tau(u)\tau(u')\tau(u-u')$  for some constant $C<\infty$. 
\begin{align*}
 \begin{minipage}{0.3\textwidth}
      \includegraphics[width=0.7\linewidth]{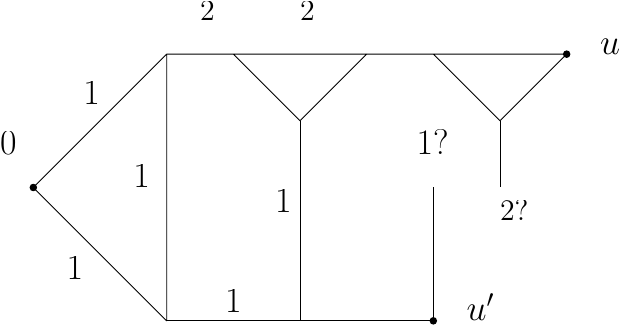} 
    \end{minipage}
& 
\begin{minipage}{0.3\textwidth}
      \includegraphics[width=0.7\linewidth]{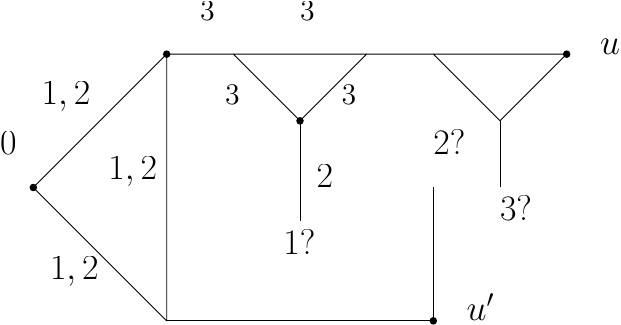}
    \end{minipage} \\
    \begin{minipage}{0.3\textwidth}
      \includegraphics[width=0.7\linewidth]{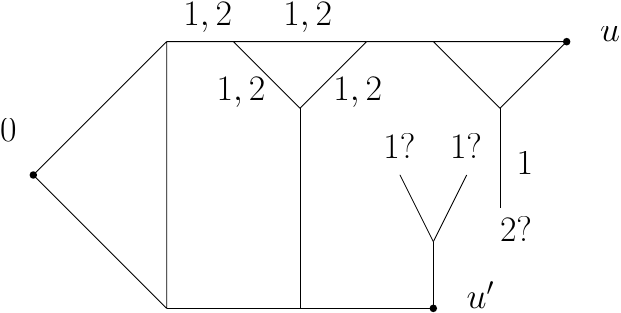} 
    \end{minipage} 
& 
\begin{minipage}{0.3\textwidth}
      \includegraphics[width=0.7\linewidth]{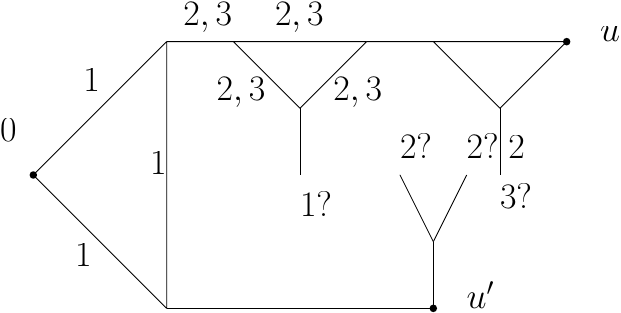}
    \end{minipage} 
    \end{align*}
\end{lemma}

Finally, the most complex diagrams we will need:
   \begin{lemma}\label{diag13}
The following eight diagrams are bounded by $C\tau(u)\tau(u')\tau(u-u')$  for some constant $C<\infty$: 
\begin{align*} 
 \begin{minipage}{0.3\textwidth}
      \includegraphics[width=0.7\linewidth]{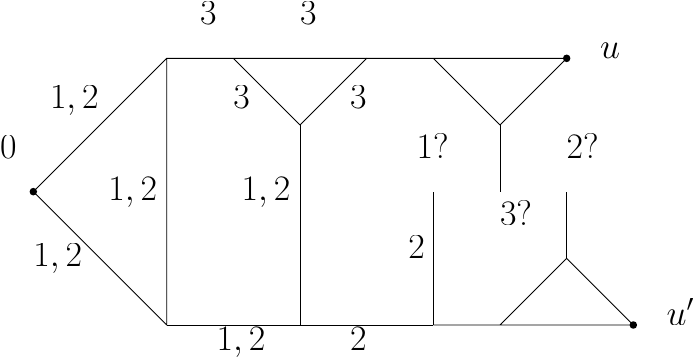} 
    \end{minipage}
& 
\begin{minipage}{0.3\textwidth}
      \includegraphics[width=0.7\linewidth]{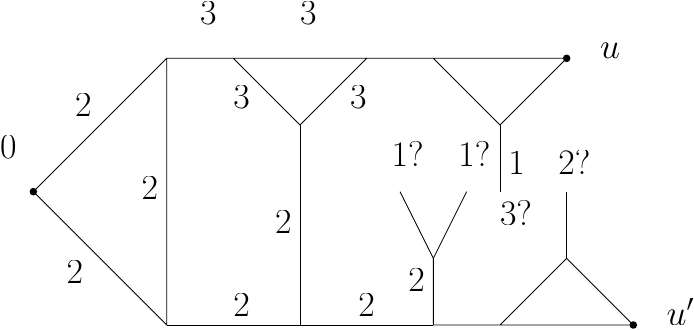}
    \end{minipage}  \\
    \begin{minipage}{0.3\textwidth}
      \includegraphics[width=0.7\linewidth]{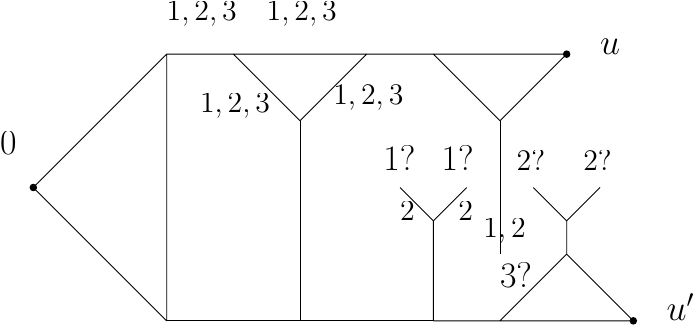} 
    \end{minipage} 
& 
\begin{minipage}{0.3\textwidth}
      \includegraphics[width=0.7\linewidth]{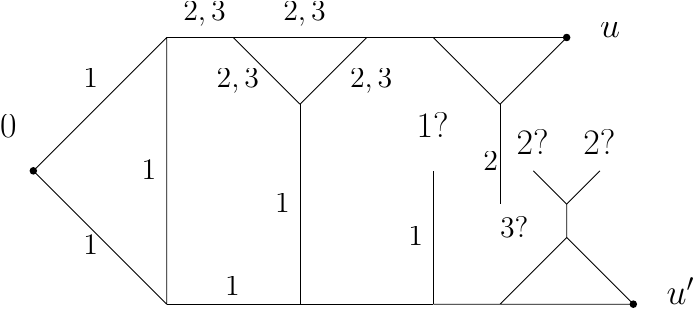}
    \end{minipage}  \\
    \begin{minipage}{0.3\textwidth}
      \includegraphics[width=0.7\linewidth]{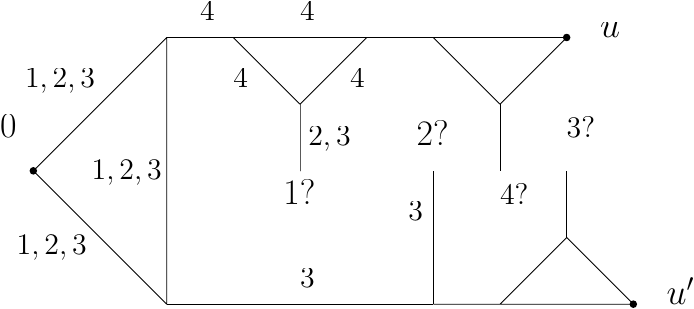} 
    \end{minipage} 
& 
\begin{minipage}{0.3\textwidth}
      \includegraphics[width=0.7\linewidth]{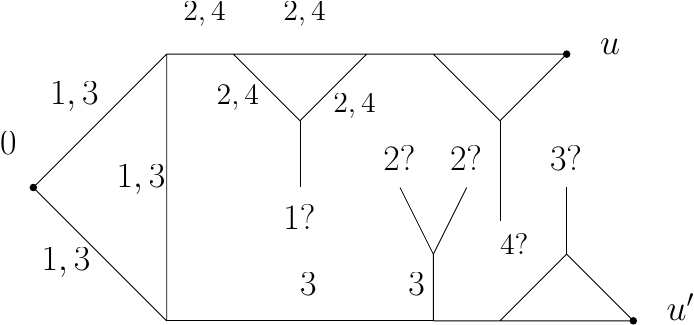}
    \end{minipage}  \\
    \begin{minipage}{0.3\textwidth}
      \includegraphics[width=0.7\linewidth]{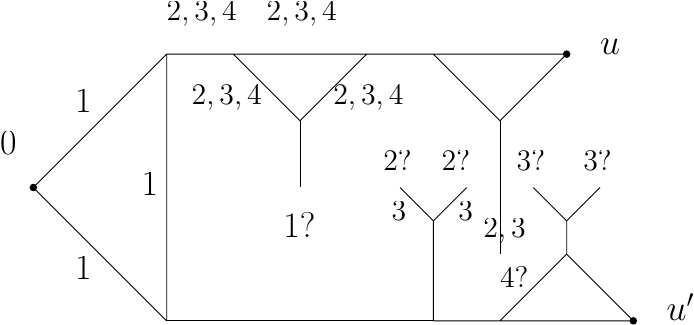} 
    \end{minipage} 
& 
\begin{minipage}{0.3\textwidth}
      \includegraphics[width=0.7\linewidth]{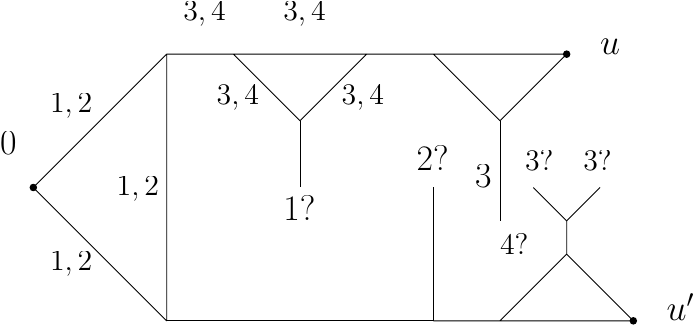}
    \end{minipage}
    \end{align*}   
\end{lemma}

\begin{remark}
One can notice, by simply counting the number of ways for the connecting edges to attach to other edges, that the seventh diagram  contains at least $3\times 5^2 \times 7^2 \times 4=14700$ standard lace expansion diagrams. This illustrates the compactness of the generalized diagram notation and the strength of the overall method. 
\end{remark}

\section{Bounding expansion coefficients in terms of  generalized diagrams}
\label{sec:bounding-diagrams}

{\colAJ{Recall that the underlying graph of our percolation model has vertex set $\Zd$ and edge set 
$\cE_D = \{ \{x,y\} \in V \times V : D(x,y) > 0 \}$.
For a set of edges $E \subset \cE_D$, we denote by $V(E)$ the set of vertices of the subgraph 
induced by $E$ (i.e., all vertices that are endpoints of an edge in $E$).}}

{\col 
Bounding the expansion coefficients $\Psi^{N,N'}$ amounts to identifying existing paths in the events that constitute these coefficients. We will now prove various lemmas to identify these paths in the underlying events. This will be related to the convolution operation $\circledast$  introduced in the previous section. Subsequently we use these lemmas to prove Proposition \ref{prop:PsiBd}. 
}

\subsection{Existence of \lq\lq extra paths\rq\rq~and generalized diagrams}\label{sect_extra_path}
In order to ``dissect'' the events underlying $\Psi$ and related quantities, we provide the following lemma. 
\begin{lemma}\label{lem_extra_path}
{\col Let $x\in\Zd$, $W\subset \Zd$, let {\colAJ{$E \subset \cE_D$}} and $z\in V(E)$. Then }
\begin{align*}
& \Bigl(\{E\text{ is occupied}\}\circ \{x \cnctd z\} \Bigr)\cap \{W \Leftrightarrow x\}\\
\subset &  \Bigl(\{E\text{ is occupied}\}\circ \{x \cnctd z\} \circ \{x \cnctd E\}\Bigr)   \bigcup \Bigl(\{E\text{ is occupied}\}\circ \{x \cnctd  z\} \circ \{W  \cnctd  x\}\Bigr).
\end{align*}
\end{lemma}
\begin{proof}
On the event $\{W \Leftrightarrow x\}$, we may choose two edge-disjoint paths $\pi_1$ and $\pi_2$ starting from $x$ and ending in $W$. Now, consider $\pi_1'$  (resp.~$\pi_2'$) the oriented path coinciding with $\pi_1$ (resp.~$\pi_2$) but ended at the first time when $\pi_1$ (resp.~$\pi_2$) reaches $V(E)$. This yields a shorter path except when $\pi_1$ (resp.~$\pi_2$) does not intersect $V(E)$ in which case we simply have $\pi_1=\pi_1'$ ($\pi_2=\pi_2'$).

Let us denote  $\mathcal{P}_{x}$ a simple path, edge-disjoint from $E$ which connects $z$ to $x$. Set $x_*$ the first vertex in the path $\mathcal{P}_{x}$ (so the closest to $z$ along that path) which belongs to either $\pi_1'$ or $\pi_2'$. Note that $x_*$ is well defined since the set of vertices in the path $\mathcal{P}_{x}$ which belongs to either $\pi_1'$ or $\pi_2'$ is non-empty because   $x$ is such a point. We can assume that $x_*$ belongs to $\pi_1'$ (otherwise just change the roles of $\pi_1'$ and $\pi_2'$).

Define $\mathcal{P}_1$ the path coinciding with $\mathcal{P}_{x}$ from $z$ to  the first (and only) visit to $x_*$ and then following $\pi_1'$ in reverse order. We know that 
\begin{itemize}
\item  $\mathcal{P}_1$ is  a path from  $z$  to $x$.
\item $\mathcal{P}_1$  is edge disjoint from $\pi_2'$. Indeed, the initial part of  $\mathcal{P}_1$ coincides with  $\mathcal{P}_{x}$ before reaching $x_*$ and, by definition of $x_*$, no point before $x_*$ is on $\pi_1'$ or $\pi_2'$. Furthermore the latter part of  $\mathcal{P}_1$ is included in $\pi_1'$ which is edge disjoint from $\pi_2' $ (since $\pi_1$ and $\pi_2$ are).
\item  $\mathcal{P}_1$  is edge-disjoint from $E$. Indeed, its initial part is included in $\mathcal{P}_{x}$ which is edge-disjoint from $E$ and its latter part is included in $\pi_1'$ which was defined as the part of $\pi_1$ reached (from $x'$) before reaching $V(E)$; hence obviously edge disjoint from $E$.
\end{itemize}

Finally, we can see that $\pi_2'$ is edge-disjoint from $E$ and  starts at $x$. There are two cases, 
\begin{enumerate}
\item  $\pi_2$ reaches $V(E)$, in which case $\pi_2'$ ends  in $V(E)$ and it implies the occurrence of
\[
 \{E\text{ is occupied}\}\circ \{x \text{ is connected (using $\pi_2'$) to }E\} \circ \{x \text{ is connected (using $\mathcal{P}_1$) to }z\},
 \]
\item $\pi_2$ does not reach $V(E)$, thus $\pi_2'$  coincides with $\pi_2$ and ends in $W$ by definition. This implies the occurrence of
\[
\{E\text{ is occupied}\}\circ \{x \text{ is connected (using $\pi_2'$) to }W\} \circ \{x \text{ is connected (using $\mathcal{P}_1$) to }z\},
 \]
\end{enumerate}
This shows the lemma. 
\end{proof}

The previous lemma has an important consequence in terms of diagrammatic graphs.
\begin{lemma}\label{lem_extra_path2}
Let $x,z\in\Zd$, $W\subset \Zd$. 
For a diagrammatic graph $\mathbf G$ that contains the label $z$, we have 
\[
(\mathcal{D}_{\mathbf  G}\circ \{z \leftrightarrow x\} )\cap \{W \Leftrightarrow x\} \subset \bigcup_{\substack{{\bf G'}\in {\bf G} \circledast x \\ \text{ or }{\bf G'} ={\bf G}\uplus \{W, x\}}} (\mathcal{D}_{{\bf G'}} \circ \{z \leftrightarrow x\}).
\]
\end{lemma}
{\colAJ{We note immediately that the terms arising from ${\bf G}\uplus \{W,x\}$ will not concern us: 
their contribution
will be bounded above by $C_W$ times the contribution of the terms arising from ${\bf G} \circledast x$, 
by Proposition \ref{neglect_w2}. Hence the reader may safely concentrate only on the latter type of 
diagrammatic graphs. See also Remark \ref{neglect_W} below.}}

\begin{proof}[Proof of Lemma \ref{lem_extra_path2}.]
{\col The statement is trivial if $z=x$.} 
{\colAJ{We will make use of the simple fact that if $I$ is any index set and $A_\alpha$, $\alpha \in I$ 
and $B$ are increasing events then 
\eqn{e:observation} 
{ \left( \cup_{\alpha \in I} A_\alpha \right) \circ B
   = \cup_{\alpha \in I} \left( A_\alpha \circ B \right). }
Let us enumerate the edges of $\mathbf{G}$ as $(\{x_i,y_i\})_{i\leq n}$. Let $f\colon V(G) \to \Zd$ be a 
labeling in the definition of $\mathcal{D}_{\bf G}$, and let $(\pi_i)_{i \le n}$ be any choice of 
edge-disjoint paths with $\pi_i$ connecting $f(x_i)$ and $f(y_i)$. Let $E$ be the set of edges
$\cup_{i \le n} \pi_i$. Thus, $\mathcal{D}_{\bf G}$ is the union of the events $\{ E \text{ is occupied}\}$
over all possible choices of $E$ in the above.
Note that $z\in V(E)$ by assumption. 

We apply Lemma~\ref{lem_extra_path} to $E$, which yields that 
\begin{align*}
& ( \{ E \text{ is occupied}\}\circ \{z \leftrightarrow x\} )\cap \{W \Leftrightarrow x\} \\
& \subset \Bigl(\{E \text{ is occupied}\}\circ \{x \cnctd z\} \circ \{x\cnctd E\}\Bigr) \bigcup \Bigl(\{E \text{ is occupied}\}\circ \{x \cnctd  z\} \circ \{W  \cnctd  x\}\Bigr).
\end{align*}
It is clear that 
\[
\{E \text{ is occupied}\}\circ \{x \cnctd  z\} \circ \{W  \cnctd  x\}\subset \mathcal{D}_{\bf G} \circ \{z \leftrightarrow x\} \circ\{W \leftrightarrow x\}
=\mathcal{D}_{{\bf G}\uplus \{W, x\}} \circ \{z \leftrightarrow x\} .
\] 
Furthermore, on the event $\{E \text{ is occupied}\} \circ \{x \cnctd E\}$, there exists $i_0\leq n$ and $x_*$ in the vertices of $\pi_{i_0}$ such that $x$ is connected to $x_*$  without using edges of $E$. This means that event 
$ \left( \bigcirc_{i\neq i_0} \{f(x_i)\leftrightarrow f(y_i)\} \right)
   \circ \{f(x_{i_0}) \leftrightarrow x_*\} \circ  \{f(y_{i_0}) \leftrightarrow x_*\} 
   \circ  \{x \leftrightarrow x_*\} $
occurs. This implies
\[
\{E \text{ is occupied}\} \circ \{x \cnctd E\} \circ \{z \leftrightarrow x\} 
\subset \bigcup_{{\bf G'}\in {\bf G} \circledast x} (\mathcal{D}_{{\bf G'}}  \circ \{z \leftrightarrow x\}).
\]
Taking the union over all choices of $E$, the result follows from the last three equations and 
\eqref{e:observation}.}}
\end{proof}

\begin{lemma}\label{lem_extra_path3}
Let $x\in\Zd$, $W\subset \Zd$, let {\colAJ{$E \subset \cE_D$}} and $z\in V(E)$. Then 
\begin{align*}
\{E\text{ is occupied}\}\cap \{W \leftrightarrow y\} 
\subset ( \{E\text{ is occupied}\}\circ \{y \cnctd E\} )
\cup (\{E\text{ is occupied}\}\circ \{W \leftrightarrow y\}).
\end{align*}
\end{lemma}
\begin{proof}
On $\{E\text{ is occupied}\}\cap \{W \leftrightarrow y\}$, we can choose a path $\pi$ from $y$ to $W$. Consider $\pi'$ the path coinciding with $\pi$ until  the first point $x_*$ where  this path meets $V(E)\cup W$. By definition, $\pi'$ is edge-disjoint from $E$, starts at $y$ and ends  in $V(E)$ or $W$, we have
\begin{align*}
& \Bigl(\{E\text{ is occupied}\}\cap \{W \leftrightarrow y\}\\
\subset &  \Bigl(\{E\text{ is occupied}\}\circ \{\pi' \text{ connects } y \text{ to }E\}\Bigr)  \bigcup \Bigl(\{E\text{ is occupied}\}\circ\{\pi' \text{ connects } y \text{ to }  W\}\Bigr),
\end{align*}
and the lemma follows.
\end{proof}

Using a proof similar to that of Lemma~\ref{lem_extra_path2}, the previous lemma allows us to obtain
\begin{lemma}\label{lem_extra_path4}
Let $W$ be a {\colAJ{finite subset of $\Zd$ containing $0$}} and also set $y\in \mathbb{Z}^d$. 
For $\bf G$ a diagrammatic graph that contains the label $0$, we have
\[
\mathcal{D}_{\bf G}\cap \{W \leftrightarrow y\} \subset  \bigcup_{\substack{{\bf G'}\in {\bf G} \circledast y \\ \text{ or }{\bf G'} ={\bf G}\uplus \{W, y\}}} \mathcal{D}_{{\bf G'}}.
\]
\end{lemma}

{\colAJ{The observation contained in the following important remark will be used repeatedly, and often
without further comment.}}
\begin{remark}\label{neglect_W}
Using Proposition~\ref{neglect_w2}, we see that, using the BK inequality, Lemma~\ref{lem_extra_path2} implies 
\[
{\bf P}[(\mathcal{D}_{\bf G} \circ \{z \leftrightarrow x\} )\cap \{W \Leftrightarrow x\} ] 
 \leq {\colAJ{ (1+C_W) \mathrm{Diag} (({\bf G} \circledast x) \uplus \{z,x\} ), }}
\]
where the right-hand side might be viewed as a generalized diagram (here $C_W$ denotes the same constant as in Proposition~\ref{neglect_w2}).

We could also take the diagrammatic event, intersect it with $\{W\leftrightarrow y\}$ which would allow for the successive application of Lemma~\ref{lem_extra_path2} and Lemma~\ref{lem_extra_path4}, leading to the inclusion
\[
 ((\mathcal{D}_{\bf G} \circ\{z \leftrightarrow x\}) \cap \{W \Leftrightarrow x\})\cap \{W \leftrightarrow y\} \\    \subset  \bigcup_{\substack{{\bf G'}\in {\bf G} \circledast x \\ \text{ or }{\bf G'} ={\bf G}\uplus\{W,x \}}} \bigcup_{\substack{{\bf G''} \in ({\bf G'}\uplus \{ z, x\}) \circledast y \\\text{ or }{\bf G''} =({\bf G'}\uplus \{ z, x \}) \uplus \{W, y\}}}\mathcal{D}_{{\bf G''}}.
\]
where we can now apply Proposition~\ref{neglect_w2} recursively (together with the BK inequality) and obtain 
\[
{\bf P}[((\mathcal{D}_{\bf G} \circ\{z \leftrightarrow x\}) \cap \{W \Leftrightarrow x\})\cap \{W \leftrightarrow y\}  ] 
 \leq {\colAJ{ (1+C_W)^2 \mathrm{Diag} ( ( ({\bf G} \circledast x) \uplus \{ z, x\}) \circledast y ), }}
\] 
which proves that all the terms arising from the box $W$ are irrelevant, up to multiplicative constants, 
for upper bounds on probabilities.
\end{remark}

In the same vein {\colAJ{as Lemmas \ref{lem_extra_path2} and \ref{lem_extra_path4} we also prove 
the next statement.}}
\begin{lemma}\label{sam}
{\colAJ{Let $\textbf G$ be a diagrammatic graph which contains the labels $u, v$, and let $b \in \Zd$.}} Then 
\[
(\{v \leftrightarrow b\} \circ \{b\leftrightarrow u\}) \cap \mathcal{D}_{{\bf G}} 
 \subset \bigcup_{{\bf G'}\in G \circledast^2 b} \mathcal{D}_{{\bf G'}}.
\]
\end{lemma}

\begin{proof}
 Let $(\pi_i)_{i\leq K}$ denote {\colAJ{occupied}} paths realizing the event $\mathcal{D}_{{\bf G}}$. Let us denote $V$ the vertices contained in those paths.

Since $\{v \leftrightarrow b\} \circ \{b\leftrightarrow u\}$, we know that there are two disjoint paths $\pi_1'$ and $\pi_2'$ starting from $b$ with the first connecting to $v$ and the second to $u$. Since those two vertices are in $V$, we know that both $\pi_1'$ and $\pi_2'$ have to intersect $V$. Consider $\pi_1''$ (resp.~$\pi_2''$) the path coinciding with $\pi_1'$ (resp.~$\pi_2'$) until  the first point $x_*$ where  $\pi_1'$ (resp.~$\pi_2'$) meets $V$. 

By definition, $\pi''_1$  and $\pi_2''$ are edge-disjoint and $\pi''_1$ (resp.~$\pi_2''$) is edge-disjoint from the paths $(\pi_i)_{i\leq K}$, starts at $b$ and ends in $V$. Hence the paths $\pi''_1$, $\pi_2''$ and $(\pi_i)_{i\leq K}$ realize the diagrammatic events linked to the generalized diagram in the statement of the lemma. The result follows.
\end{proof}

\subsection{Bounds on expansion coefficients}
We now bound the events appearing in the expansion in Section \ref{sec:expansion}. 
\begin{lemma}
\label{lem:E-bnd}\ 
{\col Let $F$ be a special cylinder event in the finite window $W\subset \Zd$.} 
\\
(i) For all $x, x' \not\in W$ we have 
\[
E(F;x,x') \subset  \bigcup_{\substack{{\bf G'}\in \{0,x'\} \circledast x \\ \text{ or }{\bf G'} =\{0,x'\}\uplus\{W,x \}}}{\bf G'} \circ \{ o \cnctd x \}.
\]
(ii) For all $\{ u_0, v_0 \} \not\in E_W$, $x' \not\in W$ we have 
\[ 
\widetilde{E}(F;u_0,v_0,x') 
   \subset E(F,u_0,x'). 
   \]
\end{lemma}

\begin{proof}
Part (ii) is immediate from the definition of the events, by removing the restriction that the edge 
$\{ u_0, v_0 \}$ should not be used.


For part (i), we know that $E(F;x,x')=\bigl(\{ o \cnctd x \} \circ \{ o \cnctd x' \}\bigr) \cap \{ W \Leftrightarrow x \}$ which allows us to use Lemma~\ref{lem_extra_path2} with
{\colAJ{the diagrammatic event $\{ o \cnctd x' \}= \mathcal{D}_{\{o,x'\}}$ which includes label $o$}}. 
\end{proof}

\begin{lemma}
\label{lem:tPsiBd0}
For any $F\in\mathfrak F_{00}$ there exists $C>0$ such that for all $x,x'\in\Zd$ we have 
\[ \tPsi^{(0)}(F;x,x') \le C\tau(x)\tau(x')  \min\{|x|,|x-x'|\}^{-(d-4)} .\]
\end{lemma}
\begin{proof}
From Lemma \ref{lem:E-bnd}(i) and the BK-inequality, we get
\begin{align*}
  \tPsi^{(0)}(F;x,x') & \leq {\bf P}\Bigl[ \bigcup_{\substack{{\bf G'}\in \{0,x'\} \circledast x \\ \text{ or }{\bf G'} =\{0,x'\}\uplus{\{W, x\}}}} \mathcal D_{{\bf G'}} \circ \{ o \cnctd x \}\Bigr]  \\
  &  \leq  C_W   \begin{minipage}{0.18\textwidth}
      \includegraphics[width=0.9\linewidth]{part0-eps-converted-to.pdf}
    \end{minipage}\
    \end{align*}
where we used Proposition~\ref{neglect_w2}. The result follows from Lemma~\ref{lem_triangle_plus}.
\end{proof}

We next bound $\sum_{N \ge 1} \tPsi^{(N)}(F;x,x')$. 
In what follows, we will often use the next lemma that is a direct consequence of
\cite[Lemma 2.5]{HaraSlade90a} and the proof of \cite[Proposition 2.4]{HaraSlade90a}.

{\colAJ{
\begin{lemma}
\label{lem2.5HS} \ \\
(i) For any $S \subset \Zd$ and $v_0, x \in \Zd$, we have
\eqnsplst{
\E_1 \left[ E'(v_0,x;S) \right] 
   &\le \sum_{y_1, z_1} I[ y_1 \in S ] \tau(v_0,z_1) \tau(z_1,y_1) \tau(y_1,x) \tau(z_1,x) \\
   &= \sum_{y_1, z_1} I[ y_1 \in S ] \tau(v_0,z_1) A^{(1)}(y_1,z_1,x). }
where we recall $A^{(1)}$ from~\eqref{def_A}. \\
(ii) For any $S \subset \Zd$ and $v_0, u_1, v_1, x \in \Zd$, we have
\[ \sum_{u_1,v_1} \E_1 \left[ E'(v_0,u_1;S) \Xi^{(N)}(v_1,x;\tC_1) \right] 
   \le \sum_{y_1, z_1} I[ y_1 \in S ] \tau(v_0,z_1) A^{(N)}(y_1,z_1,x). \]
\end{lemma}
}}

\begin{lemma}\label{lem:tPsiBd}
For all $F\in\mathfrak F_{00}$ there exist $C>0$ such that for all $x,x'\in\Zd$, 
\eqnsplst{
\sum_{N \ge 1} \tPsi^{(N)}(F;x,x')
   &= C \tau(x)\tau(x') \min\{|x|,|x'|,|x-x'|\}^{-(d-4)}.}
\end{lemma}
Together with the foregoing lemma, this proves \eqref{eq:momentbound}. 
\begin{proof}
First we consider the term $N=1$. 
{\colAJ{Recalling \eqref{eqDefPsi1} we estimate using Lemma \ref{lem2.5HS}(i), 
conditional on $\omega_0$,
\eqnsplst{
\E_1 \left[ E'(v_0,x;\tC_0) \right] 
   &= \sum_{y_1, z_1} I[ y_1 \in \tC_0 ] \tau(v_0,z_1) A^{(1)}(y_1,z_1,x). }
In the case $N \ge 2$, we similarly get using Lemma \ref{lem2.5HS}(ii) that
\[ \sum_{u_1,v_1} \E_1 \left[ E'(v_0,u_1;\tC_0) \Xi^{(N)}(v_1,x;\tC_1) \right] 
   \le \sum_{y_1, z_1} I[ y_1 \in \tC_0 ] \tau(v_0,z_1) A^{(N)}(y_1,z_1,x). \]

Now, recalling~\eqref{e:tPsiN-short}, we see that
 \begin{align}\label{bla}
 & \sum_{N \ge 1} \tPsi^{(N)}(F;x,x') \\ \nonumber
  \leq &\sum_{u_0,v_0} \sum_{y_1,z_{1}}  pD(u_0,v_0)\E_0 \Big[ I \Bigl[ \widetilde{E}(F;u_0,v_0,x') \Bigr]  I[ y_1 \in \tC_0 ]\tau(v_0,z_1)\sum_{N\geq 1} A^{(N)}(y_1,z_1,x) \Big] \\\nonumber
  \leq & C\sum_{u_0,v_0} \sum_{y_1,z_{1}}  pD(u_0,v_0)\E_0  \Big[ I \Bigl[ E(F;u_0,x') \Bigr] I[ y_1 \in \tC_0 ]\Big]  \tau(v_0,z_1)\tau(z_1,y_1) \tau(y_1,x) \tau(z_1,x)
 \end{align}
where we used Lemma \ref{lem:E-bnd}(ii) and that \eqref{eq:Aass} implies
\[
\sum_{N\geq 1} A^{(N)}(y_1,z_1,x)
   \le C \tau(z_1,y_1) \tau(y_1,x) \tau(z_1,x).
    \]  
Using \eqref{e:D*tau}, we can perform the summation over $v_0$ and get
\begin{align}\label{bla+}
 & \sum_{N \ge 1} \tPsi^{(N)}(F;x,x') 
 \leq C \sum_{u_0,y_1,z_{1}} \E_0  \Big[ I \Bigl[ E(F;u_0,x') \Bigr] I[ y_1 \in \tC_0 ]\Big]  \tau(u_0,z_1)\tau(z_1,y_1) \tau(y_1,x) \tau(z_1,x)
 \end{align}
}}
Using that $\{y_1 \in \tC_0 \} \subset \{W \leftrightarrow y_1\}$, we are led to bounding 
  \begin{equation}\label{ship}
   E(F;u_0,x') \cap \{ W \leftrightarrow y_1 \}
   \subset \{ (\{0\leftrightarrow u_0\} \circ\{0 \leftrightarrow x'\}) \cap \{W \Leftrightarrow u_0\})\cap \{W \leftrightarrow y_1\}.
   \end{equation}
 
Use Lemma~\ref{lem_extra_path2} with the graph $\{0,x'\}$ yields 
\[ 
(\{0\leftrightarrow u_0\} \circ\{0 \leftrightarrow x'\}) \cap \{W \Leftrightarrow u_0\}) \subset    \bigcup_{\substack{{\bf G'}\in \{0,x'\} \circledast u_0 \\ \text{ or }{\bf G'} =\{0,x'\}\uplus\{W,u_0 \}}} \mathcal{D}_{{\bf G'}} \circ \{ 0 \cnctd u_0 \}
\]
  and by using Lemma~\ref{lem_extra_path4} 
  \begin{align}\label{ship1}
 & ((\{0\leftrightarrow u_0\} \circ\{o \leftrightarrow x'\}) \cap \{W \Leftrightarrow u_0\})\cap \{W \leftrightarrow y_1\} \\ \nonumber
 \subset &  \bigcup_{\substack{{\bf G'}\in \{0,x'\} \circledast u_0 \\ \text{ or }{\bf G'} =\{0,x'\}\uplus\{W,u_0 \}}} (\mathcal{D}_{{\bf G'}} \circ \{ 0 \cnctd u_0 \}) \cap \{W \leftrightarrow y_1 \} \\ \nonumber
   \subset & \bigcup_{\substack{{\bf G'}\in \{0,x'\} \circledast u_0 \\ \text{ or }{\bf G'} =\{0,x'\}\uplus\{W,u_0 \}}} \bigcup_{\substack{{\bf {\bf G''}}\in ({\bf G'}\uplus \{ 0, u_0 \}) \circledast y_1 \\\text{ or }{\bf {\bf G''}} =({\bf G'}\uplus \{ 0, u_0 \}) \uplus \{W, y_1\}}}\mathcal{D}_{{\bf {\bf G''}}},
  \end{align}
and by Remark~\ref{neglect_W} we have
\[
{\bf P}[ ((\{0\leftrightarrow u_0\} \circ\{o \leftrightarrow x'\}) \cap \{W \Leftrightarrow u_0\})\cap \{W \leftrightarrow y_1\} ] \leq   C_W          \begin{minipage}{0.18\textwidth}
      \includegraphics[width=0.9\linewidth]{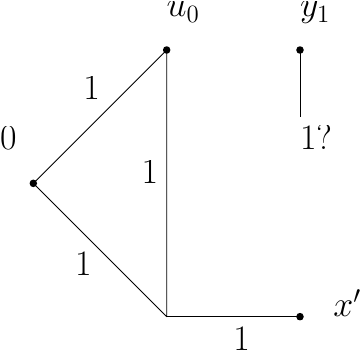}
    \end{minipage}\
    \]
  
Coming back to~\eqref{bla+}, we can include $\tau(u_0,z_1)\tau(z_1,y_1) \tau(y_1,x) \tau(z_1,x)$ 
into our diagrams to see that
\[
             \begin{minipage}{0.4\textwidth}
$\sum_{N \ge 1} \tPsi^{(N)}(F;x,x') \leq C_W$
    \end{minipage}\
                \begin{minipage}{0.4\textwidth}
      \includegraphics[width=0.7\linewidth]{last1-eps-converted-to.pdf}
    \end{minipage}\
    \]
    and the lemma follows from Lemma~\ref{extra_diag}.
\end{proof}

\subsection{Diagrammatic bounds for the terms arising in the second expansion}
\label{ssec:error-2nd}

\subsubsection{The $N=0$ terms of the double expansion}

{\colAJ{We start with the simplest case $N = 0$. The following is a preparatory lemma.}}

\begin{lemma}
\label{lem:E''-bnd}
We have
\begin{align*}
& E^{'',0}(F;u,v,x') \\
      \subset & \Bigl(\bigcup_{y\in \mathbb{Z}^d} \bigcup_{\substack{{\bf G'}\in K_{0,y,u} \circledast x' \\ \text{ or }{\bf G'} =K_{0,y,u}\uplus \{W, x'\}}} \mathcal{D}_{{\bf G'}} \circ \{y \leftrightarrow x'\} \Bigr) 
      \cup \Bigl( \bigcup_{\substack{{\bf G'}\in \{0,u\}\uplus\{W,u\} \circledast x' \\ \text{ or }{\bf G'} =\{0,u\}\uplus\{W,u\}\uplus \{W, x'\}}} \mathcal{D}_{{\bf G'}} \circ \{0 \leftrightarrow x'\}\Bigr),
\end{align*}
where we recall that $K_{0,y,u}$ denotes the complete graph with vertices {\colAJ{labelled} $0,y,u$}.
\end{lemma}

Recall from Remark \ref{rem:Wsum} that the symbol $W$ represents a union over the elements of $W$. 

\begin{remark}
For better visualization, we emphasize that the generalized diagram in the first union over $y\in \mathbb{Z}^d$ and the generalized diagram appearing in the second union can be drawn as follows
\[
                \begin{minipage}{0.3\textwidth}
      \includegraphics[width=0.7\linewidth]{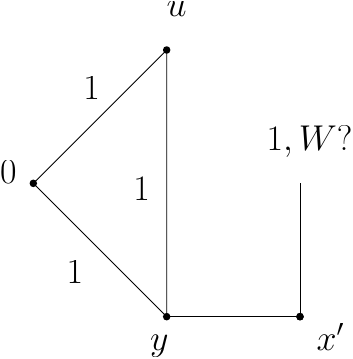}
    \end{minipage}\
    \qquad 
                    \begin{minipage}{0.3\textwidth}
      \includegraphics[width=0.5\linewidth]{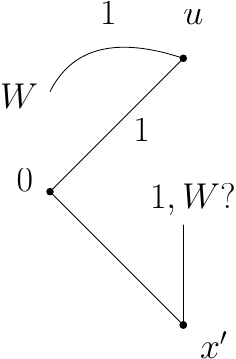}
    \end{minipage}\
    \]
    where we emphasize that Remark~\ref{neglect_W} shows that only the first diagrammatic terms (with a connection \lq\lq 1?\rq\rq and not \lq\lq 1,W?\rq\rq) are important, up to multiplicative constants, once we will move onto estimating probabilities of such events.
\end{remark}

\begin{proof}
Recalling the definition of $E^{'',0}(F;u,v,x')$ at~\eqref{def_epp}, we have
\[ 
  E^{'',0}(F;u,v,x') \subset \Bigg(\text{$\{ W \dcnctd u \} \cap 
      \Big( \{ 0 \cnctd u \} \circ \{ 0 \cnctd x' \}\Big) \Bigg) \cap \{W \Leftrightarrow x'\} $}  . 
\]
We apply Lemma~\ref{lem_extra_path2} {\colAJ{to the diagrammatic graph $\{0,x'\}$}} to get 
\[
	\{ W \dcnctd u \} \cap  \Big( \{ 0 \cnctd u \} \circ \{ 0 \cnctd x' \}\Big)\subset \bigcup_{\substack{{\bf G'}\in \{0,x'\} \circledast u \\ \text{ or }{\bf G'} =\{0,x'\}\uplus \{W, u\}}} \mathcal{D}_{{\bf G'}} \circ \{0 \leftrightarrow u\}.
\]      
and notice that 
\[
  \bigcup_{{\bf G'}\in \{0,x'\} \circledast u } \mathcal{D}_{{\bf G'}} \circ \{0 \leftrightarrow u\}= \bigcup_{y\in \mathbb{Z}^d} \mathcal{D}_{K_{0,y,u}} \circ \{y \leftrightarrow x'\}
\]
 and
  \[	\bigcup_{\substack{{\bf G'} =\{0,x'\}\uplus \{W, u\}}} \mathcal{D}_{{\bf G'}} \circ \{0 \leftrightarrow u\}
  	  =\bigcup_{{\bf G'}=\{0,u\}\uplus\{W,u\}} \mathcal{D}_{{\bf G'}} \circ \{0 \leftrightarrow x'\}.
 \]

We use the last three equations to see that
\begin{align*}
&  \Bigl(\{ W \dcnctd u \} \cap \Big( \{ 0 \cnctd u \} \circ \{ 0 \cnctd x' \}\Big) \Bigr) \cap \{W \Leftrightarrow x'\}  \\
      \subset & \Bigl(\bigcup_{y\in \mathbb{Z}^d} (\mathcal{D}_{K_{0,y,u}} \circ \{y \leftrightarrow x'\})\cap   \{W \Leftrightarrow x'\}\Bigr) 
      \cup \bigcup_{{\bf G'}=\{0,u\}\uplus\{W,u\}} \Bigl((\mathcal{D}_{{\bf G'}} \circ \{0 \leftrightarrow x'\})\cap   \{W \Leftrightarrow x'\}\Bigr),
\end{align*}
and we apply Lemma~\ref{lem_extra_path2} again {\colAJ{(this time to the diagrammatic 
graphs $K_{0,y,u}$ and ${\bf G'}=\{0,u\}\uplus\{W,u\}$, respectively)}}, getting
\begin{align*}
&  \Bigl(\{ W \dcnctd u \} \cap \Big( \{ 0 \cnctd u \} \circ \{ 0 \cnctd x' \}\Big) \Bigr) \cap \{W \Leftrightarrow x'\}  \\
      \subset & \Bigl(\bigcup_{y\in \mathbb{Z}^d} \bigcup_{\substack{{\bf G'}\in K_{0,y,u} \circledast x' \\ \text{ or }{\bf G'} =K_{0,y,u}\uplus \{W, x'\}}} \mathcal{D}_{{\bf G'}} \circ \{y \leftrightarrow x'\} \Bigr) 
      \cup \Bigl( \bigcup_{\substack{{\bf G'}\in \{0,u\}\uplus\{W,u\} \circledast x' \\ \text{ or }{\bf G'} =\{0,u\}\uplus\{W,u\}\uplus \{W, x'\}}} \mathcal{D}_{{\bf G'}} \circ \{0 \leftrightarrow x'\}\Bigr).
\end{align*}
This proves the lemma.
\end{proof}

It is convenient to introduce the abbreviation
\begin{equation}
	T(o,u,x') = \tau(o,u) \tau(o,x') \tau(u,x'). 
\end{equation}

{\colAJ{We are ready to prove the bound for the $N=0$ terms.}}

\begin{proposition}
\label{prop:n=0error} 
For all $F\in\mathfrak F_{00}$ there exist $C_W>0$ such that for all $u,v,,x\in\Zd$, 
\eqn{eq:n=0error}{
  \sum_{N'=0}^\infty \hat{\Psi}^{0,N'}(F,u,v,x')
  \le C_W T(o,u,x'). }
\end{proposition}

{\colAJ{
\begin{remark}
Observe that \eqref{eq:n=0error} also immediately provides a bound for 
$\hat{\Psi}^{0,N'}(F;u,v,u')$ present in \eqref{eq:DefBarPiNN}, by setting 
$x' = u'$.
\end{remark}
}}

\begin{proof}[Proof of Proposition \ref{prop:n=0error}.]
We start with bounding the $N'=0$ term. Recalling  Lemma \ref{lem:E''-bnd}, we apply the BK inequality with the  insight from Remark~\ref{neglect_W} to see that 
 \[
 {\bf P}[E^{'',0}(F;u,v,x')]\leq C_W              \begin{minipage}{0.18\textwidth}
      \includegraphics[width=0.8\linewidth]{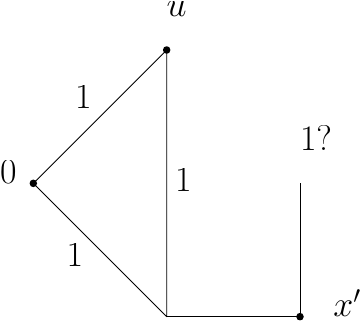}
    \end{minipage}\
    \]
    and by Lemma~\ref{please_stop} we obtain the statement for $N'=0$.

Let now $N'=1$. Consider the expression $I[E'(v'_0,x',\tD_0)]$ inside the definition of
$\hat{\Psi}^{0,1}$. This is bounded, conditionally on $\omega_0$, by
\eqnsplst{
   \E'_1 \left[ E'(v'_0,x';\tD_0) \right] 
   &\le \sum_{y'_1, z'_1} I[ y'_1 \in \tD_0 ] \tau(v'_0,z'_1) \tau(z'_1,y'_1) \tau(y'_1,x') \tau(z'_1,x') \\
   &= \sum_{y'_1, z'_1} I[ y'_1 \in \tD_0 ] \tau(v'_0,z'_1) A^{(1)}(y'_1,z'_1,x'). }
In the case $N' \ge 2$, we similarly get, {\colAJ{using Lemma \ref{lem2.5HS} and \eqref{e:D*tau},}} the bound 
\[ \sum_{u'_1, v'_1} p D(u'_1, v'_1) \E'_1 \left[ E'(v'_0,u'_1;\tD_0) \Xi^{(N')}(v'_1,x';\tC'_1) \right] 
   \le \sum_{y'_1, z'_1} I[ y'_1 \in \tD_0 ] \tau(v'_0,z'_1) A^{(N')}(y'_1,z'_1,x'). \]
Summing over $N'$, we bound {\colAJ{using \eqref{eq:Aass}}} 
\begin{equation}\label{sdf} \sum_{N' \ge 1} A^{(N')}(y'_1,z'_1,x')
   \le C \tau(z'_1,y'_1) \tau(y'_1,x') \tau(z'_1,x'). \end{equation}
Hence, also using \eqref{e:D*tau}, we have
 \begin{align}\label{bla1}
 & \sum_{N' \ge 1}\hat{\Psi}^{0,N'}(F,u,v,x')\\ \nonumber
  \leq &C\sum_{u_0',v_0'} \sum_{y_1',z_{1}'}  pD(u_0',v_0')\E_0 \Big[ I \Bigl[ E^{'',0}(F;u,v,u_0')\Bigr]  I[ y_1' \in \tD_0 ]\tau(v_0',z_1')\sum_{N'\geq 1} A^{(N')}(y_1',z_1',x') \Big] \\\nonumber
  \leq & C\sum_{u_0',v_0'} \sum_{y_1',z_{1}'}  pD(u_0',v_0')\E_0  \Big[ I \Bigl[  E^{'',0}(F;u,v,u_0') \Bigr] I[ y_1' \in \tD_0 ]\Big]  \tau(v_0',z_1')\tau(z_1',y_1') \tau(y_1',x') \tau(z_1',x') \\\nonumber
  \leq & C \sum_{u'_0, y_1',z_{1}'} \E_0 \Big[ I \Bigl[  E^{'',0}(F;u,v,u_0') \Bigr] I[ y_1' \in \tD_0 ]\Big]  \tau(u_0',z_1')\tau(z_1',y_1') \tau(y_1',x') \tau(z_1',x')
 \end{align}

Notice that {\colAJ{$ E^{'',0}(F;u,v,u_0')\cap \{ y_1' \in \tD_0\} \subset E^{'',0}(F;u,v,u_0')\cap \{ y_1' \leftrightarrow W \}$, since $\{ y_1' \in \tD_0\} \subset \{y_1'\leftrightarrow W\}$.}}
Now, using  Lemma \ref{lem:E''-bnd} and then Lemma \ref{lem_extra_path4}, we see that
\begin{align*} 
& E^{'',0}(F;u,v,u_0')\cap \{ y_1' \in \tD_0\}  \\
\subset &
 \Bigl(\bigcup_{{\colAJ{y}}\in \mathbb{Z}^d} \bigcup_{\substack{{\bf G'}\in K_{0,{\colAJ{y}},u} \circledast u_0' \\ \text{ or }{\bf G'} =K_{0,{\colAJ{y}},u}\uplus \{W, x'\}}} (\mathcal{D}_{{\bf G'}} \circ \{{\colAJ{y}} \leftrightarrow x'\}) \cap \{y_1'\leftrightarrow W\}  \Bigr) \\
 & \qquad  \cup \Bigl( \bigcup_{\substack{{\bf G'}\in K_{0,u,W} \circledast x' \\ \text{ or }{\bf G'} =K_{0,u,W}\uplus \{W, x'\}}} (\mathcal{D}_{{\bf G'}} \circ \{0 \leftrightarrow x'\})\cap \{y_1'\leftrightarrow W\}\Bigr) \\
 \subset &
 \Bigl(\bigcup_{{\colAJ{y}}\in \mathbb{Z}^d} \bigcup_{\substack{{\bf G'}\in K_{0,{\colAJ{y}},u} \circledast u_0' \\ \text{ or }{\bf G'} =K_{0,{\colAJ{y}},u}\uplus \{W, u_0'\}}} \bigcup_{\substack{{\bf G''}\in ({\bf G'}\uplus \{ {\colAJ{y}}, u_0' \}) \circledast y_1 \\\text{ or }{\bf G''} =({\bf G'}\uplus \{ {\colAJ{y}}, u_0' \}) \uplus \{W, y_1\}}}\mathcal{D}_{{\bf G''}}\Bigr) \\
 & \qquad  \cup \Bigl( \bigcup_{\substack{{\bf G'}\in K_{0,u,W} \circledast u_0' \\ \text{ or }{\bf G'} =K_{0,u,W}\uplus \{W, u_0'\}}} \bigcup_{\substack{{\bf G''}\in ({\bf G'}\uplus \{ {\colAJ{y}}, x' \}) \circledast y_1 \\\text{ or }{\bf G''} =({\bf G'}\uplus \{ {\colAJ{y}}, u_0' \}) \uplus \{W, y_1\}}}\mathcal{D}_{{\bf G''}}  \Bigr).
 \end{align*}
 
 We take the probability of the previous events and using BK, this will give rise to diagrams.  Using the usual argument with Remark~\ref{neglect_W}, we get that
 \[
 {\bf P}[E^{'',0}(F;u,v,u_0')\cap \{ y_1' \in \tD_0\} ]\leq C_W              \begin{minipage}{0.4\textwidth}
      \includegraphics[width=0.8\linewidth]{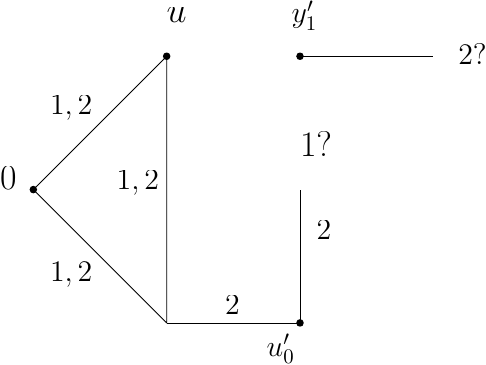}
    \end{minipage}\
    \]
and thus
\begin{align*}
& {\colAJ{\E_0  \Big[ I \Bigl[  E^{'',0}(F;u,v,u_0') \Bigr] I[ y_1' \in \tD_0 ]\Big]  \tau(u_0',z_1')\tau(z_1',y_1') \tau(y_1',x') \tau(z_1',x') }} \\
 \leq & C_W              \begin{minipage}{0.4\textwidth}
      \includegraphics[width=0.7\linewidth]{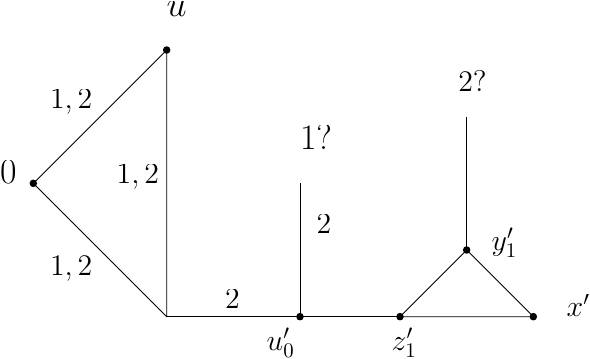}
    \end{minipage}\
    \end{align*}

Finally, by summing over the points {\colAJ{in the right hand side of~\eqref{bla1}}}, we see that
\[
 \sum_{N' \ge 1}\hat{\Psi}^{0,N'}(F,u,v,x')\leq C_W\;
    \begin{minipage}{0.4\textwidth}
      \includegraphics[width=0.7\linewidth]{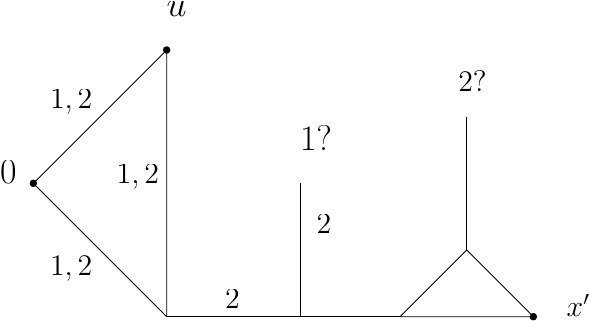}
    \end{minipage}\
 \]
and the result follows from Lemma~\ref{please_stop}.
\end{proof}

\subsubsection{The $N=1$ terms of the second expansion}

We next discuss the $N=1$ case. We start with a preparatory lemma.

\begin{lemma}\label{lem_dev_first}
For finite sets $A,B \subset \Zd$, we have
\[
  \widetilde{E}(F;u_0,v_0,x')\cap  \{y_1\leftrightarrow A\} \cap \{ B \Leftrightarrow x' \} \subset \bigcup_{{\bf G}\in {\colAJ{\mathcal{G}(u_0,x',W,A,B)}}} \mathcal{D}_{{\bf G}}.
\]
where {\colAJ{$\mathcal{G}(u_0,x',W,A,B)$}} is {\colAJ{the set of diagrammatic graphs shown below}}
\begin{align*}
 \begin{minipage}{0.4\textwidth}
      \includegraphics[width=0.5\linewidth]{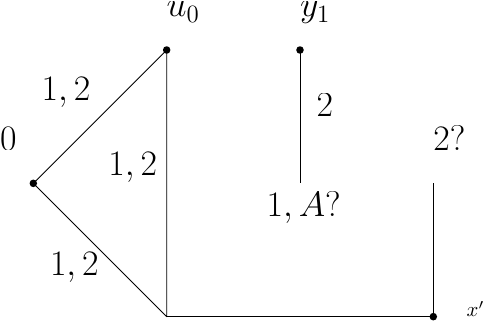}
    \end{minipage}\ \qquad  \begin{minipage}{0.4\textwidth}
      \includegraphics[width=0.5\linewidth]{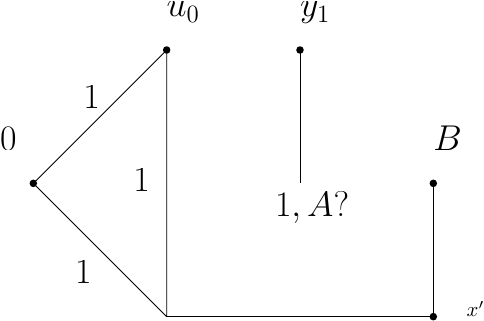}
    \end{minipage} \\[2ex]
      \begin{minipage}{0.4\textwidth}
      \includegraphics[width=0.5\linewidth]{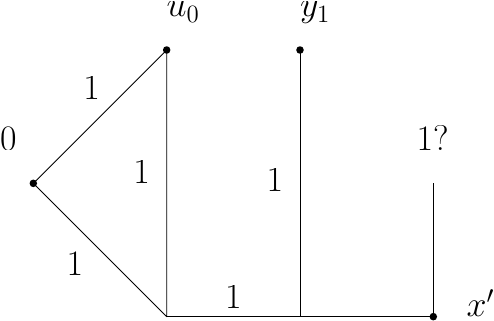}
    \end{minipage}\ \qquad  \begin{minipage}{0.4\textwidth}
      \includegraphics[width=0.5\linewidth]{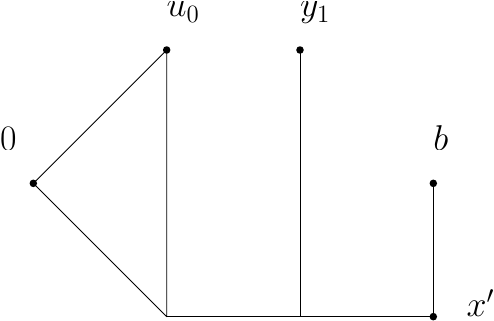}  
    \end{minipage} \end{align*}
{\colAJ{together with the corresponding diagrams with an edge $u_0 \cnctd W$ (not shown as 
due to Remark \ref{neglect_W} they will be bounded in terms of the above cases).}}
\end{lemma}

\begin{proof}
We use Lemma~\ref{lem:E-bnd}(ii),~\eqref{ship} and~\eqref{ship1} to see that 
{\colAJ{
\eqnsplst{
 \widetilde{E}(F;u_0,v_0,x')\cap   \{y_1\leftrightarrow A\}  
 &\subset {E}(F;u_0,x')\cap   \{y_1\leftrightarrow A\} \\
 &\subset  \bigcup_{\substack{{\bf G'}\in \{0,x'\} \circledast u_0 \\ \text{ or }{\bf G'} =\{0,x'\}\uplus{\colAJ{\{u_0, W\}}}}} \bigcup_{\substack{{\bf G''}\in ({\bf G'}\uplus \{ 0, u_0 \}) \circledast y_1 \\\text{ or }{\bf G''} =({\bf G'}\uplus \{ 0, u_0 \}) \uplus \{y_1, A\}}}\mathcal{D}_{{\bf G''}}. }
}}

We will upper-bound 
 \begin{align}\label{cargo}
& \widetilde{E}(F;u_0,v_0,x')\cap  \{y_1\leftrightarrow A\} \cap \{ B \Leftrightarrow x' \} \\\nonumber
 \subset &  \bigcup_{\substack{{\bf G'}\in \{0,x'\} \circledast u_0 \\ \text{ or }{\bf G'} =\{0,x'\}\uplus
 {\colAJ{\{u_0, W\}}}}} \bigcup_{\substack{{\bf G''}\in ({\bf G'}\uplus \{ 0, u_0 \}) \circledast y_1 \\\text{ or }{\bf G''} =({\bf G'}\uplus \{ 0, u_0 \}) \uplus \{y_1, A\}}}\mathcal{D}_{{\bf G''}}\cap \{ B \Leftrightarrow x' \},
 \end{align}
where the corresponding generalized diagram can be written as 
\[
\begin{minipage}{0.3\textwidth}
      \includegraphics[width=0.5\linewidth]{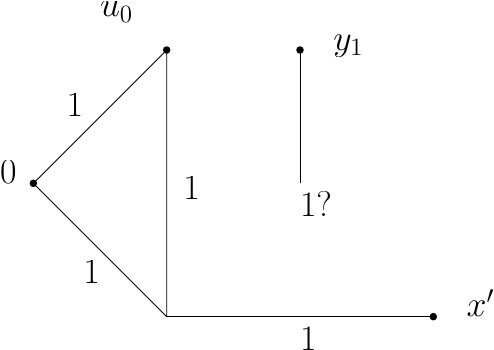}
    \end{minipage}\
    \]
{\colAJ{(We shall not spell out how to deal with the case of ${\bf G'} =\{0,x'\}\uplus {\colAJ{\{u_0, W\}}}$, as this is analogous to arguments in the proof of Proposition \ref{prop:n=0error}.)}}
    
By considering the cases where the edge $1?$ gets connected to the edge adjacent to $x'$ or to the triangle at $0$,
{\colAJ{respectively, we rewrite}} the diagrammatic graphs (just as in Example~\ref{ex_graph}) to see that
\begin{align*}
& \bigcup_{{\bf G'}\in \{0,x'\} \circledast u_0 }\bigcup_{\substack{{\bf G''}\in ({\bf G'}\uplus \{ 0, u_0 \}) \circledast y_1 \\\text{ or }{\bf G''} =({\bf G'}\uplus \{ 0, u_0 \}) \uplus \{y_1, A\}}}\mathcal{D}_{{\bf G''}}\cap \{ B \Leftrightarrow x' \}\\
=& \bigcup_{u_*,u_{**}\in \mathbb{Z}^d} (\mathcal{D}_{D_1(u_*,u_{**})}\circ \{u_{**} \leftrightarrow x'\})\cap \{  B \Leftrightarrow x' \} \\
& \qquad  \cup \bigcup_{u_*\in \mathbb{Z}^d} \bigcup_{{\bf G}\in D_2(u_*,A)} (\mathcal{D}_{\bf G} \circ \{u_{*} \leftrightarrow x'\})\cap \{  B \Leftrightarrow x' \}
\end{align*}
where 
\[
 D_1(u_*,u_{**}) =\begin{minipage}{0.3\textwidth}
      \includegraphics[width=0.5\linewidth]{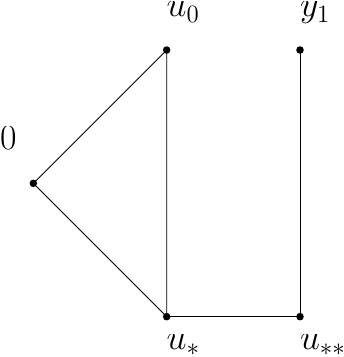}
    \end{minipage}\ \text{ and } \qquad D_2(u_*,A)   =\begin{minipage}{0.3\textwidth}
      \includegraphics[width=0.5\linewidth]{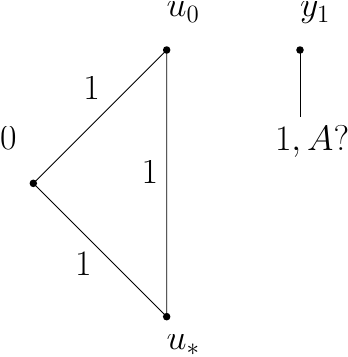}
    \end{minipage}
    \]

We use Lemma~\ref{lem_extra_path2} {\colAJ{to deal with the extra connections to $x'$}} and see that 
 \begin{align*}
& \bigcup_{{\bf G'}\in \{0,x'\} \circledast u_0 }\bigcup_{\substack{{\bf G''}\in ({\bf G'}\uplus \{ 0, u_0 \}) \circledast y_1 \\\text{ or }{\bf G''} =({\bf G'}\uplus \{ 0, u_0 \}) \uplus \{y_1, A\}}}\mathcal{D}_{{\bf G''}}\cap \{ B \Leftrightarrow x' \}\\ 
=& \bigcup_{u_*,u_{**}\in \mathbb{Z}^d} \bigcup_{\substack{{\bf G'}\in D_1(u_*,u_{**}) \circledast x' \\\text{ or }{\bf G'} =D_1(u_*,u_{**}) \uplus \{x',  B\}}} \mathcal{D}_{{\bf G'}}\circ \{u_{**} \leftrightarrow x'\} \\
& \qquad  \cup \bigcup_{u_*\in \mathbb{Z}^d} \bigcup_{{\bf G}\in D_2(u_*,A)} \bigcup_{\substack{{\bf G'}\in {\bf G} \circledast x' \\ \text{ or }{\bf G'}={\bf G} \uplus \{x',  B\}}}\mathcal{D}_{{\bf G'}} \circ \{u_{*} \leftrightarrow x'\}.
\end{align*}

The first union on the right-hand side in the previous equation is over the following graphs
\[
\begin{minipage}{0.4\textwidth}
      \includegraphics[width=0.5\linewidth]{Lem33-eps-converted-to.pdf}
    \end{minipage}\ \qquad  \begin{minipage}{0.4\textwidth}
      \includegraphics[width=0.5\linewidth]{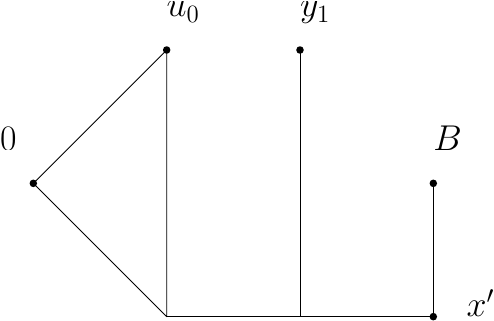}
    \end{minipage}
\]
and the second union is over the following graphs
\[
\begin{minipage}{0.4\textwidth}
      \includegraphics[width=0.5\linewidth]{Lem37-eps-converted-to.pdf}
    \end{minipage}\ \qquad  \begin{minipage}{0.4\textwidth}
      \includegraphics[width=0.5\linewidth]{Lem36-eps-converted-to.pdf}
    \end{minipage}
\]
and the result follows.
\end{proof}

We are ready to prove the bound for the $N=1$ terms.

{\colAJ{
\begin{proposition}
\label{prop:n=1error} For all $F\in\mathfrak F_{00}$ the following holds.  
\\
(i) There is a constant $C_W$ such that 
\eqn{eq:n=1error}{
  \sum_{N'=0}^\infty \Psi^{1,N'}(F;u,x')
  \le C_W T(o,u,x'). }
(ii) There is a constant $C_W$ such that uniform in $v'$,
\eqn{eq:n=1error2}{
  \hat{\Psi}^{1,0}(F,u,u',v') + \sum_{N'=1}^\infty \Psi^{1,N'}(F;u,u')
  \le C_W T(o,u,u'). }
\end{proposition}
}}

{\colAJ{
\begin{remark}
Observe that bounds on the $N' \ge 1$ terms in \eqref{eq:n=1error} and \eqref{eq:n=1error2} 
can be proved identically, by setting $x' = u'$.
\end{remark}
}}

\begin{proof}[Proof of Proposition \ref{prop:n=1error}.]
We start with the $N'=0$ term of (i), noting that the proof for (ii) will work
very similarly. We have that 
\eqnspl{bla1-old}{
 {\colAJ{\Psi^{1,0}(F,u,x')}} 
 &\leq C\sum_{u_0,v_0} pD(u_0,v_0)
 \E_0 \times \E_1\Big[I[\widetilde{E}(F;u_0,v_0,x')] I[ {\colAJ{E'(v_0,u,\tC_0)}} ] 
 I[ \{ P'_{x'} = \es \}]\Big] \\
 &\leq {\colAJ{C\sum_{u_0,v_0} pD(u_0,v_0)
 \E_0 \times \E_1\Big[I[{E}(F;u_0,x')] I[ E'(v_0,u,\tC_0) ] 
 I[ \{ P'_{x'} = \es \}]\Big]}}. }
Let us denote 
\[
T^+(v_0,z_1,y_1,u)= \{ v_0 \cnctd z_1 \} \circ \{ z_1 \cnctd y_1 \} \circ \{ y_1 \cnctd u \} \circ \{ z_1 \cnctd u \} .
\]
From Lemma \ref{lem2.5HS}, conditionally on $\omega_0$, we have
\[
 I[ E'(v_0,u;\tC_0) ] 
   \le \sum_{y_1, z_1} I[ y_1 \in \tC_0 ]  T^+(v_0,z_1,y_1,u)(\omega_1) .
\]
We use the previous equation, $ \{ P'_{x'} = \es \}\subset \{W \cup \tB_1\Leftrightarrow x'\}$ 
{\colAJ{(with $\tB_1 = \tB_1(v_0,u,\omega_1)$)}} and $\{y_1 \in \tC_0\}\subset \{y_1 \leftrightarrow_{\omega_0} W\}$ to see that 
 \begin{align}\label{bla1-new}
 & {\colAJ{\Psi^{1,0}(F,u,x')}} \\ \nonumber
  \leq &C \sum_{u_0,v_0} \sum_{y_1,z_{1}} p D(u_0,v_0) \sum_{B\subset \mathbb{Z}^d} \E_0\Big[I[{\colAJ{E(F;u_0,x')]}} I[ y_1 \leftrightarrow W ] I[ \{ W \cup B \Leftrightarrow x' \}] \Big] \\ \nonumber 
&   \qquad\qquad\qquad\qquad\qquad \times  \E_1 \Big[I[\tB_1=B] I[ T^+(v_0,z_1,y_1,u) ] \Big]  \end{align}
{\colAJ{
Analogously, the following upper bound can be derived for $\hat{\Psi}^{1,0}$ in statement (ii):
 \begin{align}\label{bla1-new2}
 & \hat{\Psi}^{1,0}(F,u,u',v') \\ \nonumber
  \leq &C \sum_{u_0,v_0} \sum_{y_1,z_{1}} p D(u_0,v_0) \sum_{B\subset \mathbb{Z}^d} \E_0\Big[I[ E(F;u_0,u')] I[ y_1 \leftrightarrow W ] I[ \{ W \cup B \Leftrightarrow u' \}] \Big] \\ \nonumber 
&   \qquad\qquad\qquad\qquad\qquad \times  \E_1 \Big[I[\tB_1=B] I[ T^+(v_0,z_1,y_1,u) ] \Big]  \end{align}
For this we merely have to replace $\tC_0$ by $\tD_0$, and $P'_{x'}$ by $P'_{u'}$, in the 
above arguments. From here onwards, the proof of (ii) will be identical to the proof of the $N'=0$ case of
(i) (with $u'$ replacing $x'$) and will not be spelled out.
}}

 The first expectation in \eqref{bla1-new} can be bounded using the BK inequality over all the generalized diagrams appearing in Lemma~\ref{lem_dev_first} with `$A$'$=W$ and `$B$'$=W \cup B$. Once more, using Remark~\ref{neglect_W}, we can neglect (when upper-bounding up to a multiplicative constant $C_W$) any connection to $W$ in both sets.  This allows us to write   
\begin{equation}\label{active}
\begin{minipage}{0.3\textwidth}
      \includegraphics[width=0.7\linewidth]{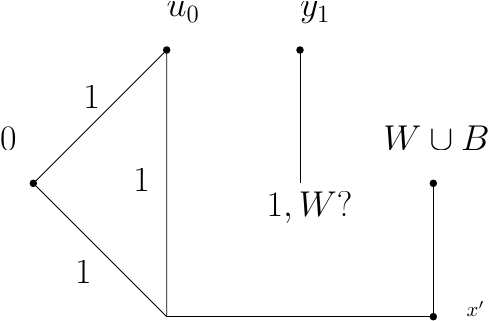} 
    \end{minipage} 
    \leq  C_W 
\begin{minipage}{0.3\textwidth}
      \includegraphics[width=0.7\linewidth]{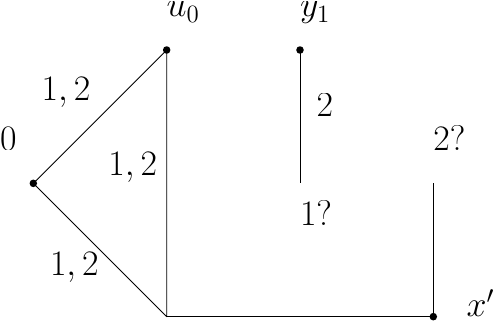}
    \end{minipage}
    + C_W  \sum_{b\in B}
\begin{minipage}{0.3\textwidth}
      \includegraphics[width=0.7\linewidth]{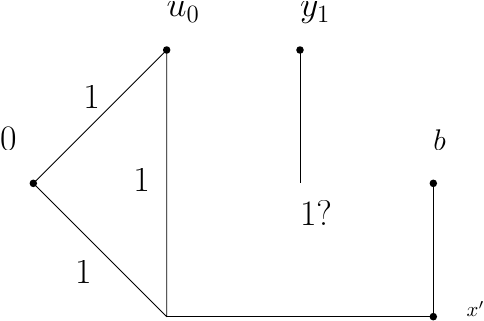}
    \end{minipage} 
\end{equation}
and
\[
    \begin{minipage}{0.3\textwidth}
      \includegraphics[width=0.7\linewidth]{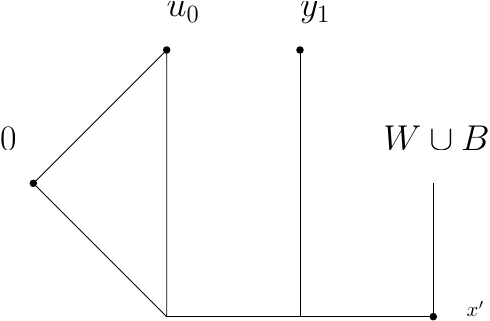} 
    \end{minipage}  \\ 
    \leq  C_W 
\begin{minipage}{0.3\textwidth}
      \includegraphics[width=0.7\linewidth]{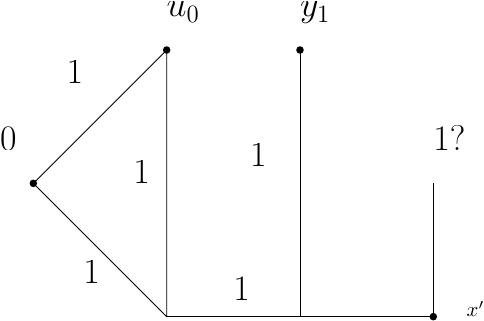}
    \end{minipage}
    + C_W  \sum_{b\in B}
\begin{minipage}{0.3\textwidth}
      \includegraphics[width=0.7\linewidth]{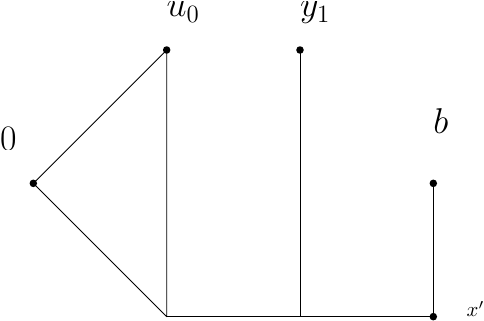}
    \end{minipage} 
\]
       
Taking this into account, we can introduce the pertinent set of diagrams. Let $\mathcal{G}_1$ be the set of diagrammatic graphs 
{\colAJ{constituting}} the two following generalized diagrams
 \begin{equation}\label{eldo1}
\begin{minipage}{0.4\textwidth}
      \includegraphics[width=0.5\linewidth]{Lem33-eps-converted-to.pdf} 
    \end{minipage} 
\qquad \text{and} \qquad 
\begin{minipage}{0.4\textwidth}
      \includegraphics[width=0.5\linewidth]{Lem36b-eps-converted-to.pdf}
    \end{minipage}
    \end{equation}
and $\mathcal{G}_2(b)$ be the set of diagrams constituting the following two generalized diagrams 
 \begin{equation}\label{eldo2}
\begin{minipage}{0.4\textwidth}
      \includegraphics[width=0.5\linewidth]{Lem35-eps-converted-to.pdf} 
    \end{minipage} 
\qquad \text{and} \qquad 
\begin{minipage}{0.4\textwidth}
      \includegraphics[width=0.5\linewidth]{Lem37d-eps-converted-to.pdf}
    \end{minipage}
    \end{equation}

With this notation (and the simplifications of $W$), Lemma~\ref{lem_dev_first}  can be written
\begin{align*}
& {\bf P}_0[{\colAJ{E(F;u_0,x')}} \cap  \{y_1 \leftrightarrow W\} \cap \{ W \cup B \Leftrightarrow x' \}] \\
\leq & C_W\sum_{{\bf G}\in \mathcal{G}_1} {\bf P}_0[\mathcal{D}_{\bf G}] + C_W  \sum_{b\in \mathbb{Z}^d} I[b\in B] \sum_{{\bf G}\in \mathcal{G}_2(b)} {\bf P}_0[\mathcal{D}_{\bf G}]
\end{align*}

 Using this estimate yields
  \begin{align}\label{crack}
 & \sum_{u_0,v_0} \sum_{y_1,z_{1}} \E_0 \Big[I[{\colAJ{E(F;u_0,x')}}] I[ y_1 \leftrightarrow W ] I[ \{ W \cup \tB_1 \Leftrightarrow x' \}] \Big]
\\ \nonumber
        \leq &C_W \sum_{u_0,v_0} \sum_{y_1, z_1} p D(u_0,v_0) \sum_{B \subset \Zd} \sum_{{\bf G}\in \mathcal{G}_1} {\bf P}_0[\mathcal{D}_{\bf G}]  \E_1\Big[I[\tB_1=B] I[T^+(v_0,z_1,y_1,u)] \Big] \\\nonumber
        & +C_W \sum_{u_0,v_0} \sum_{y_1, z_1} p D(u_0,v_0) \sum_{B\subset \mathbb{Z}^d} \sum_{b\in \mathbb{Z}^d} I[b\in B] \sum_{{\bf G}_b \in \mathcal{G}_2(b)} {\bf P}_0[\mathcal{D}_{{\bf G}_b}] \E_1\Big[I[\tB_1=B] I[T^+(v_0,z_1,y_1,u)] \Big] \\\nonumber
          \leq &C_W \sum_{u_0,v_0} \sum_{y_1, z_1} p D(u_0,v_0) \sum_{{\bf G}\in \mathcal{G}_1} {\bf P}_0[\mathcal{D}_{\bf G}] {\bf P}_1 [T^+(v_0,z_1,y_1,u)] \\\nonumber
        & +C_W \sum_{u_0,v_0} \sum_{y_1, z_1} p D(u_0,v_0) \sum_{{\bf G}_b\in \mathcal{G}_2(b)} {\bf P}_0[\mathcal{D}_{{\bf G}_b}]  {\bf P}_1\Big[\{b\in\tB_1\} \cap T^+(v_0,z_1,y_1,u) \Big] ,      
 \end{align}
where we recall that all the graphs in $\mathcal{G}_1$ and $\mathcal{G}_2(b)$ implicitly depend on $u_0$, $y_1$ and $x'$.

{\colAJ{For the first term in the right hand side of \eqref{crack}, we use the BK inequality on the term ${\bf P}_0 [\mathcal{D}_{\bf G}] \,{\bf P}_1 [T^+(v_0,z_1,y_1,u)]$, and use \eqref{e:D*tau} to see that}}
\begin{align*}
 & \sum_{u_0,v_0} \sum_{y_1, z_1} p D(u_0,v_0) \sum_{{\bf G}\in \mathcal{G}_1} {\bf P}_0[\mathcal{D}_{\bf G}] {\bf P}_1 [T^+(v_0,z_1,y_1,u)] \\
\leq & \begin{minipage}{0.4\textwidth}
      \includegraphics[width=0.5\linewidth]{Lem38-eps-converted-to.pdf} 
    \end{minipage} 
+
\begin{minipage}{0.4\textwidth}
      \includegraphics[width=0.5\linewidth]{Lem39-eps-converted-to.pdf}
    \end{minipage}
\end{align*}

{\colAJ{For the second term in the right hand side of \eqref{crack}, we apply Lemma~\ref{sam} to the diagrammatic event $T^+(v_0,z_1,y_1,u)$ to bound ${\bf P}_1\Big[\{b\in\tB_1\}  \cap T^+(v_0,z_1,y_1,u)  \Big] $ by a generalized diagram.
This, and \eqref{e:D*tau}, show that
\begin{align*}
 & \sum_{u_0,v_0}\sum_{y_1, z_1}  \sum_{{\bf G}_b \in \mathcal{G}_2(b)}pD(u_0,v_0) {\bf P}_0[\mathcal{D}_{{\bf G}_b}]  {\bf P}_1\Big[\{b\in\tB_1\}  \cap T^+(v_0,z_1,y_1,u)  \Big]  \\
\leq & \begin{minipage}{0.4\textwidth}
      \includegraphics[width=0.5\linewidth]{Lem51-eps-converted-to.pdf} 
    \end{minipage} 
+
\begin{minipage}{0.4\textwidth}
      \includegraphics[width=0.5\linewidth]{Lem52-eps-converted-to.pdf}
    \end{minipage}
\end{align*}
}}
The previous three equations imply 
  \begin{align}\label{bottom}
&\sum_{u_0,v_0} \sum_{y_1,z_{1}} \sum_{B\subset \mathbb{Z}^d} pD(u_0,v_0)
 \E_0\Big[I[{\colAJ{E(F;u_0,x')}}]  I[ y_1 \leftrightarrow W ] I[ \{ W \cup B \Leftrightarrow x' \}] \Big] \\ \nonumber 
&   \qquad\qquad\qquad\qquad\qquad \times  \E_1 \Big[I[\tB_1=B]I[T^+(v_0,z_1,y_1,u)]\Big] \\ \nonumber
\leq& C_W  \begin{minipage}{0.4\textwidth}
      \includegraphics[width=0.5\linewidth]{Lem38-eps-converted-to.pdf} 
    \end{minipage} 
+ C_W
\begin{minipage}{0.4\textwidth}
      \includegraphics[width=0.5\linewidth]{Lem39-eps-converted-to.pdf}
    \end{minipage} \\ \nonumber 
   & +C_W\begin{minipage}{0.4\textwidth}
      \includegraphics[width=0.5\linewidth]{Lem51-eps-converted-to.pdf} 
    \end{minipage} 
+C_W
\begin{minipage}{0.4\textwidth}
      \includegraphics[width=0.5\linewidth]{Lem52-eps-converted-to.pdf}
    \end{minipage}  
    \end{align}
    
This and~\eqref{bla1-new} yields
{\colAJ{
\[
\Psi^{1,0}(F;u,x')\leq C_W T(o,u,x'),
\]
by Lemma~\ref{diag10}. Likewise, $\hat{\Psi}^{1,0}(F;u,u',v') \le C_W T(o,u,u')$.}}

Consider now the $N'\geq 1$ terms. Starting with $N'=1$, consider the expression $I[E'(v'_0,u';\tD_0 \cup \tC(\tB_1))]]$ inside the definition of
{\colAJ{$\Psi^{1,1}(F;u,u')$. Using Lemma \ref{lem2.5HS},}} this is bounded, conditionally on $\omega_0$ and $\omega_1$, by
\eqnsplst{
   \E'_1 \left[ I[E'(v'_0,u';\tD_0 \cup \tC(\tB_1))] \right] 
   &\le \sum_{y'_1, z'_1} I[ y'_1 \in \tD_0  \cup \tC(\tB_1)] \tau(v'_0,z'_1) \tau(z'_1,y'_1) \tau(y'_1,u') \tau(z'_1,u') \\
   &= \sum_{y'_1, z'_1} I[ y'_1 \in \tD_0  \cup \tC(\tB_1)] \tau(v'_0,z'_1) A^{(1)}(y'_1,z'_1,u'). }
In the case $N' \ge 2$, we similarly get the bound 
\[ \E_1' \left[ E'(v'_0,u'_1;\tD_0 \cup \tC(\tB_1)) \Xi^{(N')}(v'_1,u';\tC'_1) \right] 
   \le \sum_{y'_1, z'_1} I[ y'_1 \in \tD_0 \cup \tC(\tB_1) ] \tau(v'_0,z'_1) A^{(N')}(y'_1,z'_1,u'). \]
{\colAJ{
Hence, using \eqref{eq:Aass} we have
 \begin{align}\label{bla2}
 \sum_{N' \ge 1} & \Psi^{1,N'}(F;u,u') \\ \nonumber
  \leq &C\sum_{u_0,v_0} \sum_{u_0',v_0'} \sum_{y_1',z_{1}'} pD(u_0,v_0) pD(u_0',v_0')\E_0\times \E_1 \Big[ I \Bigl[ \hat{E}''(F;u_0,v_0,u,u'_0,v'_0)\Bigr] \\ \nonumber
  &\qquad \times I[ y_1' \in \tD_0 \cup \tC(\tB_1)]\tau(v_0',z_1')\sum_{N'\geq 1} A^{(N')}(y_1',z_1',u') \Big] \\\nonumber
  \leq & C\sum_{u_0,v_0} \sum_{u_0',v'_0} \sum_{y_1',z_{1}'} pD(u_0,v_0) pD(u'_0,v'_0) \E_0\times \E_1   \Big[ I \Bigl[  \hat{E}''(F;u_0,v_0,u,u'_0,v'_0) \Bigr] \\ \nonumber
  &\qquad \times I[ y_1' \in \tD_0 \cup \tC(\tB_1)]\Big]  \tau(v_0',z_1') T(z'_1,y'_1,u')
 \end{align}

Recall that 
\eqnsplst{
 \hat{E}''(F;u_0,v_0,u,u'_0,v'_0) 
 &\subset \widetilde{E}(F;u_0,v_0,u'_0) \cap E'(v_0,u,\tD_0(F;u_0,v_0,u'_0,v'_0)) \\
 &\subset E(F;u_0,u'_0) \cap E'(v_0,u,\tD_0(F;u_0,v_0,u'_0,v'_0)). } 
This gives 
\begin{align}\label{crazy}
  & \sum_{N' \ge 1} \Psi^{1,N'}(F;u,u') 
  \leq \sum_{\substack{u_0,v_0 \\ u'_0,v'_0}} \sum_{y_1,z_1} \sum_{y_1', z_1'} p D(u_0,v_0) p D(u'_0,v'_0) \tau(v_0',z_1') T(z'_1,y_1',u') \\ \nonumber
  &\quad \E_0 \times \E_1\Big[ I[ y_1' \leftrightarrow_{\omega_0} W \cup  \tB_1 ] I[E(F;u_0,u'_0)]
     I[ E'(v_0,u,\tD_0) ] I[ \{ P'_{u'_0} = \es \}] \Big]. 
     \end{align}
}} 
The $\E_0 \times \E_1$ expectation is {\colAJ{essentially}} the same as the one appearing in~\eqref{bla1-old}, apart from the extra term $I[ y_1' \leftrightarrow_{\omega_0} W \cup  \tB_1 ]$. Following the same steps as for that equation, we obtain
\begin{align}\label{crazy1}
  & \E_0 \times \E_1[I[ y_1' \leftrightarrow_{\omega_0} W \cup  \tB_1 ] 
  I[E(F;u_0,u'_0) \cap E'(v_0,u,{\colAJ{\tD_0}}) 
     \cap \{ P'_{u'_0} = \es \}]] \\ \nonumber
  \leq &   C\sum_{y_1, z_1} \sum_{B\subset \mathbb{Z}^d} \E_0\Big[I[ y_1' \leftrightarrow W \cup  B ] 
  I [ {\colAJ{E(F;u_0,u_0')}}] I[ y_1 \leftrightarrow W ] I[ W \cup B \Leftrightarrow u_0' ] \Big] \\ \nonumber
        & \qquad \qquad \qquad \qquad \qquad \qquad  \times \E_1\Big[I[\tB_1=B]   I [ T^+(v_0,z_1,y_1,u)]  \Big].
        \end{align}

We know that ${\colAJ{E(F;u_0,u_0')}} \cap \{ y_1 \leftrightarrow W\} \cap \{ W \cup B \Leftrightarrow u_0' \}$ is included in a sum of  diagrammatic events which are described by the generalized diagrams appearing in Lemma~\ref{lem_dev_first} with  '$A$'$=W$ and '$B$'$=W \cup B$ (and with $x'$ replaced by $u_0'$). Note that to be perfectly rigorous, one should include some further generalized diagrams appearing in~\ref{cargo} with an edge leading to $W$ but, as we know from Remark~\ref{neglect_W}, they will not contribute to the $\E_0 \times \E_1$ expectation more than the main term (up to a multiplicative constant $C_W$), so we choose to omit them. 

Furthermore, with an argument similar to the one used starting at~\eqref{active}, the pertinent diagrams for the event ${\colAJ{E(F;u_0,u_0')}} \cap \{ y_1 \leftrightarrow W\} \cap \{ W \cup B \Leftrightarrow u_0' \}$ are the ones appearing in~\eqref{eldo1} and~\eqref{eldo2}. Hence, writing $ \mathcal{G}_1^{x'\to u_0'}$ (resp.~{\colAJ{$ \mathcal{G}_2(b)^{x'\to u_0'}$}}) for the diagrams in~\eqref{eldo1} (resp.~\eqref{eldo2}) with $x'$ replaced by $u_0'$, we have 
\begin{align*}
& \E_0\Big[I[ y_1' \leftrightarrow W \cup  B ] I [ \widetilde{E}(F;u_0,v_0,u_0')]I[ y_1 \leftrightarrow W ] I[ W \cup B \Leftrightarrow u_0' ] \Big]\\
\leq &C_W\sum_{{\bf G}\in \mathcal{G}_1^{x\to u_0'}} {\bf P}_0[\mathcal{D}_{\bf G} \cap \{y_1' \leftrightarrow W \cup  B\}] + C_W  \sum_{b\in \mathbb{Z}^d} I[b\in B] \sum_{{\bf G}\in \mathcal{G}_2(b)^{x\to u_0'}} {\bf P}_0[\mathcal{D}_{\bf G}\cap \{y_1' \leftrightarrow W \cup  B\}] \\
\leq &C_W\sum_{{\bf G}\in \mathcal{G}_1^{x\to u_0'}} {\bf P}_0[\mathcal{D}_{\bf G} \cap \{y_1' \leftrightarrow  B\}] + C_W  \sum_{b\in \mathbb{Z}^d} I[b\in B] \sum_{{\bf G}\in \mathcal{G}_2(b)^{x\to u_0'}} {\bf P}_0[\mathcal{D}_{\bf G} \cap \{y_1' \leftrightarrow   B\}],
\end{align*}
where in the last line we used the fact that the connections to $W$ contribute at most, up to a multiplicative $C_W$ constant, as much as the main terms. We now use Lemma~\ref{lem_extra_path4} for estimating the probabilities in the last sums, and this allows us to show that
\begin{align*}
& \E_0\Big[I[ y_1' \leftrightarrow W \cup  B ] I [ {\colAJ{E(F;u_0,u_0')}} ] 
  I[ y_1 \leftrightarrow W ] I[ W \cup B \Leftrightarrow u_0' ] \Big]\\
\leq &  C_W\sum_{{\bf G}\in \tilde{\mathcal{G}}_{1,1}}{\bf P}_0[\mathcal{D}_{\bf G}] 
  +C_W \sum_{b'\in \mathbb{Z}^d} \sum_{{\bf G}\in \tilde{\mathcal{G}}_{1,2}(b')}
   I[b'\in B]{\bf P}_0[\mathcal{D}_{\bf G}]  \\
& \qquad +C_W \sum_{b\in \mathbb{Z}^d} \sum_{{\bf G}\in \tilde{\mathcal{G}}_{2,1}(b)} 
   I[b \in B]{\bf P}_0[\mathcal{D}_{\bf G}]  
   + C_W \sum_{b\in \mathbb{Z}^d} \sum_{b'\in \mathbb{Z}^d} \sum_{{\bf G}\in \tilde{\mathcal{G}}_2(b,b')} 
   I[b\in B]  I[b'\in B]  {\bf P}_0[\mathcal{D}_{\bf G}],
\end{align*}
using the notation $\tilde{\mathcal{G}}_{1,1}=\mathcal{G}_{1,1,1}\cup \mathcal{G}_{2,1,1}$, $\tilde{\mathcal{G}}_{2,1}(b)=\mathcal{G}_{1,2,1}(b)\cup \mathcal{G}_{2,2,1}(b)$, $\tilde{\mathcal{G}}_{1,2}(b')=\mathcal{G}_{1,1,2}(b')\cup \mathcal{G}_{2,1,2}(b')$ and $\tilde{\mathcal{G}}_{2,2}(b,b')=\mathcal{G}_{1,2,2}(b,b')\cup \mathcal{G}_{2,2,2}(b,b')$ where 
  $\mathcal{G}_{1,1,1}$, $\mathcal{G}_{1,2,1}(b)$, $\mathcal{G}_{1,1,2}(b')$ and $\mathcal{G}_{1,2,2}(b,b')$ are the sets of diagrams constituting the following generalized diagrams
 \begin{align}\label{eldo11}
&	\mathcal{G}_{1,1,1}=
	\begin{minipage}{0.3\textwidth}
      \includegraphics[width=\linewidth]{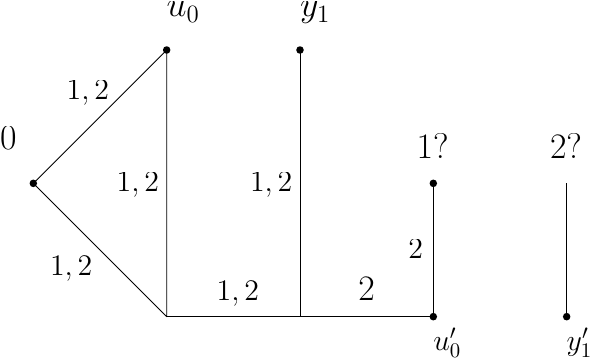} 
    \end{minipage} 
\qquad   
\mathcal{G}_{1,2,1}(b)=\begin{minipage}{0.3\textwidth}
      \includegraphics[width=\linewidth]{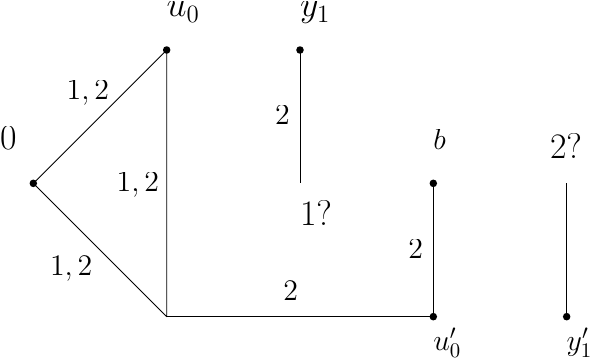}
    \end{minipage} \\
  &  \mathcal{G}_{1,1,2}(b')=
    \begin{minipage}{0.3\textwidth}
      \includegraphics[width=\linewidth]{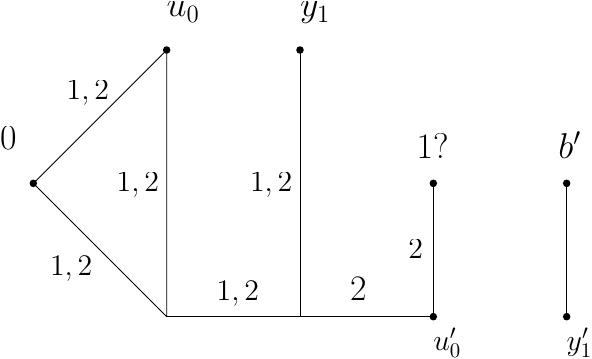} 
    \end{minipage} 
    \qquad
\mathcal{G}_{1,2,2}(b,b')=
\begin{minipage}{0.3\textwidth}
      \includegraphics[width=\linewidth]{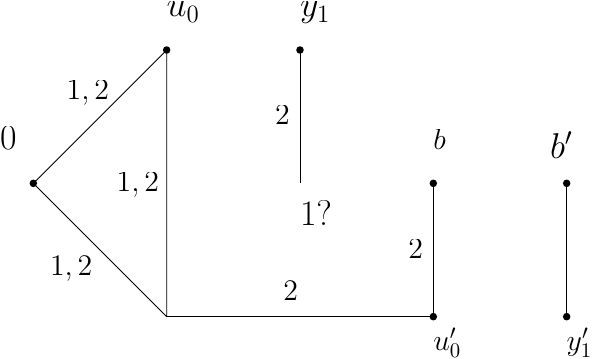}
    \end{minipage}
    \end{align}
    and $\mathcal{G}_{2,1,1}$, $\mathcal{G}_{2,2,1}(b)$, $\mathcal{G}_{2,1,2}(b')$ and $\mathcal{G}_{2,2,2}(b,b')$  are the set of diagrams constituting the two following generalized diagrams
 \begin{align}\label{eldo21}
\hskip-5em\mathcal{G}_{2,2,1}(b)=
\begin{minipage}{0.3\textwidth}
      \includegraphics[width=\linewidth]{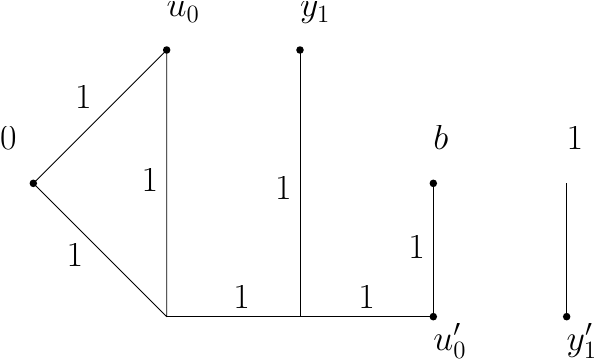} 
    \end{minipage} 
    \qquad\qquad
\mathcal{G}_{2,1,1}=
\begin{minipage}{0.3\textwidth}
      \includegraphics[width=\linewidth]{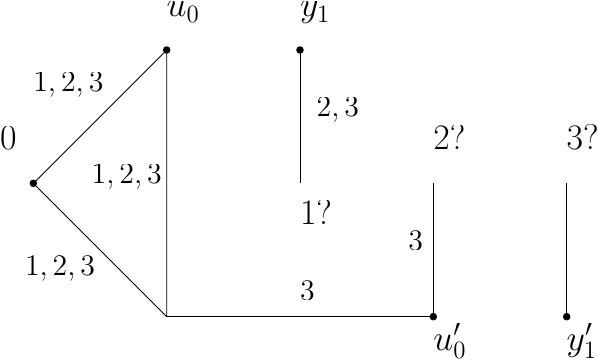}
    \end{minipage} \\
    \mathcal{G}_{2,2,2}(b,b')=
    \begin{minipage}{0.3\textwidth}
      \includegraphics[width=\linewidth]{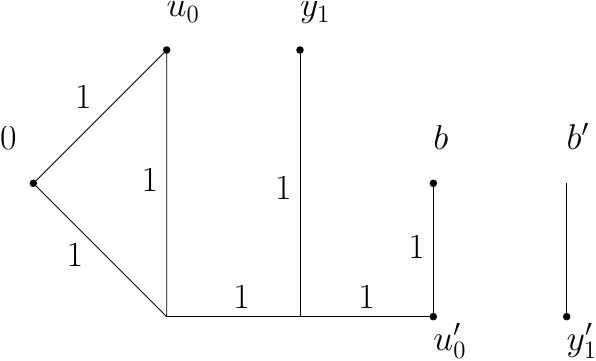} 
    \end{minipage} 
    \qquad
\mathcal{G}_{2,1,2}(b')=
\begin{minipage}{0.3\textwidth}
      \includegraphics[width=\linewidth]{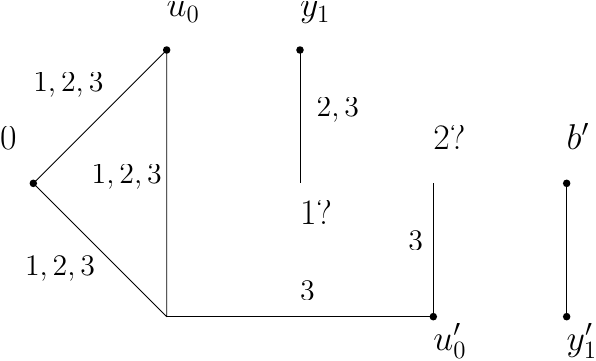}
    \end{minipage}
    \end{align}
Going back to~\eqref{crazy} and the equation that follows, the previous estimate allows us to write
\begin{align*}
  & \sum_{u_0,v_0} pD(u_0,v_0) \E_0 \times \E_1[I[ y_1' \leftrightarrow_{\omega_0} W \cup  \tB_1 ] I[ {\colAJ{E(F;u_0,u'_0)}} \cap E'(v_0,u_1,{\colAJ{\tD_0}}) 
     \cap \{ P'_{u'_0} = \es \}]] \\ \nonumber
  \leq &  C\sum_{u_0,v_0}\sum_{y_1, z_1}  \sum_{B\subset \mathbb{Z}^d} pD(u_0,v_0)
  \Bigl(\sum_{{\bf G}\in \tilde{\mathcal{G}}_{1,1}}{\bf P}_0[\mathcal{D}_{\bf G}]  
  + \sum_{b'\in \mathbb{Z}^d} \sum_{{\bf G}\in \tilde{\mathcal{G}}_{1,2}(b')} 
    I[b'\in B]{\bf P}_0[\mathcal{D}_{\bf G}]  \\ \nonumber
& \qquad + \sum_{b\in \mathbb{Z}^d} \sum_{{\bf G}\in \tilde{\mathcal{G}}_{2,1}(b)} 
    I[b \in B]{\bf P}_0[\mathcal{D}_{\bf G}] 
  + \sum_{b\in \mathbb{Z}^d} \sum_{b'\in \mathbb{Z}^d} \sum_{{\bf G}\in \tilde{\mathcal{G}}_2(b,b')} 
    I[b\in B]  I[b'\in B] {\bf P}_0[\mathcal{D}_{\bf G}]\Bigr) \\
        & \qquad \qquad \qquad \qquad \qquad \qquad  \times \E_1\Big[I[\tB_1=B] I [ T^+(v_0,z_1,y_1,u)] \Big] \\ \nonumber
  \leq &  C\sum_{u_0,v_0}\sum_{y_1, z_1}  \sum_{B\subset \mathbb{Z}^d} pD(u_0,v_0)
  \Bigl(\sum_{{\bf G}\in \tilde{\mathcal{G}}_{1,1}}{\bf P}_0[\mathcal{D}_{\bf G}]
    \E_1\Big[ I [ T^+(v_0,z_1,y_1,u)] \Big]\\
 &  \qquad\qquad \qquad\qquad 
   + \sum_{b'\in \mathbb{Z}^d} \sum_{{\bf G}\in \tilde{\mathcal{G}}_{1,2}(b')}
     {\bf P}_0[\mathcal{D}_{\bf G}]  \E_1\Big[I[b'\in  \tB_1] I [ T^+(v_0,z_1,y_1,u)] \Big] \\ \nonumber
 & \qquad\qquad\qquad\qquad 
   + \sum_{b\in \mathbb{Z}^d} \sum_{{\bf G}\in \tilde{\mathcal{G}}_{2,1}(b)}
     {\bf P}_0[\mathcal{D}_{\bf G}] \E_1\Big[I[b\in \tB_1]  I [ T^+(v_0,z_1,y_1,u)]\Big]  \\ 
 & \qquad\qquad\qquad\qquad 
   +  \sum_{b\in \mathbb{Z}^d} \sum_{b'\in \mathbb{Z}^d} \sum_{{\bf G}\in \tilde{\mathcal{G}}_2(b,b')}  
     {\bf P}_0[\mathcal{D}_{\bf G}] \E_1\Big[I[b\in\tB_1] I[b'\in\tB_1] I [ T^+(v_0,z_1,y_1,u)]\Big]\Bigr).
        \end{align*}
The ${\bf E}_1$-expectations of the second and third term can bounded using Lemma~\ref{sam} and the fourth term can be bounded using the same lemma twice. This, together with \eqref{e:D*tau}, yields
        \[
        \sum_{v_0}  pD(u_0,v_0)  \E_1\Big[I[b'\in \tB_1]   I [ T^+(v_0,z_1,y_1,u)]  \Big] \lesssim \begin{minipage}{0.4\textwidth}
      \includegraphics[width=0.5\linewidth]{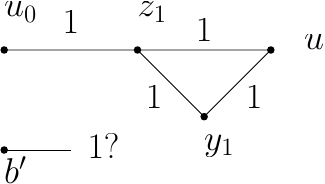}
    \end{minipage}
        \]
        and 
        \[
         \sum_{v_0}  pD(u_0,v_0)  \E_1\Big[I[b\in\tB_1] I[b'\in\tB_1]    I [ T^+(v_0,z_1,y_1,u)]\Big] \lesssim \begin{minipage}{0.4\textwidth}
      \includegraphics[width=0.5\linewidth]{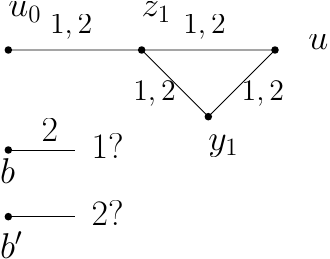}
    \end{minipage}
         \]
         
Using the last three equations (where the diagrammatic notation were introduced in~\eqref{eldo11} and~\eqref{eldo21}) in~\eqref{crazy}  we see that $\sum_{N' \ge 1} \Psi^{1,N'}(F;u,u')$ can be bounded up to multiplicative constants by the diagrams in Lemma~\ref{diag11} and the result follows.
\end{proof}

\subsubsection{The $N \ge 2$ terms of the second expansion}

Finally, we discuss the $N \ge 2$ case. 

\begin{proposition}
\label{prop:n>=2error} For all $F\in\mathfrak F_{00}$ the following holds.   \\
{\colAJ{(i) There is a constant $C_W$ such that 
\eqn{eq:n>=2error}{
  \sum_{N=2}^{\infty} \sum_{N'=0}^\infty \Psi^{N,N'}(F;u,x')
  \le C_W T(o,u,x'). }
(ii) There is a constant $C_W$ such that for all $v'$, 
\eqn{eq:n>=2error2}{
  \sum_{N=2}^{\infty} \left[ \hat{\Psi}^{N,0}(F;u,u',v') + \sum_{N'=1}^\infty \Psi^{N,N'}(F;u,u') \right]
  \le C_W T(o,u,u'). }
}}
\end{proposition}

{\colAJ{
\begin{remark}
Similarly to Proposition \ref{prop:n=1error}, the only difference between the proofs of (i) and (ii)
lies in a slight difference in handling the $N'=0$ terms.
\end{remark}
}}

\begin{proof}
First we bound the expression $\Xi^{(N)}$, conditional on $\omega_0, \omega_1$ 
{\colAJ{using Lemma \ref{lem2.5HS}}}, by the expression
\[ \Xi^{(N)}(v_1,u;\tC_1)
   \le \sum_{y_2,z_2} I[ y_2 \in \tC_1 ] \tau(v_1,z_2) A^{(N-1)} (y_2,z_2,u). \]
Hence, using~\eqref{eq:Aass}, we have
\[ 
\sum_{N=2}^\infty \Xi^{(N)}(v_1,u;\tC_1) 
\le C \sum_{y_2,z_2} I[ y_2 \in \tC_1 ] \tau(v_1,z_2) T(y_2,z_2,u).
 \]

We start with the case $N'=0$. Using the previous equation, and {\colAJ{then \eqref{e:D*tau} to 
perform the summation over $v_1$, we have}}
 \begin{align*}
\sum_{N=2}^{\infty} \hat{\Psi}^{N,0}(F;u,u',v') 
 & \leq \sum_{\substack{u_0,v_0 \\ u'_0,v'_0}}
  \sum_{u_1,y_2,z_2}
    p D(u_0,v_0) p D(u'_0,v'_0) \\
  &\qquad  \qquad \E_0 \times \E_1 \Big[I[ y_2 \in \tC_1 ] I \left[ \hat{E}''(F;u_0,v_0,u_1,u',v') \right] \Big] \tau(u_1,z_2) T(y_2,z_2,u).
     \end{align*}
This expression is similar to {\colAJ{equation \eqref{bla1-old}}} of Proposition~\ref{prop:n=1error}. 
By following the same reasoning up to~\eqref{crack} we obtain
   \begin{align}\label{yay}
 & \sum_{N=2}^{\infty} \hat{\Psi}^{N,0}(F;u,u',v') \\ \nonumber
 & \quad \leq C_W \sum_{\substack{u_0,v_0 \\ u'_0,v'_0}} \sum_{u_1,y_2,z_2} 
  p D(u_0,v_0) p D(u'_0,v'_0) \tau(u_1,z_2) T(y_2,z_2,u) \\ \nonumber
 &\qquad \Bigl( \sum_{{\bf G}\in \mathcal{G}_1} {\bf P}_0[\mathcal{D}_{\bf G}] 
   {\bf P}_1 [\{ y_2 \in \tC_1 \}\cap T^+(v_0,z_1,y_1,u_1)]  \\
 &\qquad\quad +\sum_{b\in \mathbb{Z}^d} \sum_{{\bf G}_b\in \mathcal{G}_2(b)} {\bf P}_0[\mathcal{D}_{{\bf G}_b}] 
   {\bf P}_1\Big[\{ y_2 \in \tC_1 \}\cap\{b\in\tB_1\}  \cap T^+(v_0,z_1,y_1,u_1)  \Big]\Bigr).      
 \end{align}
      
{\colAJ{Consider first the summation over $\mathcal{G}_1$ in \eqref{yay}. }}
Applying Lemma~\ref{lem_extra_path4} shows that 
\[
\{ y_2 \in \tC_1 \}\cap T^+(v_0,z_1,y_1,{\colAJ{u_1}}) \subset \bigcup_{{\bf G}\in \mathcal{G}_{T^+}} \mathcal{D}_{\bf G},
\]
 where $\mathcal{G}_{T^+}$ is the set of diagrams given by the following generalized diagram
  \[
          \begin{minipage}{0.4\textwidth}
      \includegraphics[width=0.7\linewidth]{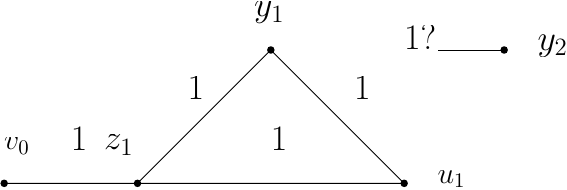}
    \end{minipage}\
  \]
and thus by BK
 \begin{align*}
 & {\colAJ{\sum_{u_1,y_2,z_2}}} \tau(u_1,z_2) T(y_2,z_2,u) 
   {\bf P}_1 [\{ y_2 \in \tC_1 \} \cap T^+(v_0,z_1,y_1,{\colAJ{u_1}})] \\
 \leq &       \qquad     \begin{minipage}{0.6\textwidth}
      \includegraphics[width=0.7\linewidth]{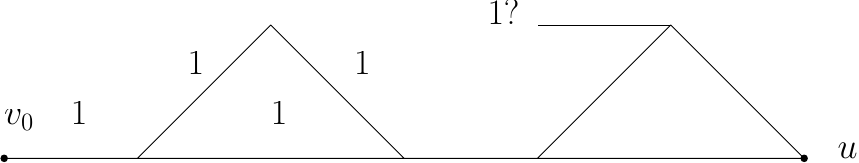}
    \end{minipage}\
   \end{align*}     

{\colAJ{Consider now the summation over $\mathcal{G}_2(b)$ in \eqref{yay}.}}
Applying Lemma~\ref{lem_extra_path4} and after that Lemma~\ref{sam} shows that 
\[
 \{ y_2 \in \tC_1 \}\cap\{b\in\tB_1\}  \cap T^+(v_0,z_1,y_1,{\colAJ{u_1}}) 
 \subset \bigcup_{{\bf G}_b \in \mathcal{G}_{T^+}(b)} \mathcal{D}_{\bf G},
\]
where $\mathcal{G}_{T^+}(b)$ is the set of diagrams given by the following generalized diagram
  \[
          \begin{minipage}{0.4\textwidth}
      \includegraphics[width=0.8\linewidth]{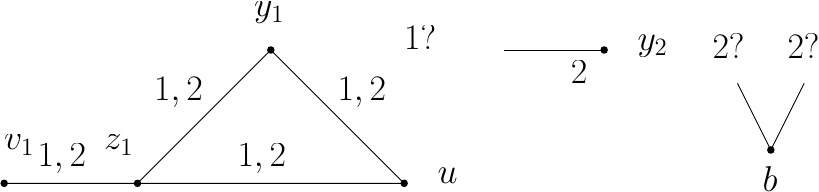}
    \end{minipage}\
  \]
which by BK means that
 \begin{align*}
 & {\colAJ{\sum_{u_1,y_2,z_2}}} \tau(u_1,z_2) T(y_2,z_2,u) {\bf P}_1 [\{ y_2 \in \tC_1 \}\cap\{b\in\tB_1\}  \cap T^+(v_0,z_1,y_1,{\colAJ{u_1}})]\\
 \leq &       \qquad     \begin{minipage}{0.6\textwidth}
      \includegraphics[width=0.6\linewidth]{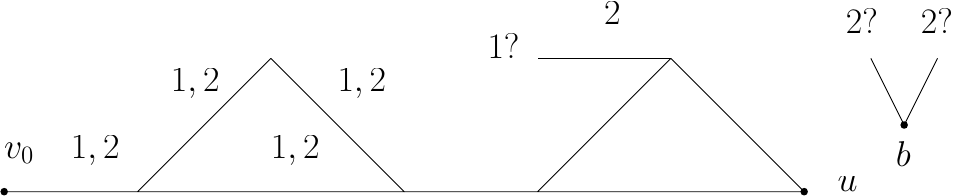}
    \end{minipage}\
   \end{align*}

Using the above in~\eqref{yay} shows that $\sum_{N=2}^{\infty} \hat{\Psi}^{N,0}(F;u,u',v')$ 
is bounded above by diagrams appearing in Lemma~\ref{diag12} and hence adequately bounded.
Minor adaptations to the arguments yield that $\sum_{N=2}^{\infty} \Psi^{N,0}(F;u,x')$ 
satisfies the analogous bound.
  
The case $N'\geq 1$ can again be solved by following the steps of Proposition~\ref{prop:n=1error}
(the case $N=1$, $N'\geq 1$). The only event appearing that we have not analyzed yet is 
 \[
\{ y_2 \in \tC_1 \}\cap\{b\in\tB_1\} \cap\{b'\in\tB_1\} \cap T^+(v_0,z_1,y_1,{\colAJ{u_1}}) 
\subset \bigcup_{{\bf G}_{b,b'} \in \mathcal{G}_{T^+}(b,b')} \mathcal{D}_{{\bf G}_{b,b'}},
\]
 where $\mathcal{G}_{T^+}(b,b')$ is the set of diagrams given by the following generalized diagram
  \[
          \begin{minipage}{0.6\textwidth}
      \includegraphics[width=1\linewidth]{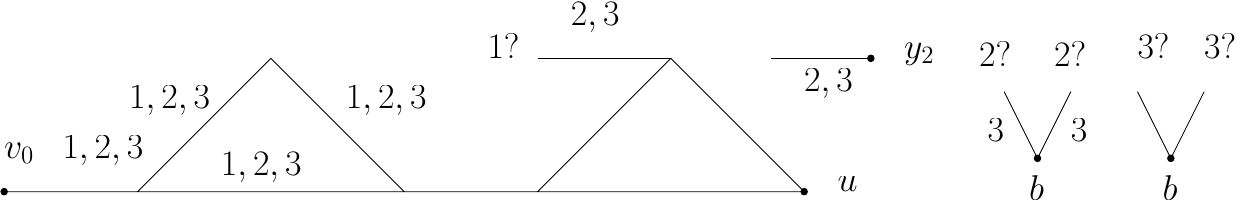}
    \end{minipage}\
  \]
Hence, by BK we have
  \begin{align*}
 & \sum_{u_1,y_2,z_2} \tau(u_1,z_2) T(y_2,z_2,u) {\bf P}_1 [\{ y_2 \in \tC_1 \}\cap\{b\in\tB_1\} \cap\{b'\in\tB_1\} \cap T^+(v_0,z_1,y_1,{\colAJ{u_1}})]\\
 \leq &       \qquad     \begin{minipage}{0.6\textwidth}
      \includegraphics[width=0.8\linewidth]{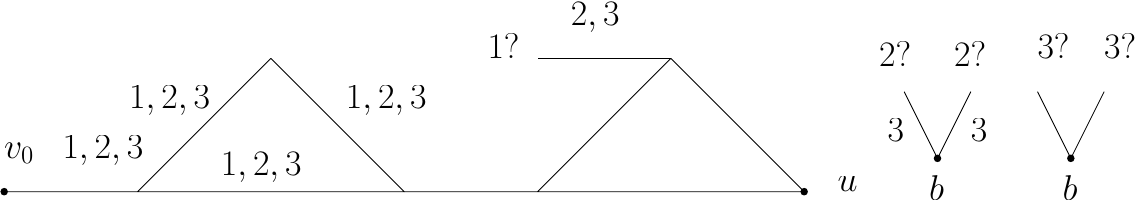}
    \end{minipage}\
   \end{align*}     
This means the diagrams appearing in $\sum_{N'=1}^\infty \sum_{N=2}^{\infty} \Psi^{N,N'}(F;u,u')$  are the  diagrams appearing in Lemma~\ref{diag13} and hence adequately bounded. 
{\colAJ{For $\sum_{N'=1}^\infty \sum_{N=2}^{\infty} \Psi^{N,N'}(F;u,x')$, the analogous bound 
also holds.}}
\end{proof}

\appendix
\section{Auxiliary bounds}
{\colAJ{
\begin{lemma}
\label{lem:D*tau}
Assume the kernel $D\colon \Zd\to[0,1]$ sums to $1$ and that $D(0,y) \le C_D (1+|y|)^{-d}$. 
Assume further that $\tau$ satisfies \eqref{eq:tauAsy}. Then there is a constant $C$ 
such that for all $x \in \Zd$ we have
\eqnspl{e:D*tau}
{ (D \ast \tau) (0,x)
 = \sum_{y \in \Zd} D(0,y) \tau(y,x)
 \le C C_D \tau(x). }
\end{lemma}

\begin{proof}
We split the sum over $y$ according to the distance between $x$ and $y$ on dyadic scales:
\eqnsplst{
 \sum_{y \in \Zd} D(0,y) \tau(y,x)
 &\le C \sum_{y \in \Zd} D(0,y) |x-y|^{2-d} \\
 &\le C \sum_{\substack{y \in \Zd: \\ |x-y| > \frac{1}{2}|x|}} D(0,y) |x-y|^{2-d}
    + \sum_{k \ge 1} \sum_{\substack{y \in \Zd: \\ \frac{|x|}{2^{k+1}} < |x-y| \le \frac{|x|}{2^k}}} 
      D(0,y) |x-y|^{2-d} \\
 &\le \frac{C}{|x|^{d-2}} \left[ \sum_{\substack{y \in \Zd: \\ |x-y| > \frac{1}{2}|x|}} D(0,y) 
    + C_D \sum_{k \ge 1} \sum_{\substack{y \in \Zd: \\ \frac{|x|}{2^{k+1}} < |x-y| \le \frac{|x|}{2^k}}} 
      |y|^{-d} 2^{k(d-2)} \right] \\
 &\le \frac{C}{|x|^{d-2}} \left[ 1
    + C_D \sum_{k \ge 1} |x|^d 2^{-kd} |x|^{-d} 2^{k(d-2)} \right] \\
 &\le \frac{C C_D}{|x|^{d-2}}. }
\end{proof}
}}

\bigskip
\subsubsection*{Acknowledgements} 
MH thanks the \emph{Centre de recherches math\'ematiques} Montreal for hospitality during a research visit in spring 2022 through the Simons-CRM scholar-in-residence program. 
Part of this work was done while MH was in residence at the Mathematical Sciences Research Institute in Berkeley supported by NSF Grant No. DMS-1928930.

\bibliographystyle{plain}
\bibliography{Bib}
\bigskip

%
%
%

\end{document}